\documentclass[a4paper]{amsart}
\usepackage[utf8]{inputenc}
\usepackage[T1]{fontenc}
\usepackage{lmodern, enumerate}
\usepackage{amssymb,amsxtra}
\usepackage{amscd}
\usepackage[mathscr]{euscript}
\usepackage{stmaryrd}
\usepackage[all]{xy}
\usepackage{xcolor}

\usepackage{nicefrac,mathtools}
\usepackage{microtype}
\usepackage{tikz-cd}
\usepackage{tkz-graph}
\usetikzlibrary{arrows}
\usepackage{comment}
\usetikzlibrary{positioning}
\usepackage[pdftitle={Inverse Semigroups},
 pdfauthor={Pere Ara, Alcides Buss, Ado Dalla Costa},
 pdfsubject={Mathematics}
]{hyperref}
\usepackage[lite]{amsrefs}
\renewcommand*{\MR}[1]{ \href{http://www.ams.org/mathscinet-getitem?mr=#1}{MR #1}}

\numberwithin{equation}{section}
\theoremstyle{plain}
\newtheorem{theorem}[equation]{Theorem}
\newtheorem{lemma}[equation]{Lemma}
\newtheorem{proposition}[equation]{Proposition}
\newtheorem{corollary}[equation]{Corollary}
\theoremstyle{definition}
\newtheorem{definition}[equation]{Definition}

\newtheorem{notation}[equation]{Notation}
\theoremstyle{remark}
\newtheorem{remark}[equation]{Remark}
\newtheorem{example}[equation]{Example}

\DeclareMathOperator{\Pol}{\mathrm{Pol}}

\newcommand*{\IS}{\mathcal{S}} %

\newcommand{\act}{\curvearrowright}

\newcommand*{\FIM}{\mathfrak{F_{im}}}

\newcommand*{\Z}{\mathbb Z}

\newcommand*{\F}{\mathbb F} 
\newcommand*{\FC}{\mathbb F_{C}} 

\newcommand*{\LI}{\mathcal{LI}} %

\newcommand*{\Group}{\Gamma} 
\newcommand*{\G}{\mathcal{G}} 

\newcommand*{\ol}{\overline}

\newcommand*{\FG}{\mathfrak{F_{gr}}}

\newcommand*{\FIS}{\mathfrak{F_{is}}}

\newcommand*{\FIC}{\mathfrak{F_{ic}}}

\newcommand{\e}{\mathbf{e}}

\newcommand{\ecb}{\mathbf{ecb}}

\newcommand*{\bcomment}[1]{{\color{blue}#1}}

\newcommand*{\nb}{\nobreakdash}

\renewcommand*{\L}{\mathcal L}
\newcommand*{\Co}{\mathcal C}
\newcommand*{\Le}{\mathcal L}
\newcommand*{\CoV}{\mathcal C V}
\newcommand*{\LeV}{\mathcal L V}

\newcommand*{\isoX}{X^{\mathrm{iso}}}
\newcommand*{\isoisoX}{X^{\mathrm{iso}}_0}

\newcommand*{\isoOmega}{\Omega^{\mathrm{iso}}}
\newcommand*{\isoisoOmega}{\Omega^{\mathrm{iso}}_0}

\newcommand*{\isoOmegaEC}{\Omega(E,C)^{\mathrm{iso}}}
\newcommand*{\isoisoOmegaEC}{\Omega(E,C)^{\mathrm{iso}}_0}

\newcommand*{\isoG}{\G^{\mathrm{iso}}}

\newcommand{\soc}{\mathrm{soc}}

\newcommand{\TUM}{\mathcal T_U \mathcal M}

\newcommand{\Cfin}{C^{\mathrm{fin}}}
\newcommand{\Cinf}{C^{\mathrm{inf}}}

\newcommand*{\E}{\mathcal E}
\newcommand*{\Y}{\mathcal X}%

\newcommand*{\cbY}{\mathcal X_{\mathrm cb}}%

\newcommand*{\YY}{\mathcal Y}%
\newcommand{\SInf}{E^{0,r}_{\infty}}

\newcommand*{\sbe}{\subseteq} 

\newcommand*{\onto}{\twoheadrightarrow}
\newcommand*{\red}{\text{red}}
\renewcommand*{\max}{\mathrm{max}}
\newcommand*{\Bmax}{\text{B-max}}
\newcommand*{\NBmax}{\text{NB-max}}

\newcommand*{\ab}{\mathrm{ab}}

\newcommand*{\tight}{\mathrm{tight}}

\newcommand*{\dual}[1]{\widehat{#1}}

\newcommand{\Etight}{\dual\E_\tight}

\newcommand*{\Ninf}{\mathcal N ^{\mathrm{inf}}}

\newcommand*{\Nfin}{\mathcal N ^{\mathrm{fin}}}

\begin{document}
\title[The Leavitt inverse semigroup of a separated graph]{The Leavitt inverse semigroup of a separated graph}

\author{Pere Ara}
\address{Departament de Matem\`atiques, Edifici Cc, Universitat Aut\`onoma de Barcelona, 08193 Cerdanyola del Vall\`es (Barcelona), Spain, and}
\email{pere.ara@uab.cat}

\author{Alcides Buss}
\address{Departamento de Matem\'atica\\
  Universidade Federal de Santa Catarina\\
  88.040-900 Florian\'opolis-SC\\
  Brazil}
\email{alcides.buss@ufsc.br}

\author{Ado Dalla Costa}
\address{Setor de Matem\'atica do Departamento de Administração Empresarial \\
	Universidade Estadual de Santa Catarina\\
	88.035-001 Florian\'opolis-SC\\
	Brazil}
 \email{adodallacosta@hotmail.com \\ ado.costa@udesc.br}

\dedicatory{Dedicated to our friend and colleague Gene Abrams, on the occasions\\ of his retirement and 70th birthday}

\begin{abstract}
We introduce and study a new inverse semigroup associated to a separated graph $(E,C)$, which we call the \emph{Leavitt inverse semigroup}. This semigroup is obtained as a quotient of the separated graph inverse semigroup $\IS(E,C)$, introduced in our previous paper \cite{ABC25}, and it provides a canonical inverse semigroup model for the tame Leavitt path algebra $\Le_K^\ab(E,C)$ over a commutative unital ring $K$. 

Our first main result describes the Leavitt inverse semigroup $\LI(E,C)$ as a restricted semidirect product of the free group on the edges of $E$ acting partially on a certain semilattice, which is isomorphic to the semilattice of idempotents of $\LI(E,C)$. This description, given in terms of Leavitt--Munn trees, yields a normal form for the elements of $\LI(E,C)$. 

We obtain a normal form for elements of $\Le_K^\ab(E,C)$, leading to explicit linear bases for $\Le_K^\ab(E,C)$.
Building on this and on the structural properties of $\LI (E,C)$, we prove that the natural homomorphism from $\LI(E,C)$ to $\Le_K^\ab(E,C)$ is injective, so that $\LI(E,C)$ embeds as the inverse semigroup generated by the canonical partial isometries in $\Le_K^\ab(E,C)$. Further applications include the determination of natural bases of the kernel $\mathcal Q$ of the natural map from the tame Cohn algebra $\Co_K^\ab (E,C)$ to the tame Leavitt path algebtra $\Le_K^\ab (E,C)$, the computation of the socle, and a characterization of the isolated points of the spectrum. Several examples, such as the Cuntz separated graph and free separations, are discussed to illustrate the theory. 
\end{abstract}



\thanks{The first author was partially supported by the Spanish State Research Agency (grant No.\ PID2023-147110NB-I00 and CEX2020-001084-M), and by the Comissionat per Universitats i Recerca de la Generalitat de Catalunya (grant No.\ 2021-SGR-01015). The second author was supported by CNPq and Fapesc - Brazil, and the third author by CNPq - 402924/2022-3.}

\maketitle

\tableofcontents

\section{Introduction}
\label{sect:introduction}

Leavitt path algebras were introduced independently by Abrams and Aranda Pino \cite{abrams05} and
by Moreno, Pardo and the first author \cite{AMP},
as algebraic analogues of Cuntz--Krieger $C^*$-algebras. Since then, they have become a central object of study in the interplay between algebra, functional analysis, and symbolic dynamics, with deep connections to $K$-theory, classification of algebras, and noncommutative geometry, see \cites{abrams15, AAS, CorHaz24} and the references therein. A particularly fruitful perspective on Leavitt path algebras arises from the theory of inverse semigroups and groupoids, following pioneering ideas of Paterson \cite{Paterson:Groupoids} and Exel \cite{Exel:Inverse_combinatorial}, which were translated to the algebraic setting independently by Steinberg \cite{Steinberg2010} and by Clark, Farthing, Sims and Tomforde \cite{CFST2014}. In this framework, inverse semigroups encode the partial symmetries generated by paths in a graph. 

In recent years, separated graphs and their associated Leavitt path algebras have attracted considerable attention, see for instance \cites{ABPS, ABC23, ABC25, AC24, Ara-Exel:Dynamical_systems, AG12, AraLolk, KocOzaydin2020}. Separated graphs generalize ordinary directed graphs by allowing multiple ``separations'' of the set of edges emitted by a vertex. This setting encompasses both the classical and the free versions of graph algebras, and it reveals a much richer structural behavior. In particular, separated graphs give rise to a hierarchy of algebras — the Cohn path algebra, the Leavitt path algebra, and their tame versions — with the tame versions admitting a natural description via inverse semigroups and groupoids.

The inverse semigroup of a directed graph was introduced by Ash and Hall in \cite{AshHall}. Since then, it has been studied in several papers, often in relation with the structure of Leavitt path algebras, see for instance \cites{JonesLawson, LuoWhangWei23, meakin-milan-wang-2021, meakin-wang-2021, mesyan-mitchell-2016}. In particular, a certain quotient $\LI (E)$ of the graph inverse semigroup $\IS (E)$ was introduced and studied in \cite{meakin-milan-wang-2021}. This inverse semigroup $\LI(E)$, called the Leavitt inverse semigroup of $E$, is characterized by the fact that it coincides with the multiplicative subsemigroup of the Leavitt path algebra $\Le_K(E)$ generated by the natural copy of $E^0 \cup E^1 \cup (E^1)^*$. 

The purpose of this paper is to develop a systematic theory of \emph{normal forms} for tame Leavitt path algebras of separated graphs, through the introduction of a new inverse semigroup, the \emph{Leavitt inverse semigroup} $\LI(E,C)$ of the separated graph $(E,C)$. This semigroup is defined as a quotient of the separated graph inverse semigroup $\IS(E,C)$ introduced in \cite{ABC25}, by imposing certain additional natural relations. It can be described concretely in terms of Leavitt--Munn trees (see Section \ref{sect:Leavitt-inversem} for the detailed definitions). The Leavitt inverse semigroup $\LI (E,C)$ generalizes the homonymous semigroup $\LI (E)$ described before, to which it reduces when taking the trivial separation of $E$. We show in Theorem \ref{thm:semigroup-generated} that $\LI(E,C)$ is naturally isomorphic to the subsemigroup of the tame Leavitt path algebra $\L_K^\ab(E,C)$ generated by $E^0\cup E^1 \cup (E^1)^*$.

Our philosophy is that the Leavitt inverse semigroup $\LI (E,C)$ is the appropriate inverse semigroup to consider when working with tame Leavitt path algebras of arbitrary separated graphs, in the same way that the graph inverse semigroup $ \IS (E,C)$ is the natural object to consider when working with tame Cohn algebras $\Co_K^\ab (E,C)$. Accordingly, the consideration of the Leavitt inverse semigroup not only leads naturally to a new normal form for elements of $\Le_K^\ab(E,C)$, it also leads naturally to a different description of the points of the {\it tight spectrum} $\Etight $ of $\IS (E,C)$. Indeed, we directly show that the tight spectrum of $\LI (E,C)$ is equivariantly homeomorphic to the tight spectrum of $\IS (E,C)$, and this fact leads to a new model for the {\it tight groupoid} $\G_\tight (E,C)$ of the separated graph, that we call the {\it Leavitt model} of $\G_\tight (E,C)$, see Section \ref{sect:tight-spectrum}. Incidentally, this leads also to the question of whether the two semigroups $\IS (E,C)$ and $\LI (E,C)$ are {\it consonant}, a new concept recently introduced by Ruy Exel \cite{exel2025}. A proof outline is given in Remark~\ref{rem:Ruy-consonant-semigroups}. 

Using these tools and a recent result \cite{CCMMR25} on the socle of Steinberg algebras, we are able to completely characterize the socle of tame Leavitt path algebras $\Le_K^\ab (E,C)$. The determination of the socle is an important structural problem that has been addressed for different classes of combinatorial algebras, including Leavitt path algebras of non-separated graphs \cite{AMBMGS2010}, Kumjian-Pask algebras associated to $k$-graphs \cite{BaH2015} and one-sided shift algebras \cite{GR2025}. 

The tame Cohn algebra $\Co_K^\ab(E,C)$ and the tame Leavitt algebra $\Le_K^\ab (E,C)$ are related by the exact sequence 
$$ 0 \longrightarrow \mathcal Q \longrightarrow \Co_K^\ab (E,C)  \longrightarrow \Le_K^\ab (E,C) \longrightarrow 0 ,$$
where $\mathcal Q= \mathcal Q (E,C)$ is the ideal of $\Co_K^\ab(E,C)$ generated by all the elements $q_X= v- \sum_{e\in X} ee^*$, where $X\in C_v$, $v\in E^0$, and $|X|<\infty$. 
For Leavitt path algebras of non-separated graphs, the structure of the ideal $\mathcal Q$ plays a central role in the study of (bivariant) $K$-theory and its applications to classification; see \cite{CorMon21, CorHaz24}. In that case, the ideal $\mathcal Q$ is always contained in the socle of the Cohn path algebra, as follows from \cite[Propositions 1.5.8 and 1.5.11]{AAS}, but we show in the present paper that this is rarely the case for separated graphs, where the structure of the ideal $\mathcal Q$ is much more complicated. However we are able here to compute a linear basis for the ideal $\mathcal Q(E,C)$ for an arbitrary separated graph $(E,C)$, which is a first step in order to deepen our understanding of its structure.  

We conclude the introduction with an outline of the paper. After a section of preliminaries, 
our first main theorem (Theorem~\ref{thm:ECLeavittMunntrees}) describes $\LI(E,C)$ as a semidirect product of the free group $\F (E^1)$ acting partially on a semilattice naturally associated to $(E,C)$, which is isomorphic to the semilattice of idempotents of $\LI(E,C)$. This provides a transparent structural picture of $\LI(E,C)$ and a normal form for its elements. 

Section \ref{sect:normal-form} contains our results on the normal form for the elements of the tame Leavitt path algebra $\Le_K ^\ab (E,C)$.  
In fact, we obtain explicit linear bases for all components of its canonical partial crossed product decomposition (Theorem \ref{thm:full-basis-forLab}). 
We observe in Example \ref{exam:non-separated-case} that, for a  non-separated graph $E$, our basis for $\Le_K(E)$ is closely related to the one obtained by Alhamadi, Alsulami, Jain and Zelmanov in \cite{zel}. As in \cite{zel}, the basis we find for $\Le_K^\ab (E,C)$ depends on a {\it choice function}, which is a function $\mathfrak E \colon \Cfin \to E^1$ such that $\mathfrak E (X)\in X$ for all $X\in \Cfin$, where $\Cfin$ is the collection of sets $X\in C$ such that $|X|<\infty$. Using the basis of $\Le_K ^\ab (E,C)$ and the structure of $\LI(E,C)$, we prove in Theorem \ref{thm:semigroup-generated} that $\LI(E,C)$ embeds faithfully into $\Le_K^\ab(E,C)$, as the inverse semigroup generated by the canonical partial isometries. 

In Section \ref{sect:relation-with-Cohn-algebras}, we show that for any separated graph $(E,C)$, the tame Cohn path algebra $\Co_K^\ab (E,C)$ is closely related to the Leavitt path algebra of another  separated graph $(\ol{E},\ol{C})$, which is explicitly built from $(E,C)$. Indeed, we work more generally within the class of (tame) Cohn-Leavitt path algebras, also known as relative Cohn path algebras, which interpolate between Cohn and Leavitt algebras. This degree of flexibility is very convenient, as shown already in the non-separated situation in e.g. \cite{Ruiz25}. We use this relationship between tame Leavitt and Cohn algebras and our previous results to obtain a linear basis for the ideal $\mathcal Q (E,C)$, which is the kernel of the canonical projection 
$\Co_K^\ab (E,C) \to \Le_K^\ab (E,C)$ (Theorem \ref{thm:BQ-basis-ofQ}). 

Section \ref{sect:tight-spectrum} contains the computation of the spectrum and the tight spectrum of the Leavitt inverse semigroup
$\LI (E,C)$. As mentioned above, we show that the tight spectrums of $\IS (E,C)$ and $\LI (E,C)$ are equivariantly homeomorphic. 
As a result, we offer three different models for the tight groupoid of a separated graph: the {\it complete model}, the {\it standard model}, and the {\it Leavitt model}. For each specific problem, it might be that one of the three models is better suited than the others. The Leavitt model is specially useful to tackle the problem of computing the socle of a tame Leavitt path algebra.
We completely solve this problem in Section \ref{sect:socle}, by combining our techniques with a recent result on the socle of Steinberg algebras \cite{CCMMR25}. 

Finally, Section \ref{sect:examples} discusses examples and applications. We use our techniques to explicitly compute the ideal
$\mathcal Q (E,C) \cap \soc (\Co_K^\ab (E,C)$ for any separated graph $(E,C)$ and any field $K$ (Theorem \ref{thm:completelyblocked}).
We then consider the example of the Cuntz separated graph, endowed with the free separation (Example \ref{exam:Cuntz-free-separation}) and the important example of the semigroup algebra of the free inverse monoid $\FIM (X)$ on a set $X$ (Example \ref{exam:free-inverse-monoid}).
For the latter, we use a representation of the free inverse monoid algebra as a full corner of the tame Cohn algebra of a certain separated graph, obtained in \cite[Example 4.10]{ABC25}. We obtain that the socle of the semigroup algebra of the free inverse monoid on a finite set $X$ is an essential ideal of the algebra, and we compute its structure. However the socle vanishes for the semigroup algebra $K[\FIM (X)]$ when $X$ is an infinite set.   

Our results unify and extend the existing theory of normal forms for Leavitt path algebras, showing that the separated case can be handled in a systematic way through inverse semigroups. 
We expect that the Leavitt inverse semigroup will serve as a fundamental tool for future work on the structure and classification of Leavitt path algebras and of the associated $C^*$-algebras.

\section{Preliminaries}
\label{sect:prelims}

In this section we review fundamental definitions and results from \cite{ABC25}. We refer the reader to \cite{ABC25} for further information.

\subsection{Directed graphs}
\label{subsect:directed-graphs}
	A \emph{directed graph} $E$ is a quadruple $E=(E^0,E^1,s,r)$ consisting of two sets $E^0$ and $E^1$ and two maps $s,r: E^1 \rightarrow E^0$. The elements of $E^0$ and $E^1$ are called vertices and edges and the maps $s,r$ are called the source and range maps, respectively. We do not require our graphs to be finite or countable.
	
	 A \emph{finite path} in $E$ is a sequence of edges of the form $\mu:= e_1 \ldots e_n$ with $r(e_i)=s(e_{i+1})$ for all $i \in \{1, \ldots, n-1\}$. The length of $\mu$ is $|\mu|:=n$ and paths with length $0$ are identified with the vertices of $E$ (we set $s(v)=r(v)=v$). We denote by $E^n$ the set of all finite paths with length $n$ and $\text{Path}(E) := \displaystyle\cup_{n=0}^{\infty} E^n$ denotes the set of all paths of $E$. We can extend the source and range maps to $\text{Path}(E)$ in the obvious way: if $\mu=e_1\ldots e_n \in \text{Path}(E)$, then $s(\mu) = s(e_1)$ and $r(\mu) = r(e_n)$. Given two paths $\mu,\nu \in \text{Path}(E)$ with $r(\mu) = s(\nu)$, one obtains a new path $\mu\nu$ by concatenation with $|\mu\nu| = |\mu|+|\nu|$. 
 
	Given a graph $E$, we define its \emph{extended graph} (also called the \emph{double graph} of $E$)
as the new graph $\hat{E} = (E^0,E^1 \cup  E^{-1}, r, s)$, containing $E$ as a subgraph, and where we set
$s(e^{-1}) := r(e)$, $r(e^{-1}) := s(e)$ for all $e \in  E^1$. This construction is standard in inverse semigroup methods, as it allows the incorporation of formal inverses of edges.

\subsection{Separated graphs}
\label{subsect:separat-graphs}
Separated graphs were introduced to interpolate between classical and free constructions, and they naturally give rise to richer inverse semigroups.
A \emph{separated graph} is a pair $(E,C)$ consisting of a graph $E=(E^0,E^1,s,r)$ and a \emph{separation} $C=\bigsqcup_{v\in E^0}C_v$ on $E$, consisting of partitions $C_v$ of $s^{-1}(v)\sbe E^1$ into pairwise disjoint nonempty subsets (with $C_v= \emptyset$ if $v$ is a sink, i.e., if $s^{-1}(v)=\emptyset$). The \emph{trivial separation} is the separation with $C_v=\{ s^{-1}(v)\}$ for all non-sinks $v\in E^0$; a graph with the trivial separation is also called trivially separated or a non-separated graph. The \emph{free separation} of $E$ is the finest separation, where each $s^{-1}(v)$ is separated into singletons, that is, $C_v=\{\{e\}:e\in s^{-1}(v)\}$.

\subsection{Inverse semigroups}
\label{subsect:inverse-semigroups}
An inverse semigroup is a semigroup $S$ endowed with an involution $s\mapsto s^{-1}$ satisfying 
$$(1)\,\, ss^{-1}s=s\quad \mbox{and}\quad  (2)\,\, (ss^{-1})(tt^{-1})=(tt^{-1})(ss^{-1})\quad\mbox{for all }s,t\in S.$$ 
In this case, the pseudo-inverse $s^{-1}$ of $s\in S$ is uniquely determined by the relations $ss^{-1}s=s$ and $s^{-1}ss^{-1}=s^{-1}$. We write $\E(S)$ for the semilattice of idempotents of $S$. 
Inverse semigroups are naturally endowed with a partial order:
$$s\leq t\Leftrightarrow ts^{-1}s=s\Leftrightarrow ss^{-1}t=s.$$
Given a semigroup $S$, we may always add a (formal) zero element $0$, getting a semigroup with zero $S_0=S\sqcup \{0\}$, as well as a unit $1$, obtaining a semigroup with unit (i.e. a monoid) $S_1=S\sqcup\{1\}$. 
Our semigroups of interest will usually have already a zero, sometimes a unit. If $S$ already has a unit, we shall generally write $S^*=S\backslash\{1\}$, and if $S$ has a zero, $S^\times =S\backslash\{0\}$. Note that these sets are typically not subsemigroups.

We refer the reader to \cite{lawson} for more information about inverse semigroups.

Let $S$ be an inverse semigroup with zero and $\Group$ a group. A {\it partial homomorphism} from $S$ to $\Group$ is a map $\pi \colon S^\times \to \Group$ such that $\pi (st) =\pi (s)\pi (t)$ whenever $st\ne 0$. This notion appears frequently when describing inverse semigroups as (restricted) semidirect products.
 A partial homomorphism $\pi$ is {\it idempotent pure} if $\pi^{-1}(1)= \E(S)^\times$. 
If such an idempotent pure partial homomorphism exists, we say that $S$ is \emph{strongly $E^*$-unitary}. 
By \cite[Proposition 2.1]{ABC25}, an inverse semigroup with $0$ is strongly $E^*$-unitary if and only if it is a restricted semidirect product $S= \mathcal E (S) \rtimes_\theta ^r \Gamma$, for a partial action $(D_g,\theta_g)$ of $\Gamma $ on $\E (S)$. 

\subsection{Fundamental groupoid of a graph, $C$-separated paths and $C$-compatible trees}
\label{subsect:fundamental-groupoid}
Recall that the free group $\F (X)$ on a set $X$ can be defined as the set of reduced words in the free monoid on $X\cup X^{-1}$, with the product given by
$\alpha \cdot \beta = \text{red} (\alpha\beta)$ where $\alpha\beta$ is the concatenation of $\alpha$ and $\beta $ and $\text{red} (-)$ indicates reduction of words . Given a directed graph $E$, a similar process leads to the {\it fundamental groupoid} $\FG (E)$ of $E$, which plays an important role in this work. First consider the path $*$-semigroup $\mathcal P (E)$, whose nonzero elements are exactly the paths on $\hat{E}$, and the product  of two paths $\lambda $ and $\mu$ is given by its concatenation $\lambda \mu$ if $r(\lambda)= s(\mu)$, or $0$ if $r(\lambda) \ne s(\mu)$.
The fundamental groupoid $\FG(E)$ is the set of {\it reduced paths} in $\mathcal P (E)$, which are the paths that do not contain subpaths of the form $xx^{-1}$ for $x\in \hat{E}^1$.
The product in $\FG(E)$ is defined as
$$\alpha \cdot \beta =  \text{red} (\alpha \beta) $$
whenever $r(\alpha)= s(\beta)$, where again $\text{red}(-)$ indicates the reduction process. This is a groupoid with set of units $E^0$. Note that $g\cdot g^{-1} = s(g)$ and $g^{-1}\cdot g = r(g)$ for all $g\in \FG(E)$. It is important to take into account that we will always use the notation $\alpha \cdot \beta$ to indicate the product in $\FG(E)$, while the notation $\alpha\beta$ is used for the concatenation product. Whenever we have an expression of the form $\alpha_1\alpha_2 \cdots \alpha_n$ for an element of $\FG(E)$, we will generally understand that this expression is in reduced form as it stands. 

General nonzero elements of $\mathcal P(E)$ will be called \emph{paths}, and elements of $\FG (E)$ will be called \emph{reduced paths}. 

We define the \emph{prefix order} $\le_p$ on $\mathcal P (E)$ by $\mu\leq _p\nu$ if there is $\gamma\in \mathcal P (E)$ such that $\nu=\mu\gamma$. A {\it lower subset} of $\mathcal P (E)$
is a subset $L$ of $\mathcal P (E)$ such that, for $\lambda \in L$, we have $\nu\in L$ whenever $\nu\le_p \lambda$.   

We denote the Cayley graph of $\FG(E)$ by $\Gamma_E$. We have $\Gamma_E^0= \FG (E)$ and there is an edge $(g,e)$ from $g$ to $g\cdot e$ whenever $g,g\cdot e\in \FG (E)$, for $e\in E^1$. 
We have $\Gamma_E = \bigsqcup_{v\in E^0} K_v$, where $K_v$ are the connected components of $\Gamma_E$, consisting of all the paths $\gamma$ such that $s(\gamma) =v$. 

We now introduce the notion of a $C$-separated path. 

\begin{definition} \cite[Definition 3.6]{ABC25}
	\label{def:Cseparatedword-and-compatible}
    Let $(E,C)$ be a separated graph. 
	A {\it $C$-separated path} is a reduced path $w=y_1\cdots y_n\in \FG (E)$, with $y_i\in \hat{E}^1$, that does not contain any subpath of the form $x^{-1}y$ with $x,y\in X$ for $X\in C$. For $v\in E^0$, we denote by $\FC (v)$ the set of all $C$-separated paths $\gamma$ such that $s(\gamma ) =v$. The set $\FC := \bigsqcup_{v\in E^0} \FC (v)$ is the set of all $C$-separated paths.  
\end{definition}

We now come to the crucial concept of $C$-compatibility of $C$-separated paths.

\begin{definition}\cite[Definition 3.7]{ABC25}
	\label{def:Ccompatible}
	Let $v\in E^0$ and $\gamma,\nu\in \FC (v)$. We say that $\gamma $ and $\nu$ are {\it $C$-compatible} in case $\nu^{-1}\cdot \gamma = \red (\nu^{-1}\gamma) \in \FC$. Otherwise $\gamma $ and $\nu$ are said to be {\it $C$-incompatible}. A subset $T$ of $\FC(v)$ is said to be $C$-compatible if each pair of elements  of $T$ is $C$-compatible. 
	\end{definition}

In terms of the Cayley graph $\Gamma_E$ of $E$, a non-empty lower subset of $\FC(v)$, for $v\in E^0$,  is a subtree $T$ of the connected component 
$K_v$ of $\Gamma _E$, containing the base point $v$ of $K_v$. 
Such a non-empty lower subset $T$ is $C$-compatible if and only if the geodesic path joining two arbitrary vertices $\gamma,\nu$ of $T$ is a $C$-separated path, see \cite[Remark 3.8]{ABC25}. The set of all non-empty finite lower $C$-compatible subsets of $\FC (v)$ is denoted by $\Y (v)$. As remarked above, $\Y(v)$ can be seen as a certain family of finite subtrees of $K_v\subseteq \Gamma_E$, and these play a central role in the description of $\LI (E,C)$ in Section \ref{sect:Leavitt-inversem}.

	For a non-empty subset $A$ of $\FG (E)$, we write $A^{\downarrow}=\{x\in \FG(E) : x\leq_p a,\, \mbox{for some }a\in A\}$ for the lower subset  generated by $A$ with respect to the prefix order. If $A=\{a\}$ is a singleton, we shall also write $a^\downarrow:=\{a\}^\downarrow$ to simplify the notation. Note that $A^\downarrow \subseteq \FC$ whenever $A\subseteq \FC$.

\subsection{Inverse semigroup, tame Cohn path algebra, and tame Leavitt path algebra of a separated graph}
\label{subsect:Cohn-Leavitt-algs}
We start with the definition of the inverse semigroup of a separated graph, which generalizes the usual construction of the graph inverse semigroup. 

\begin{definition}\cite{ABC25}
	\label{def:graphsemigroup}
    Let $(E,C)$ be a separated graph. The \emph{inverse semigroup} of $(E,C)$ is the universal inverse semigroup $\IS (E,C)$ 
    generated by $E^0\cup \hat{E}^1=E^0\cup E^1\cup E^{-1}$ subject to the following relations:
	\begin{enumerate}
		\item $vw = \delta_{v,w}v$ for all $v,w\in E^0$;
		\item $s(x)x = x$ for all $x\in \hat{E}^1$;
		\item $xr(x) = x$ for all $x\in \hat{E}^1$;
		\item $e^{-1}f = \delta_{e,f} r(e)$ for all $e,f\in X$ with $X\in C$.  
	\end{enumerate}
    \end{definition}

\begin{definition}\cite{AG12}
\label{def:Cohn-and-Leavitt-algs}
Let $K$ be a commutative ring with an involution $*$ and let $(E,C)$ be a separated graph. The Cohn path algebra $\Co_K(E,C)$ is the universal $*$-algebra over $K$ generated by 
self-adjoint elements $v\in E^0$ and elements $e\in E^1$ satisfying the following relations:
\begin{enumerate}
    \item[(V)] $vw=\delta_{v,w}v$ for all $v,w\in E^0$;
    \item[(E1)] $s(e)e= e= er(e)$ for all $e\in E^1$;
    \item[(E2)] $e^ *s(e)=e^* = r(e)e^*$ for all $e\in E^1$;
    \item[(SCK1)] $e^*f=\delta_{e,f}r(e)$ for all $e,f\in X$ with $X\in C$.
\end{enumerate}
The Leavitt path algebra $\Le_K(E,C)$ is the universal $*$-algebra over $K$ generated by the same generators and relations above together with the Cuntz–Krieger relation:
\begin{enumerate}
    \item[(SCK2)] $\sum_{e\in X}ee^*=v$ for every finite subset $X\in C_v$ with $v\in E^0$.
\end{enumerate}
\end{definition}

Given a $*$-algebra $A$ with a subset $D\sbe A$ consisting of partial isometries (i.e. $ss^*s=s$ for all $s\in D$),  there is a canonical way to downsize the complexity of the algebra $A$, as follows. 
Let $S$ be the $*$-subsemigroup of $A$ generated by $D$ and let $J=J(D)$ be the two-sided ideal of $A$ generated by all the commutators $[e(x),e(y)]$ with $x,y\in S$, where $e(x):=xx^*$.
Then $A_{\ab}:= A/J$ is the {\it tame $*$-algebra associated to} $D$. It has the property that the $*$-subsemigroup of $A_{\ab}$ generated by the canonical image of $D$ in $A_{\ab}$ is an inverse semigroup. 

\begin{definition}
    Let $K$ be a commutative ring with an involution $*$ and let $(E,C)$ be a separated graph. The {\it tame Cohn path algebra} of $(E,C)$ is the tame $*$-algebra $\Co_K^\ab (E,C)$ associated to the family $E^1$ of partial isometries of $\Co_K(E,C)$. Similarly, the {\it tame Leavitt path algebra} of $(E,C)$ is the tame $*$-algebra $\Le_K^\ab (E,C)$ associated to the same family $E^1$ of partial isometries of $\Le_K(E,C)$.  
\end{definition}

By \cite[Lemma 6.15]{ABC25}, we have $\Le_K^\ab (E,C) \cong \Co_K^\ab (E,C)/\mathcal Q$, where $\mathcal Q$ is the two-sided ideal of $\Co_K^\ab (E,C)$ generated by the projections $q_X:= v- \sum_{e\in X} ee^*$, for all finite sets $X\in C_v$, $v\in E^0$. This shows that $\Le_K^\ab(E,C)$ may be seen as the quotient of the tame Cohn algebra where all Cuntz–Krieger relations of finite type are enforced.

We summarize in the next two theorems the structure of the $*$-algebras $\Co_K^\ab (E,C)$ and $\Le_K^\ab (E,C)$. For an inverse semigroup $S$, we will denote by $\G (S)$ the universal groupoid of $S$, and by $\G_\tight  (S)$ the tight groupoid of $S$. For a separated graph, set 
$$ \G (E,C):= \G (\IS (E,C)),\qquad \G_\tight (E,C) := \G_\tight (\IS (E,C)).$$
The Steinberg algebra of an ample groupoid $\G$ will be denoted by $A_K(\G)$.

\begin{theorem}\cite[Proposition 6.5 and Corollary 6.10]{ABC25}
    \label{thm:structure-of-tame-Cohn}
Let $(E,C)$ be a separated graph. Then we have $*$-isomorphisms
$$\Co_K^\ab (E,C) \cong K[\IS (E,C)] \cong A_K(\G (E,C)) \cong C_K(\widehat{\E})\rtimes \F,$$
where $\widehat{\E}$ is the spectrum of $\E=\E (\IS(E,C))$ and $\F = \F(E^1)$ is the free group on $E^1$. 
\end{theorem}

\begin{theorem}\cite[Theorem 6.20]{ABC25}
    \label{thm:structure-of-Leavitt}
Let $(E,C)$ be a separated graph. Then we have $*$-isomorphisms
$$\Le_K^\ab (E,C) \cong A_K ( \G_\tight (E,C)) \cong C_K (\widehat{\E}_\tight ) \rtimes \F,$$
where $\widehat{\E}_\tight $ is the tight spectrum of $\E=\E (\IS(E,C))$ and $\F = \F(E^1)$ is the free group on $E^1$. 
    \end{theorem}

The isomorphisms $ A_K(\G (E,C)) \cong C_K(\widehat{\E})\rtimes \F$ and $A_K ( \G_\tight (E,C)) \cong C_K (\widehat{\E}_\tight ) \rtimes \F$ in the above theorems come from a canonical partial action $\F \act \widehat{\E}$ and its restriction
$\F \act \widehat{\E}_\tight $ to the tight spectrum $ \widehat{\E}_\tight $, respectively, where $\E = \E (\IS(E,C))$. See \cite{ABC25} for details. 

\section{The Leavitt inverse semigroup of a separated graph}
\label{sect:Leavitt-inversem}

In this section we introduce a new inverse semigroup associated to a separated graph $(E,C)$, namely the {\it Leavitt inverse semigroup} of $(E,C)$. This generalizes the Leavitt inverse semigroup of a directed graph, introduced in \cite{meakin-milan-wang-2021}.

\begin{definition}
\label{def:Leavitt-inverse-semigroup}
Let $(E,C)$ be a separated graph. The {\it Leavitt inverse semigroup} of $(E,C)$ is the universal inverse semigroup $\LI (E,C)$  
generated by $E^0\cup \hat{E}^1=E^0\cup E^1\cup E^{-1}$ subject to the following relations:
\begin{enumerate}
	\item $vw = \delta_{v,w}v$ for all $v,w\in E^0$;
	\item $s(x)x = x$ for all $x\in \hat{E}^1$;
	\item $xr(x) = x$ for all $x\in \hat{E}^1$;
	\item $e^{-1}f = \delta_{e,f} r(e)$ for all $e,f\in X$ with $X\in C$.
	\item $ee^{-1} = s(e)$ for all $e\in E^1$ such that $\{e\} \in C$.   
\end{enumerate}
In other words, $\LI (E,C)$ is the quotient of the inverse semigroup $\IS (E,C)$ of $(E,C)$ (as in Definition~\ref{def:graphsemigroup}) by the congruence generated by the relations (5).
\end{definition}

We show later that $\LI (E,C)$ is isomorphic to the inverse semigroup generated by $E^0\cup E^1$ in the abelianized Leavitt path algebra $\Le_K^\ab (E,C)$, for any commutative unital ring $K$ (see Theorem \ref{thm:semigroup-generated}). This was shown in \cite{meakin-milan-wang-2021} in the non-separated case. 

We start by establishing a normal form for the elements of $\LI (E,C)$. For this, we need to recall basic notions and terminology from \cite {ABC25}. 

We will use the notation of \cite{ABC25}. Recall from Definition \ref{def:Cseparatedword-and-compatible} that for $v\in E^0$, $\FC(v)$ is the set of all the $C$-separated paths $w$ such that $s(w)= v$.  The set  
$\Y (v)$ is the set of all non-empty  finite lower $C$-compatible subsets of $\FC (v)$, and $\Y = \bigsqcup_{v\in E^0} \Y (v)\bigsqcup \{0\}$. As in \cite[Definition 3.13]{ABC25}, we set $\YY= \Y/{\sim}$, where $\sim$ is the congruence on $\Y$ generated by $g^\downarrow \sim (gx^{-1})^\downarrow$, where $g\in \FC$, $x\in E^1$, and $gx^{-1}$ is reduced as it stands.

For $v\in E^0$, the set $\Y _0(v)$ is the set of those $I\in \Y (v)$ such
that each element in $\max (I)$ does not end in $E^{-1}$, and we set $\Y _0 = \bigsqcup_{v\in E^0} \Y _0(v)$ (\cite[Notation 3.17]{ABC25}). By \cite[Proposition 3.16]{ABC25}, each element $I\in \Y$ has a unique representative $I_0\in \Y_0$. 

We denote by $\FG (E)$ the fundamental groupoid of $E$ (see Subsection \ref{subsect:fundamental-groupoid}).

\begin{definition}
	\label{def:decomposition-gLw}
	Let $(E,C)$ be a separated graph. For each $g\in \FG (E)$, write 
	$$g= g_L w ,$$
	where $g_L$ is either $s(g)$ or a prefix of $g$ ending in an edge $e\in E^1$ such that $e\in X\in C$, with $|X| >1$, and $w= x_1\cdots x_r $, $r\ge 0$, where for each $i$, either $x_i\in E^{-1}$ or $\{x_i\}\in C$.    
	\end{definition}

Observe that due to relations (4) and (5) in Definition \ref{def:Leavitt-inverse-semigroup}, we have $g_Lg_L^* = gg^*$ in $\LI (E,C)$ for each $g\in \FG (E)$.  

Recall that an {\it inverse category} is a category $C$ such that for each arrow $\gamma$ there is a unique arrow $\gamma^{-1}$ such that $\gamma = \gamma \gamma^{-1}\gamma$ and $\gamma^{-1} = \gamma^{-1} \gamma \gamma^{-1}$.

For each graph $E$, there is a free inverse category $\FIC(E)$. It is the quotient of the free category over $\hat{E}$ by the congruence generated by the relations $\gamma \sim \gamma \gamma ^{-1}\gamma$ and $(\gamma \gamma^{-1}) (\nu\nu^{-1}) \sim (\nu \nu^{-1})(\gamma \gamma^{-1})$ whenever $s(\gamma) = s(\nu)$, see \cite[Section 4]{marg-meakin-93}.

Given an inverse category $C$ we can form an inverse semigroup $S(C)$ by adjoining a zero to $C$ and declaring any undefined product to be $0$.  We denote by $\FIS(E)$ the inverse semigroup $S(\FIC(E))$ associated to the free inverse category $\FIC(E)$. We call $\FIS(E)$ the {\it free inverse semigroup} of $E$.

Let $\Gamma_E$ be the Cayley graph of the groupoid $\FG (E)$. Hence we have
$\Gamma_E ^0 =  \FG (E)$, and there is an edge $(g,e)$ from $g$ to $g\cdot e$ whenever $g, g\cdot e\in \FG (E)$, $e\in E^1$.

The elements of $\FIC(E)$ can be interpreted as certain Munn trees in $\Gamma_E$, see \cite[Section 4]{marg-meakin-93} and \cite[Section 4]{ABC25}. This will play a central role in what follows, hence we recall here this description.

 \begin{definition}\cite[Definition 4.1(a)]{ABC25}
 	\label{def:Mun-EC-tree} Let $E$ be a directed graph.
	A {\it Munn $E$-tree} is a pair $(T,g)$, where $T$ is a finite connected subgraph (hence a subtree) of $\Gamma_E$ containing the unit $v$ of the connected component $K_v$ corresponding to $T$, and $g$ is a vertex of $T$.
\end{definition}

We now recall two basic results from \cite{ABC25}.

\begin{proposition}\cite[Proposition 4.2]{ABC25}
	\label{prop:MunnE-trees}
	Let $E$ be a directed graph. Then the free inverse semigroup $\FIS(E)$ associated to $E$ is isomorphic to the semigroup of all Munn $E$-trees, together with $0$, endowed with the product 
	$$(T_1,g_1)\cdot (T_2,g_2) = \begin{cases} (T_1\cup g_1\cdot T_2,\,  g_1\cdot g_2) & \text{ if } r(g_1) = s(g_2)\\ \qquad \,\,  \qquad \,\, 0 & \text{ if } r(g_1) \ne s(g_2) 
	\end{cases}  . $$
	The inverse of a Munn $E$-tree $(T,g)$ is given by
	$$(T,g)^{-1} = (g^{-1}\cdot T, g^{-1}).$$ 
\end{proposition}

\begin{corollary}\cite[Corollary 4.3]{ABC25}
	\label{cor:FSIE-strongly-unitary}
	Let $E$ be a directed graph. Then $\FIS(E)$ is strongly $E^*$-unitary and thus it is a restricted semidirect product 
	$$\FIS(E) =\mathcal E (\FIS(E)) \rtimes_\theta ^r \F,$$
	where $\F = \F (E^1)$ is the free group on $E^1$.  
\end{corollary}

For an arbitrary directed graph $E$, we may consider a subset $U$ of edges in $E$, and impose relations 
(1)-(3) from Definition \ref{def:Leavitt-inverse-semigroup}, relations $e^{-1}e= r(e)$ for all edges $e$ in $E$,  and relations
\begin{enumerate}
\item[(5)$_U$]  $ee^{-1} = s(e)$   for all  $e\in U$,
\end{enumerate} 
obtaining a certain inverse semigroup, which will be denoted here by $\mathcal T _U (E)$. As we will see in Proposition \ref{prop:Toeplitz-Munn-EUtrees}, all these inverse semigroups follow the same pattern.

\begin{definition}
	\label{def:canonicaldec-respect-toU}
	Let $E$ be a directed graph and let $U$ be a subset of $E^1$. For each $g\in \FG (E)$, we can uniquely write $g=g_U w$, where $g_U$ is either a vertex or a prefix of $g$ such that $g_U$ does not end in $E^{-1}$ and does not end in an edge belonging to $U$, and $w$ is a product of elements of the form $f^{-1}$ for $f\in E^1$ and elements of the form $e$ for $e\in U$. 
	
	We say that $g=g_U w$ is the {\it canonical decomposition} of $g$ with respect to $U$. For $U=\emptyset$, we obtain the  decomposition $g= g_0w$ from \cite[Notation 3.15]{ABC25}. When $(E,C)$ is a separated graph and $U$ is the set of elements $e\in E^1$ such that $\{ e\}\in C$, this agrees with the decomposition $g=g_L w$ from Definition \ref{def:decomposition-gLw}.
\end{definition}

\begin{definition}
	\label{def:ToeplitzMun-EU-tree} Let $E$ be a directed graph and let $U$ be a subset of $E^1$.
 We define a {\it Toeplitz-Munn $(E,U)$-tree} as a pair $(T,g)$ where $T$ is a finite 
		connected subgraph of $\Gamma_E$ containing the root $v$ corresponding to its connected component, and such that all maximal elements of $T$ do not end in $E^{-1}$, and do
		not end in an edge $e\in U$, and  $g= g_U w$ satisfies that $g_U\in T$.    
	\end{definition}

From now on, we will identify the free inverse semigroup $\FIS (E)$ on $E$ with the inverse semigroup $\mathcal M $ of Munn $E$-trees, together with $0$ (see Proposition \ref{prop:MunnE-trees}).
We will simply refer to the first component $T$ of a Munn $E$-tree $(T,g)$ as an $E$-tree. The set of $E$-trees is identified with the set $\E (\FIS (E))^\times $ of nonzero idempotents of $\FIS (E)$. 

Let $U$ be a subset of $E^1$. Let $\sim_U$ be the congruence on $\FIS (E)\cong \mathcal M$ generated by the relations $e^{-1}e \sim_U r(e)$ for $e\in E^1$ and $ee^{-1}\sim_U s(e)$ for $e\in U$. In terms of the representation as Munn $E$-trees, $\sim_U$ is the congruence generated by the relations $(ge^{-1})^\downarrow \sim_U g^\downarrow $ for all $e\in E^1$ such that $ge^{-1} \in \FG (E)$, and $(ge)^\downarrow \sim_U g^\downarrow$ for all $e\in U$ such that $ge \in \FG (E)$.

Obviously, we have $$\mathcal T_U(E) \cong \FIS (E)/{\sim_U} \cong \mathcal M /{\sim_U}.$$

\begin{lemma}
	\label{lem:uniqueTU}
	Let $U$ be a subset of $E^1$, and set $\mathcal M_U= \mathcal M/{\sim_U}$.
	For each $E$-tree $T$, the $\sim_U$-equivalence class $[T]_U$ of $T$ in $\mathcal M$ contains a maximum element $T_U$ with respect to the natural order in the inverse semigroup $\mathcal M$. This maximum element is the unique element $T_U$ of $[T]_U$ such that each maximal element of $T_U$ does not end in $E^{-1}$ and does not end in $e$, for $e\in U$.  
\end{lemma}

\begin{proof}
	Let $T$ be an $E$-tree. For each $g\in \FG (E)$, we consider its canonical decomposition $g=g_U w$ with respect to $U$ (see Definition \ref{def:canonicaldec-respect-toU}).
	
	Now set $A:= \{ g_U: g\in T\}$, and $T_U = \max (A)^\downarrow $. Then the set of maximal elements of $T_U$ is exactly $\max (A)$, and thus all maximal elements of $T_U$ do not end in $E^{-1}$ and do not end in $U$. Since $T_U\subseteq T$, it follows that $T\le T_U$.
	
	Also, as in the proof of \cite[Lemma 3.16]{ABC25}, we have $T_U\sim_U T$.  
	
	We now want to show that, if $S$ is an $E$-tree and $T\sim_U S$ in $\mathcal M$, then $S_U=T_U$. Let $(T_0,g_0), (T_1,g_1), \dots , (T_n,g_n)$ be a sequence of Munn $E$-trees, with $(T_0,g_0)= (T,1)$ and $(T_n,g_n)= (S,1)$, and such that for each $0\le i < n$, we can pass from $(T_i,g_i)$ to $(T_{i+1},g_{i+1})$, or from $(T_{i+1}, g_{i+1})$ to $(T_i,g_i)$, by a single step either of the form 
	$$(X,h) (\{r(e)\}, r(e))(Y,k) \sim_U
	(X,h) ((e^{-1})^\downarrow , r(e)) (Y,k)$$
	 for $e\in E^1$, or of the form
	$$(X,h) (\{s(e)\}, s(e))(Y,k) \sim_U
(X,h) (e^\downarrow , s(e)) (Y,k)$$	  
	for $e\in U$. 
	
	In both cases, we have $(Z, h\cdot k) \sim_U (Z\cup h\cdot x^\downarrow, h\cdot k)$, where $x\in E^{-1}\cup U$, and  $Z= X\cup h\cdot Y $ is an $E$-tree such that $h\in Z$. It thus suffices to show that in such a situation, one has $Z_U = (Z\cup  h\cdot x^\downarrow)_U$. If $hx$ is not reduced then 
	$h\cdot x\in Z$ and we indeed have $Z=Z\cup  h\cdot x^\downarrow $. If $hx $ is reduced, then $h_U= (hx)_U$ and hence $\{ z_U: z\in Z\} = \{z_U: z\in Z\cup  h\cdot x^\downarrow \}$, showing that 
	$Z_U=  (Z\cup h\cdot x^\downarrow)_U$, as desired. Applying inductively this fact to all the consecutive pairs of elements in the sequence $\{ (T_i,g_i)\}$, we obtain that $T_U=S_U$. Hence $S\le T_U$ for all $S\in [T]_U$,
	and $T_U$ is the largest element of $[T]_U$ in $\mathcal M _U$. Also it is clear from this fact that 
	$T_U$ is the unique element in $[T]_U$ with the property that each maximal element does not end in $E^{-1}\cup U$. 
	\end{proof}

The following lemma is a straightforward consequence of Lemma \ref{lem:uniqueTU}. We leave its proof to the reader. 

\begin{lemma}
	\label{lem:cosequenceofuniqueTU}
With the above notation, we have, for Munn $E$-trees $(T,g)$ and $(S,h)$,
\begin{enumerate}
	\item $(T,g)\sim_U (S,h)$ if and only if $S_U= T_U$ and $g=h$.
	\item The $\sim_U$-equivalence class $[T,g]_U$ has a largest element, which is $(T_U\cup g^\downarrow , g)$.
\end{enumerate}
\end{lemma}

\begin{proposition}
	\label{prop:Toeplitz-Munn-EUtrees}
	Let $E$ be a directed graph, let $U$ be a subset of $E^1$ and set $\mathcal M_U= \mathcal M/{\sim_U}$. 
	 Then the inverse semigroup $\mathcal M _U$ is isomorphic to the inverse semigroup $\TUM$ consisting of all the Toeplitz-Munn $(E,U)$-trees $(T,g)$, endowed with the product 
	$$(T_1,g_1)\cdot (T_2,g_2) = \begin{cases} \Big( (T_1'\cup g_1\cdot T_2')_U, \,\, g_1\cdot g_2\Big) & \text{ if } r(g_1) = s(g_2)\\ \qquad \,\,  \qquad \,\, 0 & \text{ if } r(g_1) \ne s(g_2) 
	\end{cases}   $$
	where $T_1',T_2'\in \E (\FIS (E))$ are representatives of $T_1$ and $T_2$ respectively, such that $g_1\in T_1'$ and $g_2\in T_2'$. The inverse is given by  	
	$(T,g)^{-1} = ((g^{-1}\cdot T')_U, g^{-1})$,
	where $T'$ is any representative of $T$ such that $g\in  T'$. 
	
	Moreover, the map $(T,g)\mapsto \omega (g)$ determines an idempotent pure partial homomorphism $\TUM \to \F$. 
\end{proposition}

\begin{proof}
	We define a map $\Phi \colon \mathcal M_U \to \TUM$ by $\Phi ([T,g]_U) = (T_U, g)$. Note that this map is well-defined since for $(T,g)\in \mathcal M$, we have $g\in T$ by definition, and thus $g_U\in T_U$, and moreover $T_U=S_U$ and $g=h$ whenever $(T,g)\sim_U (S,h)$ by Lemma \ref{lem:cosequenceofuniqueTU}(1). Moreover the map $\Phi$ is injective again by Lemma \ref{lem:cosequenceofuniqueTU}(1). If $(T,g)$ is a Toeplitz-Munn $(E,U)$-tree, then $(T\cup g^\downarrow, g)$ is a Munn $E$-tree, and $\Phi ([T\cup g^\downarrow , g]) = (T,g)$, which shows that $\Phi$ is surjective.
	
	Since $\Phi$ is a bijection, we can translate the product of $\mathcal M _U$ to a product in $\TUM$, so that $\Phi$ becomes a semigroup isomorphism. This leads to the product described in the statement.

	Let $\omega \colon \FG (E)^\times \to \F$ be the obvious 
	 idempotent pure partial homomorphism. Then the map
	 $\pi\colon \FIS (E)^\times \to \F$, $\pi (T,g) = \omega (g)$ is the canonical idempotent pure partial homomorphism of $\FIS (E)\cong \mathcal M$, see \cite[Corollary 4.3]{ABC25}. It is clear that $\pi$ 
	 factors through an idempotent pure partial homomorphism $\FIS (E)^\times /{\sim_U} \to \F$.   
\end{proof}

\begin{remark}
	Note that if $(T,g)$ is a Toeplitz-Munn $(E,U)$-tree (see Definition \ref{def:ToeplitzMun-EU-tree}), it might be that $(T,g)$ is not a Munn $E$-tree, because it might be that $g\notin T$, although by definition we always have that $g_U\in T$. This means that doing calculations in the inverse semigroup $\TUM \cong \mathcal T _U(E)$, we must be careful and take representatives $T'$ of $T$ in $\mathcal E (\FIS(E))$ such that $g\in T'$, so that $(T',g)$ is a Munn $E$-tree. We can always do the choice $T'=T\cup g^\downarrow$. We have the formulas 
$$(T,g)(T,g)^{-1} = T,\qquad (T,g)^{-1}(T,g) = (g^{-1}\cdot T \cup \{g^{-1}\}^\downarrow)_U$$
for a Toeplitz-Munn $(E,U)$-tree $(T,g)$.  
\end{remark}

Let $(E,C)$ be a separated graph. The elements in $\Y ^\times = \bigsqcup_{v\in E^0} \Y (v)$ can be interpreted as certain $E$-trees. Concretely, the $E$-trees $T$ which belong to $\Y^\times$ are characterized by the property that the geodesic path between any two vertices of $T$ is a $C$-separated path (see \cite[Remark 3.8]{ABC25}).
The elements of $\Y ^\times $ will be called {\it $C$-compatible $E$-trees}.

For the next definition, recall that, whenever $(E,C)$ is a separated graph and $g\in \FG (E)$, we have introduced in Definition \ref{def:decomposition-gLw} a canonical decomposition $g=g_Lw$. 

\begin{definition}
	\label{def:Leavitt-Munn-tree}
	Let $(E,C)$ be a separated graph. 
		\begin{enumerate}
		\item[(a)] For $v\in E^0$, we denote by $\Y _L (v)$ the set of all elements $T$ in $\Y(v)$ such that each maximal element of $T$ does not end in $E^{-1}$ and does not end in an edge $e\in E^1$ such that $\{e\}\in C$. Observe that $\Y_L(v)\subseteq \Y_0(v)\subseteq \Y(v)$. We set $\Y _L = \bigsqcup_{v\in E^0} \Y_L (v)$. We refer to the elements of $\Y_L$ as {\it Leavitt $(E,C)$-trees}. 
		\item[(b)]  A {\it Leavitt-Munn $(E,C)$-tree} is a pair $(T,g)$, where $T\in \Y_L$, and $g= g_L w\in \FC $ satisfies that $g_L\in T$. 	
			\end{enumerate}
\end{definition}

When $(E,C)$ is a separated graph, we can specialize the above results taking the subset $U$ of $E^1$ consisting of all the edges $e\in E^1$ such that $\{ e \}\in C$. This leads to the following concepts.

\begin{definition}
	\label{def:T-sub-L} Let $(E,C)$ be a separated graph. We shall denote by $\sim_L$ the congruence on $\FIS(E)$ generated by the relations $e^{-1} e\sim_L r(e)$ for all $e\in E^1$ and $ee^{-1}\sim_L s(e)$ whenever $e\in E^1$ and $\{ e\}\in C$. The corresponding equivalence classes are denoted by $[T,g]_L$, or $[T]_L$ if $T\in \mathcal E (\FIS(E))$.

	For each $E$-tree $T$, we denote by $T_L$ the unique element $T_U$ described in Lemma \ref{lem:uniqueTU}, where $U$ is the subset of $E^1$ consisting in all edges $e\in E^1$ such that $\{ e \}\in C$. Hence $T_L$ is the unique element of $[T]_L$ such that each maximal element of $T_L$ does not end in $E^{-1}$ and does not end in $e$, for $e\in E^1$ with $\{ e\}\in C$.  	
\end{definition}

 We now observe that $\Y$ is saturated with respect to the equivalence relation $\sim_L$ on $\FIS (E)\cong \mathcal M$.

\begin{lemma}
	\label{lem:extesnio-to-CcompatibleEtree}
	Let $(E,C)$ be a separated graph and let $T$ be a $C$-compatible $E$-tree. Then each $S\in \FIS(E)$
	with $S\sim_L T$ is also $C$-compatible.
		 In particular, if $(T,g)$ is a Leavitt-Munn $(E,C)$-tree, then $T\cup g^\downarrow$ is $C$-compatible, and thus it belongs to $\Y$.  
	\end{lemma}

\begin{proof}
	By Lemma \ref{lem:uniqueTU}, it is enough to show that if $T$ is an $E$-tree and $T_L$ is $C$-compatible, then $T$ is also $C$-compatible. 
	
	Let $T$ be an $E$-tree such that $T_L$ is $C$-compatible. Given two vertices $g$ and $h$ in $T$, we have to show that $g^{-1}\cdot h$ is a $C$-separated path. 
	
	We first observe that all elements of $T$ are $C$-separated. Take $g\in T$. By the construction of $T_L$ we have $g_L\in T_L$, so that we can write $g= g_L w$, where $g_L\in T_L$ and $w$ is a product of inverse edges and of edges $e$ such that $\{ e\}\in C$. Since $T_L$ is $C$-compatible, $g_L$ is $C$-separated, and it follows that $g=g_L w$ cannot contain a subword of the form $e^{-1}f$, where $e,f$ are distinct edges in $E$ such that $e,f\in X$ for $X\in C$.
	
	Consider now two vertices $g,h\in T$. If one of them is a prefix of the other, then we deduce from the above that $g^{-1}\cdot h$ is $C$-separated. Hence we may assume that $g= dxg'$ and $h= dyh'$, where $x,y\in E^1\cup E^{-1}$ and $x\ne y$. We need to show that 
	$g^{-1}\cdot h = (g')^{-1} x^{-1}y h'$ is $C$-separated. Since $(g')^{-1}x^{-1}$ and $yh'$ are $C$-separated, we only have to show that $x^{-1}y$ is not of the form $e^{-1}f$ for $e,f$ distinct edges in the same set of the partition $C$. If $dx$ is a prefix of $g_L$ and $dy$ is a prefix of $h_L$, then the above property holds because $g_L,h_L\in T_L$, and $T_L$ is $C$-compatible. If $g_L$ is a proper prefix of $dx$, then either $x\in E^{-1}$ or $x\in E^1$ and $\{x\}\in C$, hence $x^{-1}y$ is not of the form $e^{-1}f$ described above. The same conclusion can be reached if $h_L$ is a proper prefix of $dy$. This completes the proof.      
	\end{proof}

We can finally obtain our main result for the inverse semigroup $\LI (E,C)$.

\begin{theorem}
	\label{thm:ECLeavittMunntrees}
	Let $(E,C)$ be a separated graph. Then the Leavitt  inverse semigroup $\LI (E,C)$ is isomorphic to the inverse semigroup consisting of all Leavitt-Munn $(E,C)$-trees, endowed with the following product
	$$(T_1,g_1)\cdot (T_2,g_2) = \begin{cases} \Big( (T_1'\cup g_1\cdot T_2')_L, \,\, g_1\cdot g_2\Big) & \text{ if } T_1'\cup g_1\cdot T_2' \text{ is }C\text{-compatible} \\ \qquad \,\,  \qquad \,\, 0 & \text{ otherwise }  
	\end{cases}   $$
	where $T_1',T_2'\in \Y$ are any representatives of $T_1$ and $T_2$ respectively, such that $g_1\in T_1'$ and $g_2\in T_2'$.
	The inverse is given $(T,g)^{-1} = ((g^{-1}\cdot T')_L, g^{-1})$, where $T'$ is any representative of $T$ such that $g\in T'$. 
	The map $(T,g)\mapsto \omega (g)$ provides an idempotent pure partial homomorphism $\LI(E,C)^\times \to \F$, which determines a semidirect product decomposition
	$$\LI (E,C) \cong \YY^L \rtimes_{\theta}^r \F,$$
	where $\YY^L = \Y/{\sim_L}\cong \mathcal E (\LI (E,C))$. The partial action $\theta^L$, with $\theta^L_g\colon D_{g^{-1}}^L\to D_g^L$ is defined as follows:
	For each $g\in \FC(v)\setminus \{v\}$, $D_g^L$ can be identified with the set of all those $T\in \Y_L(v)$ such that $g_L\in T$, and the action $\theta_g^L$ is defined by $\theta_g^L(T)= (g\cdot T')_L$, where $T'$ is a representative of $T$ containing $g^{-1}$. 
\end{theorem}

\begin{proof}
	The proof follows the pattern of \cite[proof of Theorem 4.5]{ABC25}. We include details for the convenience of the reader. 
	
	If $S$ is any strongly $E^*$-unitary inverse semigroup with an idempotent pure partial homomorphism $\theta \colon S^\times \to \Group$, and $I$ is an ideal of $S$, then $S/I$ is also strongly $E^*$-unitary with the same group $\Group$.
	This is clear since $(S/I)^\times$ can be identified with $S\setminus I$, and thus the restriction of $\theta $ to $S\setminus I$ gives the desired idempotent pure partial homomorphism. 
	
	Now let $(E,C)$ be a separated graph. Since the Leavitt inverse semigroup $\LI (E,C)$ is the Rees quotient of the inverse semigroup $\mathcal M _L = \FIS (E)/{\sim_L}$ by the ideal generated by all the elements of the form $(e^{-1}f)^\downarrow$, for distinct $e,f\in X$, $X\in C$, it follows from Proposition \ref{prop:Toeplitz-Munn-EUtrees} and the above observation that $ \LI (E,C)$ is a strongly $E^*$-unitary semigroup, with an idempotent pure partial homomorphism $\pi \colon \LI (E,C)^\times \to \F$. The result now follows from \cite[Proposition 2.1]{ABC25} and Proposition \ref{prop:Toeplitz-Munn-EUtrees}.  
\end{proof}

\section{Normal form in $\Le_K ^\ab (E,C)$}
\label{sect:normal-form}

In order to find a linear basis for $\Le_K^\ab(E,C)$ we will follow the route of \cite{AG12}, using Diamond's Lemma but this time for commutative algebras. 

Consider the commutative subalgebra $\mathcal C$ of $\Co_K^\ab(E,C)$ generated by the idempotents of $\IS(E,C)$. By \cite[Theorem 4.8]{ABC25}, each idempotent in $\IS (E,C)$ corresponds uniquely to an
$E$-tree $T\in \Y_0$, and the idempotent of $\Co_K^\ab(E,C)$ corresponding to $T$ is
$$\e (T) := \prod_{\lambda\in \max (T)} \lambda \lambda^*\in  \Co_K^\ab(E,C).$$ 
Therefore $\mathcal C$ coincides with the linear span of the family $\{ \e (T) : T\in \Y_0\}$, and it is the contracted semigroup algebra $K[\mathcal E (\IS (E,C))]$. 

For each $g\in \FC(v)\setminus \{v\}$, denote by $\mathcal C_g$ the linear span of $D_g$, where $D_g$ is the set of those $T\in \Y_0(v)$ such that $g_0\in T$, see \cite[Theorem 4.5]{ABC25}. Set $\mathcal C_1= \mathcal C$ and $\mathcal C_g = \{0\}$ if $g\in \F \setminus \FC$.  
Then the canonical structure of $\Co_K^\ab(E,C)$ as a partial crossed product is given by
$$\Co_K^\ab(E,C) = \bigoplus _{g\in \FC} \mathcal C_g \rtimes g.$$
The ideal $\mathcal Q$ is generated by homogeneous elements from the neutral component $\mathcal C$, hence the quotient algebra 
$\Le_K^\ab(E,C)$ has the structure of a partial crossed product
$$ \Le_K^\ab(E,C) = \bigoplus _{g\in \FC} \mathcal L_g \rtimes g ,$$
where $\mathcal L _g = \mathcal C_g/\mathcal Q_g$, with $\mathcal Q_g = \mathcal C_g \cap \mathcal Q_1$, where $\mathcal Q_1= \mathcal Q \cap \mathcal C$ is the corresponding invariant ideal of $\mathcal C$. 

Similarly we may view $\Le_K^\ab (E,C)$ as a quotient of the contracted semigroup algebra $K[\LI (E,C)]$ where $\LI (E,C)$ is the Leavitt inverse semigroup of $(E,C)$, as introduced in Section \ref{sect:Leavitt-inversem}.
By Theorem \ref{thm:ECLeavittMunntrees}, the elements of $\LI (E,C)$ can be identified with the Leavitt-Munn $(E,C)$-trees (see Definition \ref{def:Leavitt-Munn-tree}), and we will show below that the elements associated to a certain subfamily of those provide a linear basis for $\Le_K^\ab (E,C)$.

\begin{definition}
	\label{def:reduced-paths}
	Let $(E,C)$ be a separated graph, and let $\Cfin$ denote the subset of $C$ consisting of elements $X\in C$ such that $|X|<\infty$. We define a {\it choice function} as a function $\mathfrak E\colon \Cfin \to E^1$ such that $\mathfrak E (X) \in X$ for all $X\in \Cfin$.
We say that $\lambda \in \FC$ is $\mathfrak E$-reduced if it does not end with $\mathfrak E (X)$ for any $X\in \Cfin $. 

	  We say that an element $T\in \Y _0$ is {\it $\mathfrak E$-reduced} if each maximal element of $T$ is $\mathfrak E$-reduced. Note that, if $T$ is $\mathfrak E$-reduced, then $T\in \Y_L$, because the maximal elements of $T$ do not end in an edge $e$ of $E$ such that $\{e\}\in C$.  
\end{definition}

The above definition is suitable to find a linear basis for $\mathcal L_1 = \mathcal C_1/\mathcal Q_1$. We will give a suitable modification of it later in order to find a basis for $\mathcal L_g = \mathcal C_g/\mathcal Q_g$. The reason for this is that the reduction process needed to find the normal form of the elements does not respect the fact that the elements belong to the linear span of $D_g$, for $g\in \FC \setminus \{1\}$. 

Observe that due to the two relations (SCK1) and (SCK2), we have
$gg^* = g_Lg_L^*$ in $\Le_K^\ab(E,C)$, indeed this holds within the Leavitt inverse semigroup. Hence, the correct domains for the partial action associated to $\LI (E,C)$ are of the form 
$D_g^L := \{ T\in \Y_L : g_L\in T \}$, see Theorem \ref{thm:ECLeavittMunntrees}.

The following extreme example clarifies the situation.

\begin{example}
	\label{exam:free-group-and-Dgprime}
	Let $(E,C)$ be the Cuntz separated graph, consisting of a single vertex, and where all elements of $C$ are singletons (see \cite[Example 3.4]{ABC23}). Then for each $g\in \F\setminus \{1\}$, $D_g^L$ is trivial. More generally, if $(E,C)$ is a separated graph with the free separation, that is, all elements of $C$ are singletons, then 
	the only idempotents of $\LI (E,C)$ are the vertices of $E$, and $D_g^L = \{\{v\}\}$ whenever $g\in \FC (v)\setminus \{v\}$. 
\end{example}

We first abstractly describe the commutative subalgebra $\mathcal C$ of $\Co_K^\ab(E,C)$ generated by the idempotents of $\IS(E,C)$. 

\begin{lemma}
	\label{lem:algebra-mathcalC-abstractly-described}
	Let $\mathcal C$ be the (commutative) subalgebra of $\Co_K^\ab(E,C)$ generated by the idempotents of $\IS(E,C)$. Then 
	$$\mathcal C \cong \Pol_K[\FC]/N,$$
	where $\Pol_K[\FC]$ denotes the commutative polynomial algebra on the set $\FC$, and $N$ is the ideal of $\Pol_K[\FC]$ generated by the following elements
	\begin{enumerate}
		\item $\lambda \cdot \mu $ if $s(\lambda)\ne s(\mu)$. 
		\item $\lambda \cdot \mu - \lambda $ for $\mu \le_p \lambda$,
		\item $\lambda \cdot \mu $ if $\lambda$ and $\mu $ are $C$-incompatible.
		\item $\lambda e^{-1} - \lambda$ where $\lambda e^{-1}\in \FC$ and $e\in E^1$.  
	\end{enumerate}
\end{lemma}

\begin{proof} Set $R:= \Pol_K[\FC]$.
	There is a $K$-algebra homomorphism
	$$\varphi \colon R/N \to \mathcal C$$
	given by $\varphi (\lambda) = \lambda \lambda^*$ for all $\lambda \in \FC$. In view of \cite[Theorem 4.8]{ABC25}, this map is surjective. Also, observe that, using the relations of $N$, we can bring each monomial of $R/N$ to a monomial of the form
	$$\lambda _1 \cdot \lambda _2 \cdots \lambda_r,$$
	where each $\lambda _i \in \FC$, all $\lambda _i $ do not end in $E^{-1}$ and $\{\lambda_1,\dots , \lambda_r\}$ is an incomparable $C$\nb-compatible subset of $\FC (v)$ for some $v\in E^0$. It follows that this set of monomials linearly spans $R/N$. Again by \cite[Theorem 4.8]{ABC25}, the images by $\varphi$ of these monomials are linearly independent in $\Co_K^\ab(E,C)$, hence they are also linearly independent in $R/N$, and $\varphi$ is injective.  
\end{proof}

Observe that the set of monomials in $\Pol_K[\FC]$ appearing in the proof of Lemma \ref{lem:algebra-mathcalC-abstractly-described}  corresponds exactly  to the set of elements in $\mathcal E = \mathcal E (\IS (E,C))$ in Scheiblich normal form. One can directly deduce that these elements form a basis of $R/N$ by an application of the results in \cite{Bergman-Diamond}, using the reduction system consisting of the pairs (1)--(4) in the next lemma. Adding a new set of pairs (5) will give the basis for the neutral component $\mathcal L_1$ of $\Le_K^\ab(E,C)$, as follows:

\begin{theorem}
	\label{thm:basis-forLabneutralcomp}
	Let $(E,C)$ be a separated graph and let $\mathfrak E\colon \Cfin \to E^1$ be a choice function. Then a $K$-basis for the neutral component $\mathcal L_1 = \mathcal C_1/\mathcal Q_1$ of $\Le_K^\ab(E,C)$ is given by the following set: 
	$$\ol{\mathcal B}_1 (E,C) = \{ \e (T) : T \text{ is $\mathfrak E$-reduced } \}.$$
\end{theorem} 

\begin{proof}	We show that the set of idempotents $\{\e (T) : T \text{ is $\mathfrak E$-reduced} \}$ is a $K$-basis for $\mathcal L_1$.    
	
	Set $e_X:= \mathfrak E (X)\in X$ for all $X\in \Cfin$. 
	
	We apply Bergman's Diamond's Lemma for commutative algebras \cite{Bergman-Diamond}, see especially \cite[Theorem 1.2 and Section 10.3]{Bergman-Diamond}. We define $\text{wt}(e_X)= 2$ and $\text{wt}(x)= 1$ for any other $x\in E^1\cup E^{-1}$. We consider a weight  function on $\FC$ such that the weight of $\lambda \in \FC$ is the sum of the weights of the edges of $\lambda$, or $\text{wt}(\lambda) = 1$ if $\lambda = v\in E^0$. We extend this weight function to the free commutative semigroup $W:= [\FC ]$ generated by $\FC$ by declaring the weight of a monomial $\lambda_1 \cdots \lambda_r$ as the sum of the weights $\text{wt}(\lambda_i)$, $i=1,\dots, r$.  
	
	For $X\in \Cfin$, set $X':= X \setminus \{e_X\}$. 
	
	We now consider the following reduction system $\mathcal S$ on the semigroup algebra $\Pol_K[\FC] = KW$: 
	\begin{enumerate}
		\item $(\lambda \cdot \mu , 0)$ if $\lambda,\mu\in \FC$ and $s(\lambda)\ne s(\mu)$. 
		\item $(\lambda \cdot \mu ,  \lambda )$ if $\lambda,\mu\in \FC$ and $\mu \le_p \lambda$,
		\item $(\lambda \cdot \mu , 0)$ if $\lambda $ and $\mu$ are $C$-incompatible, where $\lambda,\mu \in \FC$.
		\item $(\lambda e^{-1},  \lambda) $ where $\lambda e^{-1}\in \FC$ and $e\in E^1$.
		\item $(\lambda e_X, \lambda - \sum_{e\in X'} \lambda e)$ for $X\in \Cfin$ and $\lambda e_X \in \FC$.  
	\end{enumerate}	 
	
	Observe that the reduction system $\mathcal S'$ consisting of the pairs in (1)--(4) leads to the basis of $\mathcal C$ consisting of idempotents in Scheiblich normal form. 
	
	Define a partial order $\le $ on $W=[\FC]$ by
	$$\lambda \le \mu \iff \lambda = \mu \text{ or } \text{wt}(\lambda)	< \text{wt} (\mu).$$
	Then the partial order $\le $ on $W$ is a semigroup partial order compatible with the reduction system $\mathcal S$. Moreover $\le$ satisfies the descending chain condition. In order to apply the Diamond Lemma, we have to check that all ambiguities are resolvable. 
	We leave to the reader to check that all ambiguities involving pairs (1)--(4) are resolvable.
	We will check those concerning reduction pairs of the family (5), which are necessarily inclusion ambiguities with either (2) or (3). 
	
	Suppose we have an element of the form $\lambda e_X \cdot \mu$ such that $\mu \le_p \lambda e_X$. Then we have an inclusion ambiguity
	$$\lambda e_X \cdot \mu \overset{(2)}{\to} \lambda e_X, \qquad \lambda e_X \cdot \mu \overset{(5)}\to \lambda \cdot \mu  - \sum_{e\in X'} \lambda e \cdot \mu.$$
	Applying (5) to the first, we get
	$$ \lambda e_X \cdot \mu \overset{(2)}{\to} \lambda e_X \overset{(5)}{\to} \lambda - \sum_{e\in X'} \lambda e.$$
	We now distinguish two cases. If $\mu \le _p \lambda$ then we have
	$$\lambda \cdot \mu - \sum _{e\in X'} \lambda e \cdot \mu \overset{(2)}{\to} \lambda - \sum _{e\in X'} \lambda e ,$$
	and we have solved the ambiguity. If $\mu = \lambda e_X$, then 
	$$\lambda \cdot \lambda e_X - \sum _{e\in X'} \lambda e \cdot \lambda e_X \overset{(2)}{\to} \lambda e_X - \sum _{e\in X'} \lambda e \cdot \lambda e_X \overset{(3)}{\to} \lambda e_X ,$$
	and we have solved the ambiguity too.
	
	If $\lambda e_X \le_p \mu$, then we can solve the corresponding inclusion ambiguity for $\lambda e_X \cdot \mu$ just as in the case where $\mu = \lambda e_X$.  
	
	Finally we consider inclusion ambiguities involving (3) and (5). Suppose that 
	$\lambda e_X$ and $\mu$ are $C$-incompatible. If $\lambda $ and $\mu $ are already $C$-incompatible then we have
	$$ \lambda e_X \cdot \mu \overset{(2)}{\to}  0, \quad \lambda e_X \cdot \mu \overset{(5)}{\to} \lambda \cdot \mu - \sum_{e\in X'} \lambda e \cdot \mu \overset{(2)}{\to} 0-0 = 0.$$
	If $\mu = \lambda e \mu'$, where $e\in X'$, then 
	$$ \lambda e_X \cdot \mu \overset{(5)}{\to} \lambda \cdot \lambda e \mu'  - \sum_{f\in X'} \lambda f \cdot \lambda e \mu' \overset{(2),(3)}{\longrightarrow} 
	\lambda e \mu' - \lambda e \cdot \lambda e \mu' \overset{(2)}{\to} \mu - \mu = 0.$$
	
	The irreducible forms are precisely the elements which correspond to the family $\ol{\mathcal B}_1(E,C)$, that is, the elements
	$$\lambda _1\cdots \lambda_r ,$$
	where $\lambda _i $ does not end in $E^{-1}$, does not end in $e_X$ for $X\in \Cfin$, and $\lambda_1,\cdots ,\lambda _r$ is an incomparable $C$-compatible family of elements of $\FC (v)$, for some $v\in E^0$. 
	
	Hence we obtain the desired result for the neutral component.    
\end{proof}

In order to obtain a linear basis for the non-neutral components of $\L_K^\ab (E,C)$, we need the following concept:

\begin{definition}
	\label{def:relatively-reduced}
	Let $(E,C)$ be a separated graph, let $\mathfrak E\colon \Cfin \to E^1$ be a choice function, and let $g\in \FC$. 
	We say that $T\in D_g^L$ is $\mathfrak E$-reduced relatively to $g_L$ if each element in $\max (T)\setminus \{ g_L\}$ is $\mathfrak E$-reduced. That is, if either $g_L$ is $\mathfrak E$-reduced or $g_L$ is not maximal in $T$, then $T$ is $\mathfrak E$-reduced relatively to $g_L$ if and only if $T$ is $\mathfrak E$-reduced, and if $g_L$ is not $\mathfrak E$-reduced and $g_L$ is maximal in $T$, then $T$ is $\mathfrak E$-reduced relatively to $g_L$ if and only if all maximal elements of $T$ except $g_L$ are $\mathfrak E$-reduced.  
	\end{definition}

\begin{theorem}
	\label{thm:full-basis-forLab}
	A basis of $\Le_K^\ab(E,C)$ is given by the elements of the form 
	$$\ol{\mathcal B} (E,C) :=\{ \e (T)\rtimes g : T \in D_g^L \text{ and } T \text{ is $\mathfrak E$-reduced relatively to $g_L$} \}.$$  
\end{theorem}

\begin{proof}
	We work within $\Co_K^\ab(E,C)$, where we have a natural linear basis, namely $\IS (E,C)$, whose elements admit a unique normal form, the Scheiblich normal form \cite[Theorem 4.6]{ABC25}. For $g\in \FC(v)\setminus \{v\}$, the ideal $\mathcal C_g$ of $\mathcal C$ is the linear span of the elements $\e (T)$, where $T\in \Y_0 (v)$ and $g_0\in T$, and the component $\mathcal C_g\rtimes g$ is the linear span in $\Co_K^\ab(E,C)$ of corresponding terms $\e (T)\rtimes g$.

	A {\it monomial} in $\mathcal C =\mathcal C_1$ is an element of the form 
$\lambda_1 \lambda_1^* \lambda_2 \lambda_2^* \cdots \lambda _r \lambda_r^*$, where $\lambda_i \in \FC$. 
	It follows from the proof of Theorem \ref{thm:basis-forLabneutralcomp} that 	
every element of $\mathcal C $, written as a linear combination of monomials, can be transformed, using the reductions (1)--(5) applied to these monomials, to a linear combination of monomials in $\mathfrak E$-reduced Scheiblich form, that is, monomials of the form $\e (T)$ for $T$ a $\mathfrak E$-reduced $C$-compatible $E$-tree. Moreover this linear combination does not depend on the particular reductions that we use. We denote the $\mathfrak E$-reduced Scheiblich normal form of $a\in \mathcal C$ by $n(a)$. Of course $n$ is linear: $n(\sum \alpha_i a_i) = \sum \alpha_i n(a_i)$, where $\alpha_i$ are scalars and $a_i \in \mathcal C $.
	
	Now if $g\in \FC$ and  $a\in \mathcal C_{g}$ there is no guarantee that $n(a) \in \mathcal C_{g_L}$. We define
	$$n^g(a) = n(a) g_Lg_L^*,$$
	and we call $n^g(a)$ the $\mathfrak E$-reduced Scheiblich normal form of $a$ relative to $g$. Note that $n^g$ is linear on $\mathcal C_g$.
	
	Since $a\in \mathcal C_g = gg^* \mathcal C _1$, and $g_Lg_L^* gg^*= gg^*$ we have $ag_Lg_L^* = a$, hence
	$$a-n^g(a) = g_Lg_L^*(a-n(a)) \in  g_Lg_L^* \mathcal Q_1 = \mathcal Q_{g_L}.$$
	In particular $a$ and $n^g(a)$ represent the same element of $\Le_K^\ab (E,C)$. 
	
	Moreover it is clear that $n^g(a)$ is a linear combination of monomials in $\mathfrak E$-reduced form relatively to $g_L$.
	
	We want to show that $n^g(a)$ is the unique linear combination of monomials in $\mathfrak E$-reduced form relatively to $g_L$ that represents $a$. Let $\sum \alpha _i a_i$ be a linear combination of monomials $a_i$ which are in $\mathfrak E$-reduced form relatively to $g_L$ and such that $\sum \alpha_i a_i$ represents $a$.
	
	Suppose that we can prove that $n^{g} (a_i)=a_i$ for all $i$. We will show then that $n^g(a) = \sum \alpha_i a_i$. 
	For, observe that since $a$ and $\sum \alpha_ia_i$ represent the same element of $\Le_K^\ab (E,C)$, we have
	$$n(a) = n(\sum \alpha_i a_i) = \sum \alpha _i n(a_i).$$
	Therefore 
	$$n^g(a) = g_Lg_L^* n(a) = \sum \alpha _i g_Lg_L^* n(a_i) = \sum \alpha_i n^{g} (a_i) = \sum \alpha _i a_i,$$
	where in the last step we have used that $n^{g} (a_i) = a_i$ for all $i$.
	
	Hence we only need to show that $n^{g}(a) = a$ for each monomial $a$ in $\mathfrak E$-reduced form relatively to $g_L$. Let $a$ be one of such monomials, and write $a= \e (T)$ for $T\in \mathcal X_L$. We can assume that $g_L\in \max (T)$ and that $g_L$ is not $\mathfrak E$-reduced. 
	Set $A= \max(T)\setminus \{g_L\}$.
	Write 
	$$g_L = g_1 w_1 e_{X_1}\cdots w_r e_{X_r},$$ where $X_i\in \Cfin$, each $w_i$ is either trivial or a product of elements of the form $x^*$ for $x\in E^1$, and either $g_1\in A^{\downarrow}$ and $g_1w_1e_{X_1}\notin A^{\downarrow}$, or $g_1\notin A^{\downarrow}$ and $g_1$ is $\mathfrak E$-reduced. Set $X'= X\setminus \{e_X\}$ for $X\in \Cfin$. In either case we have
	$$n(a) = \e ((\{g_1\} \cup A)^{\downarrow}) - \sum _{i=1}^r \sum _{e\in X_i'} \e ( (\{g_1 w_1 e_{X_1} \cdots w_ie\}\cup A)^{\downarrow}). $$
	Now observe that 
	$$g_Lg_L^* \e ((\{g_1\} \cup A)^{\downarrow}) = 
	\e ( (\{g_L\}\cup A)^{\downarrow}) = a$$
	and that for all $i\in \{ 1,\dots , r\}$ and $e\in X_i'$
	$$g_Lg_L^*\e ( (\{g_1w_1e_{X_1}\cdots w_ie\}\cup A)^{\downarrow}) = 0, $$
	because $g_L$ and $g_1w_1e_{X_1}\cdots w_ie$ are $C$-incompatible. We conclude that 
	$$n^g(a) = [g_Lg_L^* n(a)]= a,$$
	as desired. 
	
	Hence all elements of $\Le _g \rtimes g$ have a unique representative which is a linear combination of monomials in $\mathfrak E$-reduced form relatively to $g_L$, and we conclude that $\ol{\mathcal B}(E,C)$ is a linear basis of $\Le_K ^\ab (E,C)$.  
\end{proof}

We say that an element $\e (T) \rtimes g \in \ol{\mathcal B} (E,C)$ is in $\mathfrak E$-reduced Scheiblich normal form.  The {\it $\mathfrak E$-reduced Scheiblich normal form} of an element $a\in \Le_K ^\ab (E,C)$ is the unique linear combination of elements of the basis $\ol{\mathcal B}(E,C)$ that represents $a$.

\begin{example}
	\label{exam:non-separated-case}
We observe here that, in the non-separated case, the basis $\ol{\mathcal B}(E,C)$ from Theorem \ref{thm:full-basis-forLab} is closely related to the basis given in \cite[Theorem 1]{zel} (see also \cite[Corollary 1.5.12]{AAS}). Indeed, we show below that these two bases are essentially connected through the Scheiblich normal form. This gives further evidence that this is the appropriate basis for $\Le_K ^\ab (E,C)$.

Suppose that $a= \lambda \mu^*$ is a basis element in the Cohn path algebra $\Co_K(E)$ of a non-separated graph $E$, with $r(\lambda ) = r(\mu)$. Then we write 
$$\lambda = \lambda_1 \lambda_2, \qquad \mu = \mu_1 \lambda _2,$$
where $\lambda_1$ and $\mu_1$ do not end with the same edge.
Now $\lambda _1\mu_1^{-1}\in \FC$ is a reduced word, and
$$a= \lambda \mu^* =\lambda_1 \lambda_2 \lambda_2^* \mu_1^* = \lambda _1(\lambda _2\lambda_2^*)(\lambda_1^*\lambda_1)\mu_1^* = (\lambda \lambda^*) (\lambda_1 \mu_1^*) ,$$
so that the Scheiblich normal form of $a$ is precisely 
$$a= (\lambda\lambda^*)(\lambda_1\mu_1^*).$$

Let $\mathfrak E$ be a choice function for $E$, that is, for each regular vertex $v$ of $E$, we have  $e_v:= \mathfrak E (v)\in s^{-1}(v)$. 

	Suppose that $a=\lambda \mu^*$ is in $\mathfrak E$-reduced form with respect to a choice function $\mathfrak E$, in the sense of \cite{zel}. This means that  the expression $\lambda\mu^*$ does not contain the subword $e_ve_v^*$, for regular vertices $v$ in $E$.
	
	 Suppose first that $\lambda_2$ in the above decomposition is non-trivial. Then, with $g:= \lambda_1\mu_1^{-1}\in \FC$, $\e (\lambda) \rtimes g$ is in $\mathfrak E$-reduced Scheiblich normal form, and we clearly have $g_L <_p \lambda$. 
	 
	 Conversely, if $a=\e (\lambda) \rtimes g $  is in $\mathfrak E$-reduced Scheiblich normal form and $g_L <_p \lambda$, then we can write $g= g_L x_1\cdots x_r \mu^{-1}$ in reduced form in $\FC$, where $s^{-1}(s(x_i)) = \{ x_i\}$ for $i=1,\dots , r$, $r\ge 0$, and moreover $\lambda$ is $\mathfrak E$-reduced in the sense of Definition \ref{def:reduced-paths}. It follows that $\lambda = g_L x_1\cdots x_r \lambda _2$ for a non-trivial $\mathfrak E$-reduced path $\lambda_2$, and we get
	 \begin{align*}
	 a= & \e (\lambda )\rtimes g= g_L x_1\cdots x_r \lambda_2 \lambda _2^* x_r^*\cdots x_1^* g_L^* g_L x_1\cdots x_r \mu^* \\
	  =  & g_L x_1\cdots x_r \lambda _2 \lambda _2 ^* \mu^* = (\lambda_1\lambda _2) (\mu \lambda_2)^*,
	 \end{align*}
	 where $\lambda _1 = g_Lx_1\cdots x_r$, and $a= \lambda (\mu \lambda_2)^*= (\lambda_1\lambda _2)(\mu \lambda_2)^*$ is the $\mathfrak E$-reduced normal form of $a$ in the sense of \cite{zel}.  
	 
	 Suppose now that $\lambda _2$ is trivial. Then $g := \lambda \mu^{-1}$ is already a reduced word in $\FC$, and the $\mathfrak E$-reduced Scheiblich normal form of $a$ is $ \e (g_L) \rtimes g$, where $\lambda = g_L x_1\cdots x_r$, with $s^{-1}(s(x_i))= \{x_i\}$ for all $i$ and $r\ge0$. In this case, $g_L$ may be or may not be $\mathfrak E$-reduced in the sense of Definition \ref{def:reduced-paths}, but in any case 
	 $\e (g_L)\rtimes g$ is $\mathfrak E$-reduced relatively to $g_L$. Conversely, if we have a (reduced) element $g= \lambda \mu^{-1}$ in $\FC$, then $\e(g_L)\rtimes g$ is the $\mathfrak E$-reduced Scheiblich normal form of $\lambda \mu^*$, and $\lambda \mu^* $ is its $\mathfrak E$-reduced normal form in the sense of \cite{zel}. Incidentally, note that both normal forms do not depend on the choice function $\mathfrak E$ in the case where $\lambda_2$ is trivial.   
	 \end{example}

We can now characterize the (inverse) semigroup of $\Le_K^\ab (E,C)$ generated by
$E^0\cup E^1 \cup (E^1)^*$. First, we need a definition and an easy lemma.

For an $E$-tree $T$, we define the {\it total length} $|T|$ as the number of edges of $T$. Here an edge of $T$ is a pair $(g,gx)$, where $g,gx\in T$ and $x\in \hat{E}^1$.

\begin{lemma}
	\label{lem:total-length}
	Let $T$ be a non-trivial $E$-tree and let $hx\in \max (T)$, where $x\in \hat{E}^1$. Then $T\setminus \{hx\}$ is also an $E$-tree, and $|T\setminus \{hx\}| = |T|-1$.
	\end{lemma}

\begin{proof}
	It is clear that removing a maximal element of a tree will produce another tree. The formula $|T\setminus \{hx\}| = |T|-1$ is also clear because $T\setminus \{hx\}$ has exactly one edge less than $T$, namely the edge $(h,hx)$. 
\end{proof}

\begin{theorem}
	\label{thm:semigroup-generated} Let $(E,C)$ be a separated graph. Then
	the natural homomorphism $\iota \colon \LI (E,C)\to \Le_K^\ab (E,C)$ is injective. Hence the semigroup of $\Le_K^\ab (E,C)$ generated by $E^0\cup E^1\cup (E^1)^*$ is naturally isomorphic to the Leavitt inverse semigroup $\LI (E,C)$. 
\end{theorem}

\begin{proof} We will use the picture of $\LI (E,C)$ established in Theorem \ref{thm:ECLeavittMunntrees}. 
	With this picture, we have $\iota ((T,g)) = \e (T)\rtimes g$ for any Leavitt-Munn $(E,C)$-tree $(T,g)$. 
	We fix a choice function $\mathfrak E\colon \Cfin\to E^1$ for $(E,C)$.

	Let $(T,g)$ be a Leavitt-Munn $(E,C)$-tree. Then $T\in \Y_L$ and $g_L\in T$. Set $A := \max (T) \setminus \{g_L\}$, and $B := \{ g \in A : g \text{ is not } \mathfrak E\text{-reduced}\, \}$.  
	If $B=\emptyset$, then $\e (T) \rtimes g$ is $\mathfrak E$-reduced relatively to $g_L$, so that $\iota ((T,g))= \e (T)\rtimes g \in \ol{\mathcal B}(E,C)$

If $B\ne \emptyset$, 
	write $B= \{h_1e_{X_1},\dots , h_r e_{X_r}\}$, where $X_i\in \Cfin_{r(h_i)}$ and $|X_i| \ge 2$ for $i=1,\dots , r$. Set $X'=X\setminus \{e_X\}$ for any $X\in \Cfin$ with $|X|\ge 2$.

	Using the relation $r(h_1)= e_{X_1}e_{X_1}^* + \sum _{x\in X_1'} xx^*$, we can write in $\Le_K^\ab (E,C)$
$$\e (T)\rtimes g = - \sum_{x\in X_1'} \e ((T\setminus \{h_1e_{X_1}\})\cup \{h_1x\}\}))\rtimes g + \e (T\setminus \{h_1e_{X_1}\}))\rtimes g.$$
By Lemma \ref{lem:total-length}, $|T\setminus \{h_1e_{X_1}\}| = |T| -1$, and observe that 
$$| (T\setminus \{h_1e_{X_1}\})\cup \{h_1x\} | = |T|$$
and that $\max((T\setminus \{h_1e_{X_1}\})\cup \{h_1x\})\setminus \{g_L\} = (A \setminus \{h_1e_{X_1}\})\cup \{h_1x\}$, for each $x\in X_1'$. On the other hand, the $E$-tree $T\setminus \{h_1e_{X_1}\}$ is obviously $C$-compatible, but it does not necessarily belong to $\Y_0$ or $\Y_L$. However its reduction to a tree in $\Y_L$, namely $(T\setminus \{h_1e_{X_1}\})_L$, will have total length strictly less than $|T|$.

For $(x_1,\dots ,x_r)\in X'_1\times \cdots \times X'_r$, set
$$T_{(x_1,\dots , x_r)} = (T\setminus \{h_1e_{X_1},\dots , h_re_{X_r}\})\cup \{h_1x_1,\dots ,h_rx_r\}.$$
Also, set $U:= \mathfrak E (\Cfin) = \{ e_X : X\in \Cfin\}$, and observe that the set of edges $e\in E^1$ such that $s^{-1}(s(e)) = \{ e\}$ is a subset of $U$. Let $T_U$ be the $U$-reduction of $T$, as in Lemma \ref{lem:uniqueTU}.  

Iterating the above process a finite number of times, we obtain an identity
in $\Le_K^\ab (E,C)$ of the form
\begin{equation}
	\label{eq:ng-reduction}
\e (T)\rtimes g = (-1)^r \sum_{(x_i)\in \prod_{i=1}^r X_i'} \e (T_{(x_i)})\rtimes g + \tau + \e (T_U\cup g_L^\downarrow)\rtimes g ,	
\end{equation}
  where $\tau$ is a linear combination of terms of the form $\e (S)\rtimes g\in \ol{\mathcal B} (E,C)$ such that $$|T_U\cup g_L^\downarrow | < |S| < |T_{(x_i)}| = |T|$$
  for all $(x_i)\in \prod_{i=1} ^r X_i'$. Observe that \eqref{eq:ng-reduction} gives the $\mathfrak E$-reduced Scheiblich normal form of $\e (T)\rtimes g$, since all $T_{(x_i)}$ and also $T_U\cup g_L^\downarrow$ are $\mathfrak E$-reduced relatively to $g_L$. 
  
  We are now ready to show the injectivity of $\iota$. Suppose that $\iota ((T,g)) = \iota ((T',h))$ for Leavitt-Munn $(E,C)$-trees $(T,g)$ and $(T',h)$. Since these are homogeneous elements of degree $g$ and $h$ respectively, we must have $g=h$. Let $A'$ and $B'$ be the sets as above corresponding to $T'$. Set
  $B'= \{h_1'e_{Y_1},\dots , h_s' e_{Y_s}\}$, where $Y_j\in \Cfin_{r(h_i')}$ and $|Y_j| \ge 2$ for $j=1,\dots , s$. We have a formula similar to \eqref{eq:ng-reduction} for $T'$, involving the $\mathfrak E$-reduced
  relatively to $g_L$ trees $T'_{(y_1,\dots , y_s)}$ for $(y_j)\in \prod_{j=1}^s Y_j'$, and $T'_U\cup g_L^\downarrow$. 
  
  Since $\e(T)\rtimes g = \e (T')\rtimes g$ in $\Le_K ^\ab (E,C)$, it follows that their $\mathfrak E$-reduced Scheiblich normal forms are equal. 
  Looking at the highest length terms and the lowest length terms of the developments of $\e (T)\rtimes g$ and $\e (T')\rtimes g$ in \eqref{eq:ng-reduction}, we conclude that 
  \begin{equation}
  	\label{eq:highest-terms}
  	 (-1)^r \sum_{(x_i)\in \prod_{i=1}^r X_i'} \e (T_{(x_i)}) =  (-1)^s \sum_{(y_j)\in \prod_{j=1}^s Y_j'} \e (T'_{(y_j)})  
  \end{equation}     
and $T_U\cup g_L^\downarrow = T'_U\cup g_L^\downarrow$. We claim that $A\setminus B= A'\setminus B'$. 
By symmetry, it suffices to check one of the inclusions. Let $g\in A\setminus B$. Then obviously $g\notin B'$. We want to check that $g\in A'$. Looking for a contradiction, suppose that $g\notin A'$. Then by
\eqref{eq:highest-terms} there is $j\in \{1,\dots ,s\}$ and $y\in Y_j'$ such that $g=h_j'y$. But since $g\in A= \max (T)\setminus \{g_L\}$, and $g$ is $\mathfrak E$-reduced, it follows that $g\in \max (T_U)$. Hence, using that $T_U\cup g_L^\downarrow = T'_U\cup g_L^\downarrow$, we get $g=h_j'y\in T'_U$, which is a contradiction, since $T'_U$ does not contain any path of the form $h_j'z$, for $z\in Y_j$. This shows that 
$A\setminus B = A'\setminus B'$. It is now easy to see, by using again \eqref{eq:highest-terms}, that $B=B'$.
Indeed, assume that $g=h_{i_0}e_{X_{i_0}}\in B$ and fix an edge $x\in X_{i_0}'$. Then $h_{i_0}x$ is a maximal element in some 
of the the trees $T_{(x_i)}$ and so, by \eqref{eq:highest-terms}, it must be a maximal element in some of the trees $T'_{(y_j)}$. If $h_{i_0}x\in A'\setminus B'$, then $h_{i_0}x\in A\setminus B$, but then 
$h_{i_0}x,h_{i_0}e_{X_{i_0}}\in T$, which is impossible, because $T$ is $C$-compatible. Hence $h_{i_0}x\notin A'\setminus B'$, and it follows that there is $j_0\in \{ 1,\dots ,s\}$ and $y\in Y'_{j_0}$ such that $h_{i_0}x = h_{j_0}'y$, which implies that $h_{i_0}=h_{j_0}'$ and $x=y$. In particular, $X_{i_0}=Y_{j_0}$ and thus
$$g= h_{i_0}e_{X_{i_0}} = h'_{j_0}e_{Y_{j_0}} \in B'.$$
Hence $B\subseteq B'$ and similarly $B'\subseteq B$. Since $A\setminus B = A'\setminus B'$ and $B=B'$, we get $A=A'$ and thus $(T,g)=(T',g)= (T',h)$, as desired.

This concludes the proof. 
\end{proof}

\section{Relation between Cohn path algebras and Leavitt path algebras}
\label{sect:relation-with-Cohn-algebras}

Here we want to determine the relationship between Cohn and Leavitt algebras. It was shown in 
\cite[Theorem 1.5.18]{AAS} that every Cohn path algebra of a non-separated graph is isomorphic to the Leavitt path algebra of another non-separated graph. 
In the separated situation, this apparently cannot be achieved, but there is also a nice relation.

We proceed more generally and relate the class of Cohn-Leavitt algebras with the class of Leavitt path algebras. Cohn-Leavitt path algebras of separated graphs have been introduced in \cite{AG12}. 
Here we will recall its definition, and we will also consider the abelianized (or tame) version of these algebras.

\begin{definition} cf. \cite[Definition 2.5]{AG12}
	\label{def:Cohn-Leavitt}
	Let $(E,C)$ be a separated graph, and let $S$ be a subset of $\Cfin$. Then let $CL_K(E,C,S)$ be the $*$-algebra with generators $E^0\cup E^1$ subject to relations (V), (E1), (E2) and (SCK1) of Definition \ref{def:Cohn-and-Leavitt-algs}, together with relation (SCK2) only for the sets $X\in S$.
 The $*$-algebra $CL_K(E,C,S)$ is called the {\it Cohn-Leavitt path algebra} of $(E,C,S)$. 	
	
	We denote by $CL_K^\ab (E,C,S)$ the tame $*$-algebra associated to the canonical set of generators $E^0\cup E^1$ of $CL_K(E,C,S)$. This will be called the  {\it tame Cohn-Leavitt path algebra} of the triple $(E,C,S)$.  Observe that $CL_K^\ab (E,C,S) = \Co_K^\ab (E,C)/{\mathcal Q_S}$, where $\mathcal Q_S$ is the ideal of $\Co_K^\ab (E,C)$ generated by the projections $q_X$, for $X\in S$.  
	\end{definition}

A linear basis for $CL_K(E,C,S)$ was found in \cite[Theorem 2.7]{AG12}. We will now show that we can also provide a linear basis of $CL_K^\ab (E,C,S)$ by using the same method as we used for $\Le_K^\ab (E,C)$ in 
Theorem \ref{thm:full-basis-forLab}. 

We need to extend Definition \ref{def:Leavitt-Munn-tree} to this new setting.

Given $S\subseteq \Cfin$, we set
$$S^1= \{e\in E^1 : \{e\} \in S\}.$$
Every element $g\in \FC$ can be uniquely written as $g= g_Sw$, where $g_S$ is either a vertex or ends in an edge $e\in E^1$ which does not belong to $S^1$, and $w$ is a product of elements in $E^{-1}$ and elements in $S^1$.   

For $v\in E^0$, define $\Y_S$ as the set of those $T\in \Y_0$ such that all maximal elements of $T$ do not end in an edge $e\in S^1$. 
For $g\in \FC\setminus E^0$, we denote by $D_g^S$ the set of elements $T\in \Y_S$ such that $g_S\in T$. For $S=\Cfin$, we have $g_S=g_L$, $\Y_S=\Y_L$ and $D_g^S=D_g^L$.

\begin{definition}
	\label{def:Sreduced-paths}
	Let $(E,C)$ be a separated graph and let $S$ be a subset of $\Cfin$. 
	We define a {\it $S$-choice function} as a function $\mathfrak E\colon S \to E^1$ such that $\mathfrak E (X) \in X$ for all $X\in S$.
	We say that $\lambda \in \FC$ is $\mathfrak E$-reduced if it does not end with $e_X$ for all $X\in S$. 
	
	We say that an element $T\in \Y _0$ is {\it $\mathfrak E$-reduced} if each maximal element of $T$ is $\mathfrak E$-reduced. 
	
	Let $g\in \FC$. 
	We say that $T\in D_g^S$ is $\mathfrak E$-reduced relatively to $g_S$ if each element in $\max (T)\setminus \{ g_S\}$ is $\mathfrak E$-reduced. That is, if either $g_S$ is $\mathfrak E$-reduced or $g_S$ is not maximal in $T$, then $T$ is $\mathfrak E$-reduced relatively to $g_S$ if and only if $T$ is $\mathfrak E$-reduced, and if $g_S$ is not $\mathfrak E$-reduced and $g_S$ is maximal in $T$, then $T$ is $\mathfrak E$-reduced relatively to $g_S$ if and only if all maximal elements of $T$ except $g_S$ are $\mathfrak E$-reduced.  
	\end{definition}

With these definitions at hand, we can state the following result, whose proof follows the lines of the proof of Theorem \ref{thm:full-basis-forLab}.
 
\begin{theorem}
	\label{thm:full-basis-forCLS}
	Let $(E,C)$ be a separated graph, and let $S$ be a subset of $\Cfin$, and let $\mathfrak E$ be an $S$-choice function for $(E,C)$. Then a linear basis 
	of $CL_K^\ab (E,C,S)$ is given by the elements of the form 
	$$\ol{\mathcal B} (E,C,S) :=\{ \e (T)\rtimes g : T \in D_g^S \text{ and } T \text{ is $\mathfrak E$-reduced relatively to $g_S$} \}.$$  
\end{theorem}

Note that in the extreme cases where $S=\emptyset$ and $S= \Cfin$, we recover the already known basis of $\Co_K^\ab (E,C)$ and $\Le_K^\ab (E,C)$, respectively. 

We are now ready to study the relationship with Leavitt path algebras. 

Let $(E,C)$ be a separated graph and $S$ a subset of $\Cfin$. We build a new separated graph $(E_S,C^S)$ as follows. First we set
$$(E_S)^0 = E^0 \sqcup \{ v_X: X\in \Cfin \setminus S \}$$
$$(E_S)^1= E^1 \sqcup \{ d_X : X\in \Cfin \setminus S \},$$
The source and range maps in $E_S$ are such that the natural map $E\to E_S$ is a graph homomorphism, and moreover 
$$s(d_X) = v,  \quad r(d_X) = v_X  \, \, \text{ if } X\in  \Cfin_v \setminus S.$$
The partitions $C^S_v$ are given by $C^S_v=\{\ol{X} :X\in C_v\}$, where $\ol{X} =X\cup \{d_X\}$ if $X\in \Cfin \setminus S$ and $\ol{X} = X$ otherwise. 

In the case where $S=\emptyset$, so that the Cohn-Leavitt path algebra coincides with the Cohn path algebra, we write $(\ol{E},\ol{C})$ for the separated graph $(E_\emptyset, C^\emptyset)$.

\begin{proposition}
	\label{prop:relationLPACPA}
	Let $(E,C)$ be a separated graph and $S$ a subset of $\Cfin$. Let $p:= \sum_{v\in E^0} v\in \mathcal M (\Le_K (E_S,C^S))$. Then we have a natural $*$-isomorphism
	$$\varphi\colon CL_K (E,C,S) \cong p\Le _K (E_S,C^S)p$$
	such that $\varphi (v)= v$ for all $v\in E^0$ and $\varphi (e)= e$ for all $e\in E^1$. 
	
	In particular $\Co_K (E,C) \cong p\Le _K (\ol{E},\ol{C}) p$. 
\end{proposition}

\begin{proof}
	We have a natural $*$-homomorphism $\varphi \colon CL_K (E,C,S) \to p \Le_K (E_S,C^S) p$ defined by $\varphi (v)= v$ and $\varphi (e)= e$ for all $v\in E^0$ and $e\in E^1$. Note that $R:=p\Le _K (E_S,C^S)p$ is generated as $*$-algebra by $E^0\cup E^1$ together with the projections $d_Xd_X^*$, for $X\in \Cfin\setminus S$, with the Cohn path algebra relations amongst the elements of $E^0\cup E^1$, the Leavitt path algebra relations $v=\sum _{e\in X} ee^*$ for all $X\in S$, and the additional relations that $d_Xd_X^*$ is a projection, and that
	$$v = \sum _{e\in X} ee^* +d_Xd_X^*$$
	for each $X\in \Cfin_v\setminus S$, $v\in E^0$. 
	Note that these relations automatically imply that $e^*d_Xd_X^* = 0=d_Xd_X^* e$ for all $e\in X\in \Cfin\setminus S$ and that $d_Xd_X^*\le v$ whenever $X\in \Cfin_v \setminus S$. 
	Note that $\varphi (q_X) = d_Xd_X^*$ for $X\in \Cfin\setminus S$.
	
	We can define a map $\psi \colon p\Le_K(E_S,C^S)p\to CL_K(E,C,S)$ by 
	$\psi (v)= v$, $\psi (e)= e$ for $v\in E^0$ and $e\in E^1$, and $\psi (d_Xd_X^*)= q_X = v-\sum_{e\in X} ee^*$ for $X\in \Cfin_v\setminus S$. Using the above observations, we see that this map preserves all the relations amongst the generators of $p\Le_K(E_S,C^S)p$, and hence induces a $*$-homomorphism which is clearly the inverse of $\varphi$. 
	This completes the proof. 
\end{proof}
	
	We now extend the above result to the tame algebras.
	
\begin{theorem}
	\label{thm:relationLPACPA-TAME}
	Let $(E,C)$ be a separated graph and $S$ a subset of $\Cfin$. Let $p:= \sum_{v\in E^0} v\in \mathcal M (\Le_K (E_S,C^S))$. Then we have a natural $*$-isomorphism
$$\ol{\varphi} \colon CL_K^\ab (E,C,S) \cong p\Le _K^\ab (E_S,C^S)p$$
such that $\ol{\varphi} (v) = v$ for all $v\in  E^0$ and $\ol{\varphi} (e) = e$ for all $e\in E^1$.  
	
	In particular, $\Co_K^\ab (E,C) \cong p \Le_K^\ab (\ol{E},\ol{C}) p$. 	
	\end{theorem}

\begin{proof}
	Observe that the $*$-homomorphism $\varphi$ from Proposition \ref{prop:relationLPACPA} induces a surjective
	$*$-homomorphism $\ol{\varphi} \colon  CL_K^\ab (E,C,S) \onto p\Le _K^\ab (E_S,C^S)p$.
	To show that it is injective, we prove that $\ol{\varphi}$ sends a basis of  $CL_K^\ab (E,C,S)$ to 
	a linearly independent family in  $\Le_K^\ab (E_S,C^S)$. Let $\mathfrak E$ be a $S$-choice function for $(E,C)$. This means that we have a map $\mathfrak E\colon S\to E^1$ such that $\mathfrak E (X) \in X$ for all $X\in S$. We extend $\mathfrak E$ to a choice function $\ol{\mathfrak E} \colon (C^S)^{\text{fin}} \to E_S^1$ by letting $\ol{\mathfrak E} (X) = \mathfrak E (X) \in X$ for each $\ol{X} = X \in S$ and $\ol{\mathfrak E} (\ol{X}) = d_X \in \ol{X} = X\cup \{d_X\}$ for all $X\in \Cfin \setminus S$.  
	
	We consider the basis $\mathcal B := \ol{\mathcal B}_{\mathfrak E} (E,C,S)$ of $CL_K^\ab (E,C,S)$ given in Theorem \ref{thm:full-basis-forCLS} with respect to the $S$-choice function $\mathfrak E$, and the basis
	$\ol{\mathcal B} := \ol{\mathcal B}_{\ol{\mathfrak E}} (E_S,C^S)$ of $\Le _K^\ab (E_S,C^S)$ with respect to the choice function $\ol{\mathfrak E}$ given in Theorem \ref{thm:full-basis-forLab}. If $\e (T)\rtimes g$ is an element of the basis $\mathcal B$, then $g_S\in T$ and $T$ is $\mathfrak E$-reduced relatively to $g_S$. But then, since neither of the edges of $g$ and neither of the edges of $T$ are of the forms $d_X$ or $d_X^{-1}$ for $X\in \Cfin \setminus S$, we see that $g_L = g_S$ and that $T$ is $\ol{\mathfrak E}$-reduced relatively to $g_L$. Here $g_L$ is computed with respect to the separated graph $(E_S,C^S)$, and we remark that $|\ol{X}|\ge 2$ for all $\ol{X}\in C^S\setminus S$. Hence $\e (T)\rtimes g$ belongs to the basis $\ol{\mathcal B}$. Hence $\ol{\varphi}$ restricts to an injective map from the basis $\mathcal B$ of $CL_K^\ab (E,C,S)$ into the basis 
	$\ol{\mathcal B}$ of $\Le_K ^\ab (E_S,C^S)$, as desired.  
	\end{proof}

Let $(E,C)$ be a bipartite separated graph, see  \cites{Ara-Exel:Dynamical_systems, AraLolk, AraWeighted}. This means that we have a decomposition $E^0=E^{0,0}\sqcup E^{0,1}$ such that $s(E^1)= E^{0,0}$ and $r(E^1)= E^{0,1}$.  
We then define the $*$-algebra $\LeV_K^{ab}(E,C) = V \Le_K^{ab}(E,C)V$, where $V= \sum_{v\in E^{0,0}} v \in \mathcal M (\Le_K^\ab (E,C))$, see for instance \cite{AraWeighted}. Similarly, we may consider $\CoV _K^\ab (E,C)= V \Co_K ^\ab (E,C) V$. Note that since, for $w\in E^{0,1}$, we have $w= e^*e$ for any edge $e\in E^1$ such that $r(e)=w$, we always have that  $\LeV_K^{ab}(E,C) $ is a full corner of $\Le_K^\ab (E,C)$, and similarly  $\CoV_K^\ab (E,C) $ is a full corner in $\Co_K^\ab (E,C)$.  

With these definitions, the $*$-isomorphism $\ol{\varphi}$ from Theorem \ref{thm:relationLPACPA-TAME} nicely
restricts to a $*$-isomorphism between the corresponding full corners. For this, note that, when $(E,C)$ is bipartite, the separated graphs $(E_S,C^S)$ associated to a subset $S$ of $\Cfin$ are also bipartite in a natural fashion, namely we set $E_S^0 = E_S^{0,0}\sqcup E_S^{0,1}$, with $E_S^{0,0} = E^{0,0}$ and $E_S^{0,1} = E^{0,1}\sqcup \{ v_X: X\in \Cfin \setminus S \}$.

\begin{corollary}
	\label{cor:lfbipartite} Let $(E,C)$ be a bipartite separated graph, let $S$ be a subset of $\Cfin$,  and let $(E_S,C^S)$ be the bipartite separated graph constructed above. Then we have
	$$V \cdot CL_K^\ab  (E,C,S) \cdot V \cong \LeV _K^\ab (E_S,C^S) .$$
	In particular $ \CoV_K^\ab  (E,C) \cong \LeV _K^\ab (\ol{E}, \ol{C}) .$ 
\end{corollary}

\begin{proof}
	Let $\ol{\varphi} \colon CL_K^\ab  (E,C,S) \to p \Le_K^\ab (E_S,C^S) p$ be the isomorphism constructed in Theorem \ref{thm:relationLPACPA-TAME}. Since $\ol{\varphi} (V) = V$, 
	$$V CL_K^\ab (E,C,S) V   \cong  Vp\Le_K ^\ab (E_S,C^S) pV = V \Le_K^\ab  (E_S,C^S) V = \LeV_K^{ab} (E_S,C^S) $$
	showing the result. 
\end{proof}

\section{A linear basis of the ideal $\mathcal Q$}

Recall that $\mathcal Q$ is the ideal of $\Co_K^\ab (E,C)$ generated by the idempotents $q_X$, for all $X\in \Cfin$, so that $\Le_K^\ab (E,C)\cong \Co_K^\ab (E,C)/\mathcal Q$. Our goal in this section is to find a linear basis of $\mathcal Q$. For non-separated graphs, it is not difficult to find such a basis, see for instance \cite[Proposition 1.5.11]{AAS} and Example \ref{exam:Q-for-nonseparated} below. However the separated case is much more involved. 

We start by recalling some necessary notation and terminology from \cite{ABC25}.

\begin{notation}
	For an element $g= x_1\cdots x_n \in \F$ of length $n\ge 1$, with $x_j\in E^1\cup E^{-1}$, and $1\le i \le n$, we denote by $[g]_i$ the prefix of $g$ of length $i$, that is, $[g]_i := x_1\cdots x_i\in \F$.
\end{notation}

\begin{notation}\cite[Notation 5.10(1),(2)]{ABC25}
	\label{notati:opencompactbasis}
	\begin{enumerate}
		\item For $v\in E^0$, write
		$$\mathcal N (v) = \{ x_1^{-1}x_2^{-1}\cdots x_n^{-1}y \in \FC (v) : n\ge 0, x_1,\dots , x_n,y\in E^1\}.$$
	\item 	For each $T\in \Y_0$, write  
$$\mathcal N (T) : = \{gg' \in \FC : g\in T,\,  g' \in \mathcal N (r(g)),\,  g[g']_1\notin T, \, \text{ and }  T\cup \{gg'\}^{\downarrow}\in \Y_0 \}.$$
\end{enumerate}
\end{notation}

We think of $\mathcal N (T)$ as the set of neighbors of $T$. Now for each $T\in \Y_0$ and each finite subset of $\mathcal N (T)$, we can define idempotents in $\Co_K^\ab (E,C)$, as follows:

\begin{definition}\cite[Definition 4.3]{ABC25}
	\label{def:idempotents-IsetminusF}
	Let $(E,C)$ be a separated graph and let $T\in \Y_0$. 
	\begin{enumerate}
		\item The idempotent $\e (T)$ associated to $T$ is 
		$$\e (T) = \prod _{\lambda\in \max (T)} \lambda \lambda^*.$$ 
		\item Suppose now that $\lambda \in \mathcal N (T)$. Then there is a unique decomposition $\lambda = \lambda_0\lambda_1$, where $\lambda_0\in T$, $\lambda_1= x_1^{-1}\cdots x_r^{-1}y$, $r\ge 0$, $x_1,\dots , x_r,y\in E^1$ and $\lambda_0[x_1^{-1}\cdots x_r^{-1}y]_1\notin T$. Define
		$$\e (\lambda_0\setminus \lambda) = \lambda_0\lambda_0^*-\lambda \lambda^* = \lambda_0 (r(\lambda_0) - \lambda_1\lambda_1^*)\lambda_0^* \in \Co_K^\ab(E,C).$$
		\item For a finite subset $F$ of $\mathcal N (T)$, define the idempotent $\e (T\setminus F)$ as follows: 		
		$$\e (T\setminus F) = \e (T)\cdot \prod_{\lambda_0\lambda_1\in F} \e(\lambda_0\setminus \lambda_0 \lambda_1) = \prod _{\lambda_0\lambda_1\in F} \e (T\setminus \{\lambda_0\lambda_1\}).$$
		\item For $v\in E^0$ and $e\in s^{-1}(v)$ we define
		$$\e (v\setminus e) = v-ee^*.$$
		\item For $v\in E^0$ and $X\in \Cfin_v$, define
		$$q_X = \e (v\setminus X) = \prod_{e\in X} \e (v\setminus e) = v - \sum_{e\in X} ee^* \in \Co_K^\ab(E,C).$$
	\end{enumerate}
\end{definition}

Observe that if $\mu = \mu_0 \mu_1$ with $\mu_1 = x_1^{*}\cdots x_r^{*}$ and $\mu_0$ does not end in $(E^1)^*$, where $x_1,\dots ,x_r\in E^1$, and $X\in \Cfin$, then 
$$\mu q_X \mu^* = \e (\mu_0\setminus \mu_0(\mu_1X)) = \prod_{e\in X} \e (\mu_0\setminus \mu_0(\mu_1e)),$$
where we use the notation introduced in Definition \ref{def:idempotents-IsetminusF}. Intuitively, when considering $\e (\mu_0\setminus \mu_0(\mu_1X))$, we are blocking the exit in the direction of $X$ in the path $\mu = \mu_0\mu_1$. This motivates the following definition:   

\begin{definition}
	\label{def:blocking family}
	Let $(E,C)$ be a separated graph and $T\in \Y_0$. A {\it blocking family} for $T$ consists of a finite subset $F$ of $\mathcal N (T)$, where
	$$F= \bigsqcup_{i=1}^r \gamma_0^i(\gamma_1^i X_i)$$
	where $\gamma_0^i\in T$, and $\gamma_0^i\gamma_1^iX_i = \{ \gamma_0^i\gamma_1^ix : x\in X_i \}$ for $X_i \in \Cfin_{r(\gamma_0^i\gamma_1^i)}$, with $\gamma^i x = (\gamma_0^i)(\gamma_1^i x)$ being the standard form of the element $\gamma^i x\in \mathcal N (T)$, as described in Definition \ref{def:idempotents-IsetminusF}(2).
	
	We say that a blocking family $F$ is {\it prime} if it consists of a single set, that is if $F= \gamma X$ for a single $X\in \Cfin$. Observe that 
	$$\e \Big(T \setminus \big(\bigsqcup_{i=1}^r \gamma_0^i(\gamma_1^i X_i)  \big)\Big) = \prod_{i=1}^r \e (T\setminus \gamma_0^i\gamma_1^iX_i) .$$ 
\end{definition}

Again, intuitively when considering $\e (T\setminus F)$ for a blocking family $F$ for $T$ as above, we are blocking each of the exits in the direction of the sets $X_i$ at the paths $\gamma_0^i\gamma_1^i$, for $i=1,\dots , r$. Note that, if the $C$-separated word $\gamma^i_0\gamma_1^i$ ends in $y^{-1}$, with $y\in E^1$, then $y\notin X_i$, because all the paths $\gamma_0^i \gamma_1^i x$, for $x\in X_i$ must be $C$-separated.  

We will use throughout the following expression of the elements $\e(T\setminus gX)$. 
For $T\in \Y_0$ and $F=gX$ a prime blocking family for $T$, write $g=g_0w$, where $g_0\in T$ does not end in $E^{-1}$ and $w$ is a product of elements $x_i^{-1}\in E^{-1}$. Then the element $\e (T\setminus gX)$ can be written as follows:
\begin{equation}
	\label{eq:e-of-I-setminus-gX}
	\e(T\setminus gX) = \e(T) - \sum_{x\in X} \e (T\cup \{gx\}^{\downarrow}).
\end{equation}
Indeed we have
\begin{align*}
	\e(T\setminus gX) & = \e(T)g_0(r(g_0)- \sum_{x\in X} wxx^*w^*)g_0^* = \e(T)(g_0g_0^*) - \sum_{x\in X} \e (T\cup \{gx\}^{\downarrow}) \\ & = \e(T) - \sum_{x\in X} \e (T\cup \{gx\}^{\downarrow})
\end{align*}
because, since $g_0 \in T$, we have $\e (T)(g_0g_0^*)= \e(T)$. This establishes the formula \eqref{eq:e-of-I-setminus-gX}.

The following basic lemma will be used in the proof of Lemma \ref{lem:keyforbasis-of-L}.

\begin{lemma}
	\label{lem:key-for-keyforbasis-of-L}
	Let $(E,C)$ be a separated graph, let $I,J\in \Y _0$ and let $F\in \mathcal N (I)$, $G\in \mathcal N (J)$. Then we have
	$$\e(I\setminus F) \e(J\setminus G)= \begin{cases}
		\e((I\cup J)\setminus (F\cup G)) & \text{ if } I\cup J \text{ is $C$-compatible  and }\\ & \quad F\cap J = G\cap I = \emptyset \\
		\qquad \quad  0 & \text{ otherwise }
	\end{cases}
	$$
\end{lemma}

\begin{proof} This follows from the identity
	$$\mathcal Z(I\setminus F) \cap \mathcal Z(J\setminus G)= \begin{cases}
		\mathcal Z ((I\cup J)\setminus (F\cup G)) & \text{ if } I\cup J \text{ is $C$-compatible  and }\\ 
		&\quad F\cap J = G\cap I = \emptyset \\
		\qquad \quad  \emptyset & \text{ otherwise }
	\end{cases}
	$$
	which has been shown in \cite[Lemma 5.11]{ABC25}.
\end{proof}

In the following lemma, we obtain a generating family for the ideal $\mathcal Q$. 

\begin{lemma}
	\label{lem:keyforbasis-of-L}
	Let $(E,C)$ be a separated graph. Consider the set
	$$\mathcal B'' = \{ \e (T\setminus F) \rtimes g : T\in \Y _0, [T]\in D_g \text{ and } F=\gamma X \text{ is a prime blocking family for } T\}. $$
	Then $\mathcal Q$ is the $K$-linear span of $\mathcal B''$. 
\end{lemma}

\begin{proof}
	It is enough to prove that for $\e(I) \rtimes g, \e (J)\rtimes h \in \IS(E,C)$, the element 
	$$(\e (I)\rtimes g)q_X(\e (J) \rtimes h)$$ is in the $K$-span of $\mathcal B''$.
	We may assume that this product is nonzero. In this case we have $r(g) = v$, where $X\in C_v$, and thus we have
	$$ (\e (I)\rtimes g) q_X= \e(I\setminus gX) \rtimes g.$$
	Now writing $((gg^{-1})\rtimes g)(\e (J)\rtimes h) = \e (J')\rtimes (g\cdot h)$ for some $J'\in \Y_0$, we get from Lemma~\ref{lem:key-for-keyforbasis-of-L}:
	$$(\e (I)\rtimes g)q_X(\e (J) \rtimes h) = (\e(I\setminus gX)e(J'))\rtimes (g\cdot h) = \e ((I\cup J')\setminus gX) \rtimes (g\cdot h)\in \mathcal B''.$$
\end{proof}

\begin{remark}
	\label{rmk:noĺinearly-independent}
	The generating family $\mathcal B''$ of $\mathcal Q$ given in Lemma \ref{lem:keyforbasis-of-L} is not linearly independent in general. We show this with an example. Let $(E,C)$ be the Cuntz separated graph, with $|C| >1$, let $a,b$ be two distinct edges of $E$ and set $E^0=\{v\}$. Let $A=\{a\}$ and $B=\{b\}$. Then we have
	$$\e (A\setminus B) = aa^*- aa^* bb^*,\qquad \e (B\setminus A) = bb^*-aa^* bb^*.$$
	Hence we obtain the following non-trivial linear relation for the distinct elements $\e(A\setminus B), \e(B\setminus A), \e(v\setminus A), \e(v\setminus B)$:
	$$\e (A\setminus B) + \e (v\setminus A) = \e(B\setminus A) + \e(v\setminus B).$$
 \end{remark}

On the other hand, we can prove, using our previous results, that the elements $\e (I\setminus F)$, where $I\in \Y_0$ and $F$ is a blocking family for $I$, are all distinct in $\Co_K^\ab (E,C)$:

\begin{lemma}
	\label{lem:blocking-are-different}
	Let $(E,C)$ be a separated graph, $I,J\in \Y_0$, $F$ a blocking family for $I$, and $G$ a blocking family for $J$. If $\e(I\setminus F) = \e (J\setminus G)$, then $I=J$ and $F=G$.  
\end{lemma}

 \begin{proof}
 	Write $\ol{\varphi}(\e (I\setminus F)) = \e (T)$ and $\ol{\varphi}(\e( J\setminus G)) = \e (T')$, where $T,T'\in \Y_L (\ol{E},\ol{C})$, and $\ol{\varphi}$ is the isomorphism established in Theorem \ref{thm:relationLPACPA-TAME}. Here 
 	$$T= I\cup \{\gamma_1d_{X_1},\dots , \gamma_r d_{X_r}\}^\downarrow ,$$
 	where $F= \bigsqcup _{i=1}^r \gamma_i X_i$, and similarly with $T'$. 
 		Since $\e (T)= \e (T')$ we get $T=T'$ by Theorems \ref{thm:semigroup-generated} and \ref{thm:ECLeavittMunntrees}. So it suffices to notice that $I$ and $F$ can be recovered from $T$. Indeed $I = S_0$, where $S$ is the tree given by
 	$$S = T\setminus  \{ \gamma d_X : \gamma d_X \in T\},$$
 	and then $F$ is the blocking family $\{\gamma X : \gamma d_X\in T\}$ for $I=S_0$. 
 	 \end{proof}

We can now apply our previous results to find a linear basis of the ideal $\mathcal Q$.
Somewhat surprisingly, this linear basis is obtained by taking a certain subset of elements of the form $\e (T\setminus F)\rtimes g$, where $F$ is a blocking family for $T$, and $T$ misses certain edges.

In the following key definition, we will use the following terminology and notation. 
For $T\in \Y_0$, $F$ a blocking family for $T$, and $h\in \max (T)$, we say that $h$ is {\it blocked by $F$} if $F$ contains a block of the form $hwX$, where $w$ is a product of inverse edges and $X\in \Cfin$. Correspondingly, $h$ is {\it non-blocked by $F$} if no such block is contained in $F$, so that no exit of $T$ in the direction of $h$ is blocked. The set of $F$-blocked maximal elements of $T$ will be denoted by $\Bmax (T\setminus F)$, and the set of non-$F$-blocked maximal elements of $T$ will be denoted by $\NBmax (T\setminus F)$. Let $\mathfrak E$ be a choice function for $(E,C)$. If $g\in \FC$ and $g_0\in T$, we say that $\e (T\setminus F)$ is {\it $\mathfrak E$-reduced relatively to $g_0$} if all elements in $\NBmax (T\setminus F)\setminus \{g_0\}$ are $\mathfrak E$-reduced. In other words,  $\e (T\setminus F)$ is  $\mathfrak E$-reduced relatively to $g_0$ if all non-$F$-blocked maximal elements of $T$, except possibly $g_0$, are $\mathfrak E$-reduced.    

\begin{definition}
	\label{def:key-basis-for-Q}
	Let $(E,C)$ be a separated graph and let $\mathfrak E \colon \Cfin \to E^1$ be a choice function for $(E,C)$. The {\it canonical $\mathfrak E$-basis of $\mathcal Q$} is the set $\mathcal B (\mathcal Q) $ consisting of all elements of the form $\e(T\setminus F)\rtimes g$, where $T\in \Y_0$, $F$ is a non-empty blocking family for $T$, $g_0\in T$, and $\e (T\setminus F)$ is $\mathfrak E$-reduced relatively to $g_0$. 
\end{definition}

We now show that the set $\mathcal B (\mathcal Q)$ is indeed a linear basis for $\mathcal Q$.

\begin{theorem}
	\label{thm:BQ-basis-ofQ}
	Let $(E,C)$ be a separated graph and $\mathfrak E$ a choice function for $(E,C)$. Then the canonical $\mathfrak E$-basis $\mathcal B (\mathcal Q)$ is a linear basis of $\mathcal Q$.
	\end{theorem}

\begin{proof}
	Observe first that, for $\e (T\setminus F)\rtimes g\in \mathcal B (\mathcal Q)$, we have $\e (T\setminus F)\in \mathcal Q$, because $F$ is a non-empty blocking family for $T$, and thus $\e (T\setminus F)\rtimes g \in \mathcal Q$.

	Now let $(\ol{E},\ol{C})$ be the separated graph associated to $(E,C)$, so that we have a $*$-isomorphism $\ol{\varphi} \colon \Co_K^\ab (E,C) \to p\Le_K^\ab (\ol{E},\ol{C})p$ by Theorem \ref{thm:relationLPACPA-TAME}. Recall that $\ol{E}^0=E^0 \sqcup \{v_X: X\in \Cfin \}$, $\ol{E}^1 = E^1\sqcup \{ d_X: X \in \Cfin \}$, and $\ol{C}_v = \{\ol{X}: X\in \Cfin_v \}\sqcup \Cinf_v$, where $\ol{X}= X \sqcup \{ d_X\}$ for each $X\in \Cfin $ and $\Cinf = \{X\in C : |X|=\infty\}$. We define a choice function 
	$\widetilde{\mathfrak E}\colon \ol{C}^{\text{fin}} \to \ol{E}^1$ for $(\ol{E}, \ol{C})$ by
	$$\widetilde{\mathfrak E}(\ol{X}) = \mathfrak E (X) \in E^1 .$$
	Let $\ol{\mathcal B} := \ol{\mathcal B}_{\widetilde{\mathfrak E}} (\ol{E}, \ol{C})$ be the linear basis of $\Le _K^\ab (\ol{E},\ol{C})$ associated to the choice function $\widetilde{\mathfrak E}$, see Theorem \ref{thm:full-basis-forLab}. The elements of $\ol{\mathcal B}$ are of the form $\e (S)\rtimes g$, where  $S\in D_g^{\ol{L}}$ is $\widetilde{\mathfrak E}$-reduced relatively to $g_{\ol{L}}$, and $g\in \F _{\ol{C}}$. Note that the edges $d_X$ are terminal, in the sense that any path $\gamma \in \F_{\ol{C}}$ such that $s(\gamma )\in E^0$ and ending in $d_X$ for some $X\in \Cfin$ cannot be properly extended to a path in $\F_{\ol{C}}$ . 
	
	Set $\ol{\mathcal Q}:= \ol{\varphi} (\mathcal Q) \subset p\Le_K^{\ab} (\ol{E}, \ol{C}) p$. By the previous observation, any basis element $\e (S)\rtimes  g \in \ol{\mathcal B}\cap p\Le_K^\ab (\ol{E},\ol{C})p$ satisfies that $g\in \FC$, that is, $g$ cannot contain any edge of the form $d_X$, for $X\in \Cfin$. 
	It follows that a basis of $\ol{\mathcal Q}$ is given by all the elements 
	$\e (S)\rtimes g \in \ol{\mathcal B}$ such that $S$ contains some path of the form $hd_X$, for some $h\in \FC$ and some $X\in \Cfin$, and $g\in \FC$. Note that these paths $hd_X$ are necessarily maximal in $S$, since they cannot be properly extended. We denote this basis of $\ol{\mathcal Q}$ by $\mathcal B (\ol{\mathcal Q})$     	
	  
	  In order to complete the proof, we have to check that $\ol{\varphi} (\mathcal B (\mathcal Q)) = \mathcal B (\ol{\mathcal Q})$. Suppose first that $\e (T\setminus F)\rtimes g\in \mathcal B (\mathcal Q)$. Then $T\in \Y_0$, $F= \bigsqcup_{i=1}^r \gamma^0_i(\gamma^1_i X_i) $ is a non-empty
	  blocking family for $T$, in the sense of Definition \ref{def:blocking family}, $g_0\in T$, and all non-$F$-blocked maximal elements of $T$, except possibly $g_0$, are $\mathfrak E$-reduced.  Now observe that 
	 $\ol{\varphi} (\e (T\setminus F)\rtimes g) = \e (S)\rtimes g$,
	  where 
	  $$\max (S) = \NBmax (T\setminus F) \bigsqcup  \{ \gamma_i^0\gamma _i^1 d_{X_i} : i=1,\dots , r\}.$$
	  On the other hand, since $|\ol{X}| \ge 2$ for all $\ol{X}\in \ol{C}^{\text{fin}}$, we have $g_{\ol{L}} = g_0$ for all $g\in \FC$, so that $g_{\ol{L}} \in S$. Since  
	  $$\max (S)\setminus \{g_{\ol{L}}\} = (\NBmax (T) \setminus \{g_0\}) \bigsqcup  \{ \gamma_i^0\gamma _i^1 d_{X_i} : i=1,\dots , r\}$$
	and $\e (T\setminus F)$ is $\mathfrak E$-reduced relatively to $g_0$, it follows that $S$ is $\widetilde{\mathfrak E}$-reduced relatively to $g_{\ol{L}}= g_0$. Moreover $r\ge 1$ because $F$ is non-empty, so we conclude that $\e (S)\rtimes g\in \mathcal B (\ol{\mathcal Q})$. 
	
	Conversely, suppose that $\e (S)\rtimes g \in \mathcal B (\ol{\mathcal Q})$. We can write
	$\max (S) = M_1 \sqcup M_2$, where $M_1$ are the maximal elements of $S$ which end in $E^1$ and $M_2$ are the maximal elements of $S$ which end in $d_X$ for some $X\in \Cfin$. We can write
	$$M_2 = \{ \gamma_i d_{X_i} : i=1,\dots , r\},$$
	where $\gamma_i\in \FC$ and $X_i\in \Cfin$, for $i=1,\dots ,r$. 
	
	Set $S' := S\setminus M_2$, and note that $S'\in \Y (= \Y (E))$, because no path in $S'$ involves the edges $d_X$, for $X\in \Cfin $. Now let 
	$$A:= \max \{ h_0 : h\in S'\}  \quad  \text{ and } \quad T= A^\downarrow .$$
	Then $T\in \Y_0$, and $T\subseteq S'\subseteq S$. We now will define a suitable blocking family $F$ for $T$. Take an element $\gamma_i d_{X_i}\in M_2$. We can write $\gamma_i = \gamma_i^0 \gamma_i^1$ where $\gamma^0_i$ is a maximal prefix of $\gamma_i$  with the property $\gamma^0_i\in T$. Note that $\gamma_i^1$ might be trivial. Define
	$$F:= \bigsqcup_{i=1}^r \gamma_i^0 (\gamma_i^1 X_i).$$
	We claim that $F$ is a blocking family for $T$. The only thing we need to check for this is that 
	$F\cap T= \emptyset$. Suppose that $\gamma_i^0\gamma_i^1 x\in T$ for some $x\in X_i$. If $\gamma_i^1$ is non-trivial, this cannot happen, so we may assume that $\gamma_i^1$ is trivial. In this case, $\gamma _i x \in S$ because $T\subseteq S$, and also $\gamma_i d_{X_i}\in S$, with $x,d_{X_i} \in \ol{X}_i$, which is impossible because $S$ is $\ol{C}$-compatible. Hence we obtain that $F$ is a blocking family for $T$. Moreover $g_0 = g_{\ol{L}} \in S'$ and thus $g_0\in T$ by the  definition of $T$.  
	
	We now check that $\NBmax (T\setminus F) = M_1$. By construction we have $M_1\subseteq \NBmax (T\setminus F)$. Suppose that $\gamma \in \NBmax (T\setminus F)$. In particular, $\gamma \in \max (T)$, and since $T\subseteq S$, we can extend $\gamma $ to a maximal element $h$ of $S$. If $h\in M_1$, then 
	$h\in T$ and thus $\gamma = h\in M_1$ by maximality of $\gamma$ in $T$. If $h\in M_2$ then
	$\gamma \le_p \gamma_i^0 \gamma_i^1 d_{X_i}$ for some $i\in \{ 1,\dots, r\}$.  But since $\gamma\in \max (T)$, it follows that $\gamma = \gamma _i^0$, which implies that $\gamma \in \Bmax (T\setminus F)$, in contradiction with our assumption. We have therefore shown that $\NBmax (T\setminus F) = M_1$.
	Since $S$ is $\widetilde{\mathfrak E}$-reduced relatively to $g_0$, it follows that $ \e (T\setminus F)$ is $\mathfrak E$-reduced relatively to $g_0$. It follows that $\e (T\setminus F)\rtimes g \in \mathcal B (\mathcal Q)$. 
    
    Finally we show that $\ol{\varphi} (\e (T\setminus F)\rtimes g)= \e (S)\rtimes g$. By the first part of the proof we have
    $$\ol{\varphi} (\e (T\setminus F)\rtimes g)= \e (S'')\rtimes g ,$$
    where $\max (S'') = \NBmax (T\setminus F) \bigsqcup \{ \gamma_i d_{X_i}: i=1\dots , r\}$ and since
    $\NBmax (T\setminus F) = M_1$, we get $\max (S'') = M_1\sqcup M_2 = \max (S)$, which shows that
    $S= S''$. Hence  $\ol{\varphi} (\e (T\setminus F)\rtimes g)= \e (S)\rtimes g$, as desired. 
    This completes the proof. 
	 \end{proof}

\section{The spectrum and the tight spectrum of the Leavitt inverse semigroup}
\label{sect:tight-spectrum}

In this section we compute the spectrum and the tight spectrum of the Leavitt inverse semigroup $\LI (E,C)$ of a separated graph $(E,C)$. We show in a direct way that the tight spectrum of $\LI (E,C)$ is naturally homeomorphic to the tight spectrum of the inverse semigroup $\IS (E,C)$. See \cite[Proposition 5.1]{meakin-milan-wang-2021} for the non-separated case. As a consequence, the tight groupoids associated to $\IS (E,C)$ and $\LI (E,C)$ are naturally isomorphic. Motivated by these developments, we will provide in the final part of this section three different pictures of this tight groupoid.

All results in this section are 'Leavitt' versions of the corresponding results in \cite[Section 5]{ABC25}. The proofs are easy adaptations of the proofs offered in \cite{ABC25}, so we will omit them. It is however convenient to have clear statements of all definitions and theorems. In general we will use the same notation as in \cite{ABC25}, but with a decoration of an $L$ in order to distinguish the notions used here from the ones introduced in \cite{ABC25}. 

We start with a `Leavitt' analogue of the set $\mathfrak F$ from \cite[Definition 5.1]{ABC25}. 

\begin{definition}
	\label{def:filters-ultra-tight}
	\begin{enumerate}
		\item For $v\in E^0$, let $\mathfrak F_L(v)$ be the set of all (possibly infinite) non-empty $C$-compatible lower subsets $Z$ of $\FC(v)$ such that each $z\in Z$ can be extended to an element of $Z$ not ending in $E^{-1}$ and not  ending in an edge $e$ such that $\{e\}\in C$. Set $\mathfrak F_L = \bigsqcup_{v\in E^0} \mathfrak F_L(v)$. 
		\item We order $\mathfrak F_L$ by inclusion, and let $\mathfrak U_L$ be the set of all maximal elements of $F_L$.
	\end{enumerate}
\end{definition}

Note that $\mathfrak F_L \subseteq \mathfrak F$ and that $\Y_L$ is the collection of finite trees in $\mathfrak F_L$, that is, 
$$\Y_L = \{ Z\in \mathfrak F_L : |Z| <\infty\}.$$ 

Recall from Theorem \ref{thm:ECLeavittMunntrees} that $\YY^L := \Y/{\sim_L}$ is isomorphic with the semilattice  $\mathcal E (\LI (E,C))$ of idempotents of the Leavitt inverse semigroup $\LI (E,C)$.

\begin{proposition}(cf. \cite[Proposition 5.1]{ABC25})
	\label{prop:characterizing-filt-ultra-L}
	With the above notation we have
	\begin{enumerate}
		\item The set $\mathcal F_L$ of filters of $\YY^L$ is order-isomorphic with 
		$\mathfrak{F}_L$.
		\item The set of ultrafilters of $\YY^L$ is isomorphic to $\mathfrak{U}_L$.
	\end{enumerate}
\end{proposition}

The isomorphism $\iota \colon\mathcal F_L\to \mathfrak F _L$ is provided as follows. For a filter $\eta$ on $\YY^L$, we consider the collection $C$ of all the elements of the form $g_L$, where $[g^\downarrow]\in \eta$, and $g= g_Lw$ is the canonical decomposition of $g$. Let $Z$ be the lower subset of $\FC $ generated by $C$. Then $\iota (\eta) = Z$. 

Recall from \cite[Definition 5.5]{ABC25} the notion of a maximal local configuration. The set of ultrafilters $\mathfrak U$ on $\mathcal E (\IS (E,C))$ was identified in \cite[Theorem 5.7]{ABC25} with the set $\mathfrak U'$ of all non-trivial lower subsets $Z'$ of $\FC (v)$ such that $Z'_g$ is a maximal local configuration for all $g\in Z'$, where $v$ ranges on all vertices of $E$. 

We show here that the same set $\mathfrak U'$ is isomorphic to the set $\mathfrak U_L$ of ultrafilters on $\mathcal E (\LI (E,C))$.

\begin{theorem}
	\label{cor:ultrafilters} Let $\mathfrak U'$ be the set of all non-trivial lower subsets $Z'$ of $\FC (v)$ such that $Z'_g$ is a maximal local configuration for all $g\in Z'$, where $v$ ranges over all vertices of $E$. Then there is a bijection $\mathfrak U'\cong \mathfrak U_L$ between $\mathfrak U '$ and the set of ultrafilters $\mathfrak U_L$.      
\end{theorem}

\begin{proof}
	The proof is an adaptation of the proof of \cite[Theorem 5.7]{ABC25}. The map $\varphi \colon \mathfrak U' \to \mathfrak U_L$ is defined by sending $Z'\in \mathfrak U'$ to $Z:=Z'\setminus S$, where $S$ is the subset of $Z'$ of those paths $\gamma $ such that $\gamma $ ends in an inverse edge or in an edge $e$ such that $\{e\}\in C$, and $\gamma $ cannot be extended to a 
	path $\lambda \in Z'$ such that $\lambda$ ends in an edge $e\in E^1$ such that $e\in X\in C$ with $|X| \ge 2$.

	The inverse of $\varphi$ is the map $\psi \colon \mathfrak U _L \to \mathfrak U '$ given by
	$$\psi (Z) = Z\bigsqcup \{ gw : gw \in \text{NC}(g) \},$$
	where $\text{NC}(g)$ is the set of all paths $gw\in \FC$ such that $w=w_1\cdots w_r$, for some $r\ge 1$,
	each $w_i$ is an inverse edge or an edge such that $\{w_i\}\in C$, and $gw_1\notin Z$. 
	
	With these definitions at hand, the proof follows the steps of the one of \cite[Theorem 5.7]{ABC25}. 
\end{proof}

For any non-empty lower subset $Z$ of $\FC (v)$, $v\in E^0$, we will denote by $Z_L$ the largest element of 
$\mathfrak F_L$ contained in $Z$ (in the order given by inclusion), that is $Z_L = Z\setminus S$, where $S$ is the subset of $Z$ of those paths $\gamma $ ending in an inverse edge or in an edge $e$ such that $\{e\}\in C$, and $\gamma $ cannot be extended to a 
path $\lambda \in Z$ such that $\lambda$ ends in an edge $e\in E^1$ such that $e\in X\in C$ with $|X| \ge 2$.
The set $Z_L$ can also be described as the lower subset of $\FC$ generated by the set of elements of the form $g_L$, where $g$ ranges on $Z$. Observe that this definition extends the one given in Definition \ref{def:T-sub-L}.

We now extend in various ways the notions and notations introduced in \cite[Notation 5.10]{ABC25}. Indeed we include the case of infinite trees, we consider elements from $\Y_L$ (and $\mathfrak F_L$), and we distinguish between finite and infinite elements in the 'neighborhoods' $\mathcal N (Z)$.

\begin{notation}
	\label{notati:opencompactbasis1}
	\begin{enumerate}
		\item For $v\in E^0$, write $\Ninf_L (v)$ for the set of elements $\gamma y \in \FC (v)$ such that $\gamma$ is a product of inverse edges and edges $e\in E^1$ such that $\{e\}\in C$, and $y\in X$ with $X \in \Cinf$. Similarly we denote by $\Nfin_L (v)$ the corresponding set of elements $\gamma y$ such that 
		$y\in X$, where $X\in C$ and $1< |X| <\infty$.  
		\item 	For each $Z\in \mathfrak F _L$, write  
		$$\Ninf_L  (Z) : = \{gg' \in \FC : g\in Z,\,  g' \in \Ninf_L (r(g)),\,  g[g']_1\notin Z, \, \text{ and }  Z\cup \{gg'\}^{\downarrow}\in \mathfrak F_L \}.$$
		\item 	For each $Z\in \mathfrak F _L$, write  
		$$\Nfin_L  (Z) : = \{gg' \in \FC : g\in Z,\,  g' \in \Nfin_L (r(g)),\,  g[g']_1\notin Z, \, \text{ and }  Z\cup \{gg'\}^{\downarrow}\in \mathfrak F_L \}.$$
		We also set $\mathcal N _L(Z) = \Nfin _L (Z) \sqcup \Ninf _L(Z)$. 
		\item For each $T\in \Y_L$ and each finite subset $F$ of 
		$\mathcal N _L (T)$ we set
		$$\mathcal Z_L (T\setminus F) = \{ Z\in \mathfrak F_L :  T\subseteq Z \text{ and } F\cap Z= \emptyset \},$$
		and $\mathcal Z _L(T):= \mathcal Z _L(T\setminus \emptyset) = \{ Z\in \mathfrak F _L :  T\subseteq Z \}$.  
	\end{enumerate}
	Observe that for $T = \{v\}\in \Y_L(v)$, we have $\Ninf_L (T)= \Ninf_L (v)$ and $\Nfin_L (T) = \Nfin_L (v)$, so the notations introduced in (1), (2) and (3) are coherent.
\end{notation}

We now introduce a fundamental definition.

\begin{definition}
	\label{def:exits}
	Let $Z\in  \mathfrak F_L$. An element in $\Ninf_L (Z)$ will be called an {\it infinite exit of} $Z$.
	An element of $\Nfin_L (Z)$ will be called a {\it finite exit of} $Z$. We say that $Z$ {\it does not have finite (respectively, infinite)} exits if $\Nfin_L (Z)= \emptyset$ (respectively, $\Ninf_L (Z)=\emptyset$). We say that $Z$ {\it does not have exits} if $\Ninf_L (Z) = \Nfin_L (Z)= \emptyset$.    
	\end{definition}

The following is the Leavitt version of \cite[Lemma 5.11]{ABC25}.

 \begin{lemma}
	\label{lem:opencompactbasisL}
	With the above notation, the family of sets $\mathcal Z_L (T\setminus F)$, where $T$ ranges on $\Y _L$ and $F$ ranges on all finite subsets of $\mathcal N_L (T)$,
	is a basis of open compact subsets of $\mathfrak F_L$. Each open compact subset of $\mathfrak F_L$ is a finite disjoint union of sets $\mathcal Z_L (T\setminus F)$.  
\end{lemma}

The tight spectrum $\Omega (E,C)$ of the inverse semigroup $\IS (E,C)$ was identified in \cite[Theorem 5.14]{ABC25} with the space $\mathfrak T' \cup  \SInf$, where $\mathfrak T'$ is the set of all non-trivial lower subsets $Z$ of $\FC (v)$ such that $Z_g$ is a finite-maximal local configuration for all $g\in Z$, for $v\in E^0$, and $\SInf$ is the set of infinite sources of $E$ (\cite[Definition 5.13]{ABC25}). An identical proof identifies this space with the tight spectrum of $\LI (E,C)$, as follows.

\begin{theorem}
	\label{thm:new-charac-tight-spectrum}
	Let $(E,C)$ be a separated graph. Then the space $\Omega(E,C)$ is $\F$-equivariantly homeomorphic with the tight spectrum of $\LI (E,C)$. Hence we have equivariant homeomorphisms
	$$\Omega (E,C) = \Etight (\IS(E,C)) \cong \Etight (\LI (E,C)).$$
	Moreover the space $\Omega (E,C)$ can be identified with the space $\mathfrak T_L$ of all those $Z\in \mathfrak F_L$ such that $Z$ does not have finite exits, endowed with the topology generated by a basis of compact open sets consisting of all the sets of the form $\mathcal Z_L (T\setminus F)\cap \mathfrak T_L$, where $T$ ranges on $\Y_L$ and $F$ ranges on all the finite subsets of $\Ninf_L(T)$.    
	\end{theorem}

\begin{remark}
    \label{rem:Ruy-consonant-semigroups}
The notion of consonant inverse semigroups has been recently introduced by Exel \cite{exel2025}. Roughly, two inverse semigroups are consonant if they have the same tight groupoid, see \cite[Theorem 9.5]{exel2025} for the precise statement. Here we notice that the natural surjection $\pi \colon \IS (E,C) \to \LI (E,C)$ is a consonance for every separated graph $(E,C)$. 

Let $S$ be an inverse semigroup and let 
$\Delta \colon S \to  \text{Bis}^{\text{oc}}(\G_\tight(S))$ 
be the natural map defined in \cite{exel2025}, 
where $\text{Bis}^{\text{oc}}(\G_\tight(S))$ is the inverse semigroup of open compact bisections of the tight groupoid $\G_\tight (S)$ of $S$. Then there is a natural injective semigroup homomorphism  $U \mapsto 1_U$ from  
$B:= \text{Bis}^{\text{oc}}(\G_\tight(S))$ to the multiplicative semigroup of the Steinberg algebra $A_K (\G_\tight (S))$.
Hence we may identify $B$ with an inverse subsemigroup of the multiplicative semigroup of $A_K (\G_\tight (S))$. 
In particular $\Delta (S)$ is identified with an inverse subsemigroup of $A_K (\G_\tight (S))$.

Now let $(E,C)$ be a separated graph and let $S := \IS(E,C)$ be the inverse semigroup of the separated graph $(E,C)$. By 
Theorem \ref{thm:structure-of-Leavitt} there is a natural  $*$-isomorphism $\Le^\ab_K (E,C) \cong A_K (\G_\tight (S))$.
Under this isomorphism, $\Delta (S)$ corresponds to the natural image of $S$ in $\Le_K^\ab(E,C)$, which is the subsemigroup of 
$\Le_K^\ab(E,C)$ generated by $E^0 \cup E^1 \cup  (E^1)^*$. But now by Theorem \ref{thm:semigroup-generated}, 
this semigroup is isomorphic to $\LI (E,C)$. Hence there is an isomorphism $\ol{\Delta}\colon  \LI(E,C) \to \Delta (S)$ such that $\ol{\Delta} \circ \pi = \Delta$.
Since, by \cite[Proposition 8.11]{exel2025}, $\Delta \colon  S \to  \Delta (S)$ is a consonance, so is $\pi \colon  \IS(E,C) \to \LI(E,C)$.
\end{remark}

We will refer to the concrete description of $\Omega (E,C)$ given in Theorem \ref{thm:new-charac-tight-spectrum} as the {\it Leavitt model of} $\Omega(E,C)$.
As an immediate consequence of Theorem \ref{thm:new-charac-tight-spectrum}, we obtain:

\begin{theorem}
	\label{thm:new-groupoid-model}
		Let $(E,C)$ be a separated graph. Then there is an isomorphism of groupoids
		$$\mathcal G _\tight (E,C) := \mathcal G_\tight (\IS(E,C)) \cong \mathfrak T_L \rtimes \F.$$
		Moreover 
		$$\Le_K^\ab (E,C) \cong A_K ( \mathfrak T_L \rtimes \F) \quad \text{ and } \quad \mathcal O  (E,C) \cong C^* (\mathfrak T_L\rtimes \F).$$
\end{theorem} 

\begin{proof}
	This follows from Theorem \ref{thm:new-charac-tight-spectrum}, \cite[Proposition 6.8]{ABC25} and \cite[Theorem 6.20]{ABC25}.
\end{proof}

A corresponding model for the tight groupoid $\mathcal G _\tight (E,C)$ is in order. Identifying 
$\Omega (E,C)$ with the space  
$$\mathfrak T_L = \{ Z\in \mathfrak F_L : Z \text{ does not admit finite exits }\} \, ,$$
we have the picture 
$$\G_\tight (E,C) \cong \{ (Z',g,Z) : Z,Z'\in \mathfrak T_L, g\in \FC, (g^{-1})_L\in Z \text{ and }  Z'= [g\cdot (Z\cup \{g^{-1}\}^\downarrow )]_L\}.$$
Observe that $g \cdot (Z\cup \{g^{-1}\}^\downarrow) = g\cdot Z \cup g^\downarrow$. 

The inverse of $(Z',g,Z)$ is given by 
$$(Z',g,Z)^{-1} = (Z, g^{-1}, Z').$$
The product of $(Z_1',g_1,Z_1)$ and $(Z_2',g_2,Z_2)$ is defined if and only if $Z_1= Z_2'$, and in this case the product is given by
$$(Z_1',g_1,Z_1)\cdot (Z_2',g_2,Z_2) = (Z_1', g_1\cdot g_2, Z_2).$$
 The topology on $\G_\tight (E,C)$ is generated by the cylinders $\mathcal Z (T',g,T)$, with $T,T'\in \Y_L$,
 $(g^{-1})_L\in T$, and $T' = [g\cdot (T\cup \{g^{-1}\}^\downarrow )]_L$, where
 $$\mathcal Z (T',g,T) = \{(Z',g,Z) \in \G_\tight (E,C) : T\subseteq Z\}.$$  

It is instructive to check that the definitions given above make sense, as follows. 

\begin{lemma}
\label{lem:well-defined-inverse-product}
The inverse and product in $\G_\tight (E,C)$ given above are well defined, and correspond to the groupoid   structure on the tight groupoid of $(E,C)$.
	\end{lemma}
	
	   \begin{proof}
	   	We first check that the inverse is well-defined, that is, for $(Z',g,Z)\in \G_\tight (E,C)$, we have
	   	that $(Z,g^{-1},Z')\in \G_\tight (E,C)$. By definition, $Z\in \mathfrak F_L$, $(g^{-1})_L\in Z$, and
	   	$Z'= [g\cdot Z \cup g^\downarrow]_L$. Hence we see that $g_L\in Z'$. 
	   	
	   	We claim that $Z'\cup g^\downarrow = g\cdot Z$. Indeed, it is clear that $Z'\cup g^\downarrow =
	   	(g\cdot Z \cup g^\downarrow)_L \cup g^\downarrow $ is contained in $g\cdot Z$. Now let $z\in Z$. We can write $g=x_1\cdots x_r x_{r+1}\cdots x_n$, $z= x_n^{-1}\cdots x_{r+1}^{-1}y_1\cdots y_m$ in reduced form, where $m\ge 0$ and $y_1\ne x_r$ if $m\ge 1$. If $m=0$ then $g\cdot z = x_1\cdots x_r \in g^\downarrow \subseteq  [g\cdot Z \cup g^\downarrow]_L \cup g^\downarrow $. If $m\ge 1$, then 
	   	$g\cdot z = x_1\cdots x_r y_1\cdots y_m$, and  since $Z\in \mathfrak F_L$, there exists $y_{m+1},\dots ,y_{m+k}\in E^1\cup E^{-1}$, with $k\ge 0$, such that 
	   	$$x_n^{-1} \cdots x_{r+1}^{-1}y_1\cdots y_my_{m+1} \cdots y_{m+k} \in Z \quad \text{ and } \quad y_{m+k}\in X\in C^{>1},$$
	   	and it follows that $x_1\cdots x_r y_1\cdots y_{m+k} \in (g\cdot Z)_L$ and thus that $g\cdot z \in (g\cdot Z)_L$, because $(g\cdot Z)_L$ is a lower set. Therefore $g\cdot z \in (g\cdot Z \cup g^\downarrow)_L \cup g^\downarrow $ also in this case. 
	   	
	   	 We now check that the product is well-defined. Let $(Z_1',g_1,Z_1), (Z_2',g_2,Z_2)$ be two triples such that $Z_1= Z_2'$. We know that $(g_1^{-1})_L \in Z_1$ and that $(g_2)_L\in Z_2'= Z_1$. Hence assuming that $Z_1\in \mathfrak F_L(v)$, we get
	   	 $$s(g_2) = s((g_2)_L)= v = s((g_1^{-1})_L) = r(g_1) ,$$
	   	 so that $s(g_2)= r(g_1)$. Hence $g_1\cdot g_2$ is at least in $\FG (E)$. Hence, to show that 
	   	 $g_1\cdot g_2 $ belongs to $\FC$, it suffices to show that $g_2^{-1}\cdot g_1^{-1}$ belongs to $\FC$, which in turn is equivalent to check that $(g_2^{-1}\cdot g_1^{-1})_L \in \FC$.  
	   	 
	   	We will check that  $(g_2^{-1}\cdot g_1^{-1})_L \in Z_2$, which in particular implies that  $(g_2^{-1}\cdot g_1^{-1})_L \in \FC$. Write $g_2 = x_1\cdots x_n$ and $g_1 = y_1\cdots y_t x_r^{-1} \cdots x_1^{-1}$ in reduced form, with $y_t\ne x_r$, $y_t\ne x_{r+1}^{-1}$, $0\le r\le n$, and $x_i,y_j\in E^1\cup E^{-1}$. We then have
	   	$g_1\cdot g_2 = y_1\cdots y_t x_{r+1} \cdots x_n$, and 
	   	$$(g_1\cdot g_2)^{-1} = g_2^{-1}\cdot g_1^{-1} = x_n^{-1} \cdots x_{r+1}^{-1} y_t^{-1} \cdots y_1^{-1}.$$   
	   	 
	   	 We distinguish two cases:
	   	 \begin{enumerate}
	   	 	\item $(g_2^{-1}\cdot g_1^{-1})_L = x_n^{-1}\cdots x_{r+1}^{-1}y_t^{-1}\cdots y_i^{-1}$, with $i \le t$.
	   	 	\item $(g_2^{-1} \cdot g_1^{-1})_L = x_n^{-1}\cdots x_j^{-1}$, with $j\ge r+1$. 
	   	 \end{enumerate}
	   	 
	   	 If (2) happens, then $(g_2^{-1}\cdot g_1^{-1})_L \le_p (g_2^{-1})_L \in Z_2$ and thus $(g_2^{-1}\cdot g_1^{-1})_L\in Z_2$. If (1) holds, then $(g_1^{-1})_L = x_1\cdots x_r y_t^{-1}\cdots y_i^{-1}$ and since
	   	 $$(g_1^{-1})_L \in Z_1 = Z_2' = g_2\cdot Z_2 \cup g_2^\downarrow ,$$
	   	 we get 
	   	 $$(g_2^{-1}\cdot g_1^{-1})_L =  x_n^{-1}\cdots x_{r+1}^{-1}y_t^{-1}\cdots y_i^{-1} = g_2^{-1}\cdot (g_1^{-1})_L \in g_2^{-1}\cdot (g_2\cdot Z_2 \cup g_2^\downarrow ) = Z_2\cup (g_2^{-1})^\downarrow .$$
	   	  If $(g_2^{-1}\cdot g_1^{-1})_L\in Z_2$, we are done. Otherwise we have $(g_2^{-1}\cdot g_1^{-1})_L\in (g_2^{-1})^\downarrow$, which is impossible since $x_{r}\ne y_t$.

	   	Now we check that $[(g_1\cdot g_2)\cdot Z_2 \cup (g_1\cdot g_2)^\downarrow ]_L = Z_1'$.  Indeed we have
	   	\begin{align*}
	   		g_1\cdot Z_1 \cup g_1^\downarrow & = g_1\cdot (g_2\cdot Z_2 \cup g_2^\downarrow) \cup g_1^\downarrow \\
	   		& = (g_1\cdot g_2) \cdot Z_2 \cup g_1\cdot g_2^\downarrow \cup g_1^\downarrow \\ 
	   		& = (g_1\cdot g_2) \cdot Z_2 \cup (g_1\cdot g_2)^\downarrow \cup g_1^\downarrow .
	   		   	\end{align*}
	   	  	   	Hence $Z_1'= [g_1\cdot Z_1 \cup g_1^\downarrow]_L = [(g_1\cdot g_2) \cdot Z_2 \cup (g_1\cdot g_2)^\downarrow \cup g_1^\downarrow]_L$, and it suffices to show that $(g_1)_L \subseteq  
	   	  	   	(g_1\cdot g_2) \cdot Z_2 \cup (g_1\cdot g_2)^\downarrow$. Suppose first that $(g_1)_L= y_1\cdots y_l$ for some $l\le t$. Then we have
	   	  	   	$$(g_1)_L = y_1\cdots y_l \le_p (g_1\cdot g_2)_L .$$
	   	  	   	Assume now that $(g_1)_L = y_1\cdots y_t x_r^{-1}\cdots x_s^{-1}$ for some $s\le r$.  Then $x_n^{-1} \cdots x_{r+1}^{-1}x_r^{-1}\cdots x_s^{-1} = (g_2^{-1})_L$, and hence
	   	  	   	$$(g_1)_L = (y_1\cdots y_tx_{r+1}\cdots x_n)\cdot (x_n^{-1}\cdots x_s^{-1}) = (g_1\cdot g_2)\cdot (g_2^{-1})_L \in (g_1\cdot g_2)\cdot Z_2.$$
	   	  	   	Hence $(g_1)_L\in (g_1\cdot g_2)\cdot Z_2$ in this case.

	   	  	   	Finally, it is easily seen that this structure agrees with the groupoid structure of $\mathfrak T_L\rtimes \F$, which by Theorem \ref{thm:new-groupoid-model} is isomorphic to $\mathcal G_\tight (E,C)$. This concludes the proof. 
	   	   	   \end{proof}

Hence we have three different pictures or models of the tight groupoid $\G_\tight (E,C)$:

$\bullet $ The {\it complete model} is the model in which the elements of $\Omega (E,C)$ are the {\it configurations}, that is, the lower subsets $Z$ such that all the local configurations $Z_g$ are finite-maximal, together with the set $\SInf$ of infinite sources of $E$, see \cite[Definition 5.13 and Theorem 5.14]{ABC25}. This is the model considered in \cite{Lolk:tame} in the case of finitely separated graphs, for instance. The groupoid model in this picture is slightly simpler: The elements of the groupoid are of the form $(Z',g,Z)$, where $Z,Z'\in \mathfrak T'\cup  \SInf $, $g^{-1}\in Z$, and $Z'= g\cdot Z$. The inverse of $(Z',g,Z)$ is $(Z,g^{-1},Z')$, and $(Z_1,g,Z_2)(Z_2,h,Z_3)= (Z_1,g\cdot h ,Z_3)$. 

$\bullet$ The {\it standard model} is the model introduced in \cite[Section 5]{ABC25}. It is based on the description of the elements of $\IS (E,C)$ as elements in $\Y_0$. We call it the standard model because it is the model that is commonly used to describe space $\Omega (E,C)$ in the case of a non-separated graph.
The formulas for the product and the inverse are similar to the ones corresponding to the Leavitt model, replacing the decoration $L$ with the decoration $0$. The topology is generated by corresponding cylinder sets. 

$\bullet$ The {\it Leavitt model} is the model we have introduced here. It is based on the description of the elements of $\LI (E,C)$ as elements in $\Y_L$.

\medskip

We believe that all three  models have advantages and disadvantages. For a specific problem, it might be that one of the three models is better than the others. Note that the standard and the Leavitt model agree when $|X|\ge 2$ for all $X\in C$. We will see in the next section that the Leavitt model is specially suited to describe the socle of the tame Leavitt path algebras of separated graphs.

\section{Isolated points of a groupoid. The socle.}
\label{sect:socle}

In this section, we use recent results on the socle of a Steinberg algebra \cite{CCMMR25} to compute the socle of the tame Leavitt algebra $\Le^\ab _K (E,C)$ of a separated graph. 

{\bf In this section, $K$ will denote a field, and an algebra will mean an algebra over $K$.}

We first collect some basic definitions. 

Recall that a point $x$ of a topological space $X$ is said to be an {\it isolated point} if $\{ x\}$ is an open subset of $X$. Note that if $X$ is Hausdorff (which is the only case that will interest us here) then points are closed and so $\{x\}$ is a clopen subset of $X$ for each isolated point $x$ of $X$. We denote the set of isolated points of $X$ by $\isoX$. 

Let $\G$ be an ample locally compact Hausdorff groupoid. We set $X:= \G^{(0)}$, the unit space of $\G$. 
Observe that for $\gamma \in \G$, we have $\gamma \in \isoG$ if and only if $s(\gamma)\in \isoX$ if and only if $r(\gamma) \in \isoX$. We define an equivalence relation $\mathcal R$ on $X$ by $x\mathcal R y$ if and only if there exists $\gamma \in \G$ such that $s(\gamma) = x$ and $r(\gamma) = y$. That is, two points in $X$ are related if and only if they belong to the same orbit. Observe that all elements in the orbit of an isolated point are also isolated points. 

If $x\in \isoX$, then $\{x\}$ is a compact open subset of $\G$, so that $1_{\{x\}}$ is a projection in $A_K(\G)$.

For $x\in \isoX$, let $\Lambda_x$ be a transversal of $\G_x/\G^x_x$ containing $x$, that is $\Lambda $ contains exactly one representative of each left coset $\gamma \G^x_x$, where $\gamma \in \G_x$. Then we have the following result. For $\gamma \in \isoG$, we will denote $1_{\{\gamma\}}$ just by $1_\gamma$.

\begin{theorem}
\label{thm:structure-isolated}
Let $x\in \isoX$. Then we have
\begin{enumerate}
	\item $$1_x A_K(\G) 1_x \cong K \G^x_x,$$ 
	that is, the corner $*$-algebra $1_x A_K(\G) 1_x$ is $*$-isomorphic to the group $*$-algebra $K\G^x_x$ of the isotropy group $\G_x^x$. 
	\item $$A_K(\G) 1_x A_K(\G) \cong M_{\Lambda_x} (K\G_x^x),$$ 
	that is, the two-sided ideal
	$A_K(\G) 1_x A_K(\G)$ generated by $1_x$ is $*$-isomorphic to the $*$-algebra of $\Lambda_x \times \Lambda_x$ matrices with coefficients in $K\G^x_x$.  
	\end{enumerate}
\end{theorem}

\begin{proof}
	(1) follows from \cite[Proposition 2.3]{ADN22}. We give a direct proof of (2) here. First note that $A_K(\G) 1_x A_K(\G)$ is linearly spanned by the characteristic functions $1_\gamma$, where $\gamma \in \mathcal G _y^z$, for $y,z$ in the orbit $r(\mathcal G_x)$ of $x$ in $X$.
	
	Given such $\gamma \in \mathcal G _y^z$ there are unique $\lambda_1,\lambda_2\in \Lambda_x$ such that $r(\lambda_1)= y$ and $r(\lambda_2)= z$. Then the element $1_\gamma$ is sent to the matrix whose only non-zero component is in the $(\lambda_2,\lambda_1)$ spot, and this component is $\lambda_2^{-1}\gamma \lambda_1\in \mathcal G_x^x$. It is now straightforward to show that this assignment defines an isomorphism between
 $A_K(\G) 1_x A_K(\G)$ and $M_{\Lambda_x} (K\mathcal G _x^x)$. 
	\end{proof}

We now collect some useful information on isotropy. For a partial action $(U_t, \theta_t) $ of a discrete group $\Gamma $ on a topological space $X$, and $z\in X$, we denote by $F_z$ the isotropy of $z$, that is the set of elements $t \in \Gamma$ such that $z\in U_{t^{-1}}$ and
$\theta_{t} (z)=z$. It is a general well-known fact that $F_z$ is a subgroup of $\Gamma$. Moreover if $\G = X\rtimes \Gamma$ is the transformation groupoid associated to the partial action, then $\G_z^z \cong  F_z$, hence the isotropy groups $\G_z^z$ are isomorphic to subgroups of $\Gamma$.   

\begin{definition}\cite[Definition 1.9]{Lolk:tame}
	\label{def:cycles}
	A non-trivial $C$-separated path $\alpha$ in a separated graph $(E,C)$ is called a {\it closed path} if $r(\alpha ) =s(\alpha)$, and it is called a {\it cycle} if $\alpha^2 := \alpha \alpha$ is also a $C$-separated path. In this case $\alpha^n$ is $C$-separated for each $n\in \Z$.\footnote{Here we interpret $\alpha^0$ as the trivial path at the vertex $s(\alpha)=r(\alpha)$, and for $n<0$ we set $\alpha^n := (\alpha^{-1})^{-n}$. In particular, if $\alpha \alpha$ is $C$-separated, then so are $\alpha^{-1}$ and all integer powers $\alpha^n$.}
 A {\it simple closed path} is a closed path $\alpha$ such that the only vertex repetition occurs at the end, that is, $\alpha_1 <_p \alpha_2 \le_p \alpha$ and $r(\alpha _1)= r(\alpha_2)$ implies that $\alpha_1 = s(\alpha)$ and  $\alpha_2 = \alpha$.   
\end{definition}

\begin{lemma}
	\label{lem:trivial-properties-isotropy}
	With the above notation, let $Z\in \Omega (E,C)_v$, where $v\in E^0$. Then the set 
	$$F_Z= \{ \alpha \in \F : Z\in D_{\alpha^{-1}} \text{ and } \theta_\alpha (Z)=Z\}$$ is a subgroup of $\F$, consisting of elements that belong to $\FC(v)$. 
	If $\alpha $ is a non-trivial element of $F_Z$, then $\alpha = \beta \gamma \beta^{-1}$, where $\gamma $ is a cycle and $\gamma \in F_{\theta_{\beta^{-1}}(Z)}$.
	\end{lemma}

\begin{proof}
	As mentioned above, $F_Z$ is a subgroup of $\F$. Observe that $\gamma \in \FC (v)$ for all $\gamma \in F_Z^*$. If $\alpha $ is a non-trivial element of $F_Z$, then we can uniquely write $\alpha = \beta \gamma \beta^{-1}$, where $\gamma = \gamma _1 \cdots \gamma_r$ is a closed path with $\gamma _r \ne \gamma_1 ^{-1}$. Since $\alpha \cdot \alpha = \beta \gamma \gamma \beta^{-1}  \in \FC$, we see that $\gamma $ is a cycle. Clearly $\gamma \in F_{\theta_{\beta^{-1}}(Z)}$.
	\end{proof}

For a ring with local units $A$, the {\it left socle} of $A$ is the left ideal generated by all the minimal left ideals of $A$. The left socle is a two-sided ideal of $A$. A similar definition gives the right socle of $A$. If $A$ is {\it semiprime}, meaning that $A$ does not contain nonzero nilpotent ideals, then the left and the right socle coincide, and each minimal left (or right) ideal is generated by an idempotent, see \cite[Section IV.3]{jac}. 

The left socle also coincides with the right socle in case the algebra $A$ admits an involution. In that case, we simply will name it the socle, denoted by $\text{soc} (A)$. Note that this is always the case when $A=A_K(\G)$, the Steinberg algebra of an ample groupoid $\G$.  

We can now state a basic property of the algebra $\Le^\ab_K (E,C)$. We will use throughout that 
$\Le^\ab_K (E,C)$ is isomorphic to the crossed product $C_K(\Omega(E,C))\rtimes \F \cong A_K(\G_\tight (E,C))$. This follows from Theorem \ref{thm:structure-of-Leavitt} and the fact that
$\widehat{\E}_\tight \cong \Omega(E,C)$ (Theorem \ref{thm:new-charac-tight-spectrum}).  

\begin{proposition}
	\label{prop:Leavitt-is-semiprime}
	Let $(E,C)$ be a separated graph. Then the algebra $\Le^\ab_K (E,C)$ is semiprime.
\end{proposition}
 
 \begin{proof}
 	By \cite[Theorem 4.10]{Steinberg2019}, $\Le_K^\ab (E,C) \cong A_K (\G_\tight (E,C))$ is semiprime if $K\G_x^x$ is semiprime for each $x\in \G^{(0)}$. Since in our case $\G_x^x$ is always a free group, the result follows. 
 	\end{proof}
 
 \begin{definition}\cite{CCMMR25}
 	\label{def:LP}
 	Let $\G$ be an ample Hausdorff groupoid. We say that $\G$ satisfies condition (LP) if, for every $x\in \G^{(0)}$,
 	$$\text{soc} (K(\G^x_x))\ne 0 \iff \G_x^x = \{x\}.$$
 	 \end{definition}

\begin{theorem}\cite{CCMMR25}
\label{thm:mainCGMMR}
Let $\G$ be an ample Hausdorff groupoid satisfying condition {\rm(LP)}, with $X=\G^{(0)}$. The socle $\mathrm{soc}(A_K(\G))$ of $A_K(\G)$ coincides with the ideal generated by all the characteristic functions $1_{x}$, where $x$ is an isolated point in $X= \G^{(0)}$ such that $\G_x^x= \{x\}$. Moreover 
$$\mathrm{soc} (A_K (\G)) \cong \bigoplus_{[x]\in \isoisoX/{\mathcal R}} M_{|[x]|} (K) ,$$
	where $\isoisoX$ is the set of elements $x\in \isoX$ such that $\G_x^x= \{x\}$, so that $\isoisoX/{\mathcal R}$ is the set of orbits of $\isoisoX$.  
\end{theorem}

In the following, we will use the Leavitt model for the space $\Omega (E,C)$ and the canonical action of $\F = \F(E^1)$ on it, see Section \ref{sect:tight-spectrum}. 
Recall that the notion of a Leavitt $(E,C)$-tree has been introduced in Definition \ref{def:Leavitt-Munn-tree}(a). 

\begin{notation}
	\label{not:isoOmega}
	Let $(E,C)$ be a separated graph. We denote by $\isoOmega$ the set of Leavitt $(E,C)$-trees which do not admit an exit, and we denote by $\isoisoOmega$ the subset of $\isoOmega$ of those Leavitt $(E,C)$-trees $T$ such that there is no closed path $\alpha = \beta \gamma \beta^{-1}$, where $\gamma $ is a cycle, $(\alpha^{-1})_L \in T$ and $(\alpha\cdot T \cup \alpha^\downarrow)_L = T$. We will write $\isoOmegaEC$ and $\isoisoOmegaEC$ in case we need to specify the separated graph $(E,C)$.  
\end{notation}

We now characterize the set of isolated points in $\Omega (E,C)$, where $(E,C)$ is a separated graph. 

\begin{proposition}
	\label{prop:isoand-isoiso-forOmega}
	Let $(E,C)$ be a separated graph. Then the set of isolated points of $\Omega (E,C)$ coincides with the set $\isoOmega$ of all (finite) Leavitt $(E,C)$-trees which do not admit exits. Moreover the set $\isoisoOmega$ is the set of isolated points $z\in \Omega (E,C)$ with trivial isotropy. 
		\end{proposition}

\begin{proof}
	Suppose that $Z$ is an isolated point in $\Omega (E,C) = \mathfrak T _L$. Then by Theorem \ref{thm:new-charac-tight-spectrum}, $Z$ is an element in $\mathfrak F_L$ with no finite exits.  Moreover, there is an open set of the form $V:=\mathcal Z_L (T\setminus F) \cap \mathfrak T_L$, where $T\in \Y_L$ (that is, $T$ is a Leavitt $(E,C)$-tree), and $F$ a finite subset of $\Ninf_L (T)$, such that $V= \{Z\}$. But if $T$ admits an exit, then we can build at least two different elements $Z_1,Z_2\in \mathcal Z_L (T\setminus F)$, which then can be extended to two different elements $Z_1',Z_2'\in \mathcal Z_L (T\setminus F) \cap \mathfrak T _L$, and we get a contradiction. It follows that $T$ has no exits, and then again by Theorem \ref{thm:new-charac-tight-spectrum}, $T\in \Omega (E,C)$. Moreover $F=\emptyset$, because $T$ has no infinite exits. It follows that $Z=T$ is a Leavitt $(E,C)$-tree without exits. 
	
	Conversely, if $T$ is a Leavitt $(E,C)$-tree without exits, then $T\in \Omega (E,C)$ and $\mathcal Z_L (T) = \{T\}$, and so $T$ is an isolated point in $\Omega (E,C)$. 
	
	The last statement follows from Lemma \ref{lem:trivial-properties-isotropy}. 
		\end{proof}

The condition that $T\in \isoOmega$ imposes serious restrictions, and $\isoOmega = \emptyset $ for many separated graphs. 

\begin{lemma}
	\label{lem:isotropy-when-all-groups-havemoreon-element}
	Let $(E,C)$ be a separated graph such that $|X| \ge 2$ for all $X\in C$, and suppose that $T$ is a Leavitt $(E,C)$-tree without exits. Suppose that $h\in \FC$ is such that $T\cup h^\downarrow \in \Y$ and $h\notin T$, and write $h= gg'$, where  $g\in T$, $g[g'] _1\notin T$. Then all the edges of $g'$ are inverse edges $e^{-1}\in E^{-1}$ with the property that $|C_{s(e)}| = 1$.
	Moreover, $T\in \isoisoOmega$ and thus $\isoOmega = \isoisoOmega$. 
\end{lemma}

\begin{proof}
	If some of the edges of $g'$ is positively oriented, then we would obtain an exit of $T$, because $|X|\ge 2$ for all $X\in C$. Hence all edges of $g'$ are inverse edges. Let $e^{-1}$ be one of these edges, and suppose that $e\in X\in  C_{s(e)}$. We then have $gg''e^{-1} \le_p h$ for some initial segment $g''$ of $g'$. If there is $Y\in C_{s(e)}$ such that $X\ne Y$, then choosing $y\in Y$, we obtain an exit $gg''e^{-1}y$ of $T$, which is a contradiction. Hence $|C_{s(e)}| = 1$.

	We now show that the isotropy group at $T$ is trivial. Suppose, by way of contradiction that $\alpha \in F_{T}$ is a non-trivial element in the isotropy group of $T$. By Lemma \ref{lem:trivial-properties-isotropy}, we have $\alpha = \beta \gamma \beta^{-1}$, where $\gamma $ is a cycle. Moreover $\alpha ^n = \beta \gamma^n \beta^{-1}$ is also an element in $F_{T}$ for all $n\in \Z$, by Lemma \ref{lem:trivial-properties-isotropy}. Replacing $\alpha$ with $\alpha^{-1}$ if needed, we may assume that $\gamma$ contains an edge $e\in E^1$. In particular, $\beta \gamma^{n-1} \gamma _L = \beta \gamma^{n-1} \gamma_0 \le _p (\alpha^n)_L$ for all $n>0$, and thus the lengths of the paths $(\alpha^n)_L$ are unbounded. Since $(\alpha^n)_L\in T$ for all $n$, and $T$ is a finite tree, this is impossible, and thus we get a contradiction. Hence $F_T$ is trivial, and so $T\in \isoisoOmega$.
	\end{proof}

\begin{theorem}
	\label{thm:socle-tameLeav-algebra}
	Let $(E,C)$ be a separated graph. Then $\mathrm{soc} (\Le^\ab (E,C))$ coincides with the ideal generated by all projections $\e (T)$, where $T\in \isoisoOmega$. Moreover  
	$$ \mathrm{soc} (\Le^\ab (E,C)) \cong \bigoplus_{[T]\in \isoisoOmega/{\mathcal R}} M_{|[T]|} (K).$$
		For a fixed $T\in \isoisoOmega$, we can specify a set of matrix units in the component 
		$M_{|[T]|} (K)$ of $T$ by setting 
	   $$e_{\alpha,\beta}:= 1_{(\theta_{\alpha}(T), \alpha\cdot \beta^{-1} , \theta_{\beta} (T))},$$
	   where $\alpha,\beta \in \FC$ are such that $(\alpha^{-1})_L, (\beta^{-1})_L\in T$.
\end{theorem}

\begin{proof}
	The first part follows from Theorem \ref{thm:mainCGMMR} and Proposition \ref{prop:isoand-isoiso-forOmega}. The second part follows easily from the multiplication rules in the groupoid $\G_\tight (E,C)$ given in Section~\ref{sect:tight-spectrum}.
\end{proof}

With some extra hypothesis, we can obtain more precise information about the sizes of the matrices in the homogeneous components. 

\begin{proposition}
	\label{prop:finite-orbits}
	Let $(E,C)$ be a separated graph such that $|X|\ge 2$ for all $X\in C$ and $|C_v|\ge 2$ for all $v\in E^0$ such that $C_v\neq \emptyset$. Then if $T\in \isoOmega$, the size of the orbit of $T$ is exactly $|T|$, the cardinality of $T$, and the elements of the orbit of $T$ are exactly $\{ \theta_{\alpha} (T): \alpha^{-1}\in T\}$. Moreover $r(\gamma)$ is a sink of $E$ for each $\gamma \in \max (T)$. Hence, we have
		$$ \mathrm{soc} (\Le^\ab (E,C)) \cong \bigoplus_{[T]\in \isoisoOmega/{\mathcal R}} M_{|T|} (K).$$
\end{proposition}
 
\begin{proof}
	Obviously, the elements $\theta_{\alpha} (T)$, where $\alpha^{-1}\in T$, belong to the orbit of $T$. 
	By Lemma \ref{lem:isotropy-when-all-groups-havemoreon-element}, if $\gamma \in \FC$ and $\gamma_L = \gamma_0\in T$, then $\gamma \in T$. Hence the only elements in the orbit of $T$ are the ones described in the statement. It is also clear from Lemma \ref{lem:isotropy-when-all-groups-havemoreon-element} that $r(\gamma)$ is a sink of $E$ for each $\gamma\in \max (T)$.  
	\end{proof}

Hence the situation for the separated graphs satisfying the hypothesis of Proposition \ref{prop:finite-orbits} is completely different from the situation in the non-separated case, in which the orbit of the points in $\isoisoOmega$ can easily be infinite, as in the case of the Toeplitz algebra.

\section{Examples}
\label{sect:examples}

In this section, we consider some examples. We first consider two easy examples, the case of a non-separated graph and the case where $E$ is a finite, acyclic and connected as non-oriented graph, and $C$ is the free partition.
We then consider the problem of describing the ideal $\mathcal Q$ of an arbitrary tame Cohn algebra $\Co _K^\ab (E,C)$ of a separated graph, which in some cases is closely related to the socle of $\Co _K^\ab (E,C)$. This is obviously related to the socle of $\Le_K ^\ab (\ol{E},\ol{C})$, where $(\ol{E},\ol{C})$ is the separated graph associated to $(E,C)$ in Section \ref{sect:relation-with-Cohn-algebras}. After this, we will consider two specific non-trivial examples.

\begin{example}
	\label{exam:socle-non-separated-graph}
	Let $E$ be a non-separated graph. Then $\soc (\Le_K (E))$ is the ideal generated by the line points of $E$, which are the vertices $v$ of $E$ such that the tree of $v$ does not contain any bifurcation vertex or any cycle \cite[Theorem 2.6.14]{AAS}. 
	In the Leavitt model, a line point $v$ is represented by a single vertex $\{v\}$, which does not have exits and does not have isotropy. This differs from the representation in the standard model, where the line point is represented by the path consisting of the tree of $v$, which can be finite or infinite. The third picture, that is, the complete model, also gives a different representation. The points in the orbit of a line point $v$ are represented in the Leavitt model by the vertices in the tree of $v$, the line points connecting to $v$, and the finite paths $e_1\cdots e_r$ in $E$ such that $|s^{-1}(s(e_r))| \ge 2$ and $r(e_r)$ is a line point connecting to $v$. Note that these are all finite trees without exits, as predicted by Proposition \ref{prop:isoand-isoiso-forOmega}. \qed
	\end{example}

\begin{example}
	\label{exam:finite-acyclic}
	Let $(E,C)$ be a separated graph where $E$ is a finite graph which is connected and acyclic as a non-oriented graph. Suppose that $C$ is the free separation, that is, $|X|= 1$ for all $X\in C$. Then 
	$$\Le_K (E,C) = \Le_K^\ab (E,C) = \soc (\Le_K (E,C)) \cong M_n (K),$$
	where $n=|E^0|$. 
	
	Indeed, each vertex is a Leavitt $(E,C)$-tree without exits, because $|X|= 1$ for all $X\in C$. Moreover $T_L$ consists of a vertex for each $T\in \Y$. 
	Since the graph is connected as a non-oriented graph, there is only one orbit of vertices. Since the graph is acyclic as a non-oriented graph, the isotropy of each vertex in trivial. Hence the result follows from Theorem \ref{thm:socle-tameLeav-algebra}. \qed
\end{example}

Let $(E,C)$ be a separated graph. Recall from Section \ref{sect:relation-with-Cohn-algebras} that there is a natural isomorphism $\Co^\ab _K(E,C)\cong p\Le^\ab (\ol{E},\ol{C})p$. The separated graph $(\ol{E},\ol{C})$ is built in Section \ref{sect:relation-with-Cohn-algebras} by enlarging the graph $(E,C)$. Each set $X\in \Cfin$ is enlarged with a new edge $d_X$ such that $r(d_X) =v_X$ is a sink in $\ol{E}$, and $\ol{C}^{\text{fin}} = \{ \ol{X} : X\in \Cfin \}$, where $\ol{X}=X\cup \{d_X\}$. 

We are going to describe a family of projections in the socle of $\Co^\ab _K(E,C)$. We will use the above isomorphism with $p\Le^\ab (\ol{E},\ol{C})p$ and the results of Section \ref{sect:socle} for the proof of the main result of this section, Theorem \ref{thm:completelyblocked}. However, the definition of the family of projections is made directly in terms of the separated graph $(E,C)$. 

For the next definition, recall the notion of a {\it blocking family} $F$ for $T$, where $T\in \Y_0$, which has been introduced in Definition \ref{def:blocking family}.

\begin{definition}
	\label{def:completely-blocked} Let $(E,C)$ be a separated graph, and let 
	$T\in \Y_0$. We say that $T$ {\it can be completely blocked} if $\mathcal N (T)$ is finite and non-empty. This implies that for any element $\gamma_0\gamma_1 y\in \mathcal N(T)$, we have $y\in Y\in \Cfin$, and that there is a maximum blocking family $F=\sqcup _{i=1}^r \gamma_0^i(\gamma_1^iX_i)$ for $T$, which contains all the elements of $\mathcal N (T)$.
	We denote by $\ecb (T)$ the idempotent given by
	$\ecb (T) = \e (T\setminus F)$, where $F$ is the maximum blocking family for $T$. The set of elements $T\in \Y_0$ which can be completely blocked will be denoted by $\cbY$.        
	\end{definition}

\begin{theorem}
	\label{thm:completelyblocked} Let $(E,C)$ be a separated graph. 
	For each $T\in \cbY$, the idempotent $\ecb (T)$ belongs to $\mathcal Q \cap \mathrm{soc} (\Co_K ^\ab (E,C))$. Moreover,  $\mathcal Q \cap \mathrm{soc} (\Co_K ^\ab (E,C))$ is the ideal of $\Co_K ^\ab (E,C)$
	generated by all the idempotents $\ecb (T)$, for $T\in \cbY$. 
\end{theorem}

\begin{proof} First observe that $\ecb (T)= \e (T\setminus F)$ belongs to the ideal $\mathcal Q$, because $\mathcal N (T)$ is non-empty and finite.

	Let $(\ol{E},\ol{C})$ be the separated graph associated to $(E,C)$. Any $T\in \Y_0$ can be considered also as a $\ol{C}$-compatible tree in $(\ol{E},\ol{C})$. Moreover $T\in \ol{\Y}_L$, that is, $T$ belongs to the set $\Y_L(\ol{E},\ol{C})$, so the concept of an exit has been defined for $T$ with respect to $(\ol{E},\ol{C})$. By the isomorphism  $\Co^\ab _K(E,C)\cong p\Le_K^\ab (\ol{E},\ol{C})p$, the idempotent $\ecb (T)$ corresponds to $\e (T')$, for the Leavitt $(\ol{E},\ol{C})$-tree $T'$, where 
	$$T' = T\cup \big( \bigcup_{i=1}^r \{ \gamma^i_0 \gamma^i_1 d_{X_i} \}^\downarrow \Big),$$
	with $F= \sqcup _{i=1}^r  \gamma_0^i(\gamma_1^iX_i)$ the maximum blocking family for $T$. It follows that $T'$ has no exits in $(\ol{E},\ol{C})$, so that $T'$ is an isolated point in $\Omega (\ol{E},\ol{C})$ by Proposition \ref{prop:isoand-isoiso-forOmega}.

Since $|\ol{X}| \ge 2$ for all $\ol{X}\in \ol{C}$, it follows from Lemma \ref{lem:isotropy-when-all-groups-havemoreon-element} that $\Omega^{\text{iso}} (\ol{E},\ol{C}) = \Omega^{\text{iso}}_0 (\ol{E},\ol{C})$, hence $T'\in \Omega^{\text{iso}}_0 (\ol{E},\ol{C})$ and $\e (T')\in \soc (\Le_K^\ab (\ol{E},\ol{C}))$ by Theorem \ref{thm:socle-tameLeav-algebra}.

Hence we get 
	$$\e (T') \in p\Le_K^\ab (\ol{E},\ol{C})p\cap \soc (\Le_K^\ab (\ol{E},\ol{C})) = \soc (p\Le_K^\ab (\ol{E},\ol{C})p) \cong \soc (\Co_K^\ab (E,C)),$$
	and thus $\ecb (T) \in  \soc (\Co_K^\ab (E,C))\cap \mathcal Q$, as claimed. 
	       
	      Finally, suppose that $e$ is a minimal idempotent in $\mathcal Q$.

	       Let $\ol{\varphi}\colon \Co_K^\ab (E,C) \to p\Le_K^\ab ({E},\ol{C})p$ be the isomorphism from Theorem \ref{thm:relationLPACPA-TAME}, and set $\ol{\mathcal Q} = \ol{\varphi} (\mathcal Q)$. 
	        Then $\ol{\varphi} (e)$ is a minimal idempotent in the ideal 
	        $p\,  \soc (\Le_K^\ab (\ol{E},\ol{C})) p\cap \ol{\mathcal Q}$ of $p \Le_K^\ab (\ol{E},\ol{C})  p$. 
	        Now observe that 
	        $$p\, \soc (\Le_K^\ab (\ol{E},\ol{C})) p \cong \bigoplus_{[T]\in \ol{\Omega}^{\text{iso}}_0/{\mathcal R}} M_{\delta (T)}(K),$$
	        	where $\delta (T)$ is the number of elements in the orbit of $T$ which belong to $\ol{\Y}_L(v)$, for some $v\in E^0$. Indeed a system of matrix units for the component 
	        	$M_{\delta(T)}(K)$ is given by the elements $e_{\alpha,\beta}$ as in Theorem \ref{thm:socle-tameLeav-algebra}, where $\alpha,\beta \in \FC(v)$ for some $v\in E^0$. 
	       It follows that there exists some $T'\in \ol{\Omega}^{\text{iso}}_0$ such that $T'\in \ol{\Y}_L (v)$ for some $v\in E^0$ and 
	       $\ol{\varphi} (e)\sim \e (T')$ within $p\, \soc (\Le_K^\ab (\ol{E},\ol{C}))p$.
	       Now let 
	       $$T'':= T'\setminus \{ \gamma d_X : \gamma d_X\in T' \},$$
	       which is a tree in $\Y(v)$ and set $T= (T'')_0 \in \Y_0(v)$. Then $\mathcal N (T)$ is finite because $T'$ has no exits with respect to $(\ol{E},\ol{C})$, and $\ol{\varphi} (\ecb (T)) = \e (T')$. It follows that       
	       $e\sim \ecb(T)$ within $\soc (\Co_K^\ab (E,C))$. In particular, $\ecb (T)$ is a minimal idempotent in $\mathcal Q$.

	       It remains to check that $\mathcal N (T)$ is non-empty. Set  $f := \ecb (T)= \e (T\setminus F)$, where $F$ is the maximum blocking family for $T$. We have to show that $F\ne \emptyset$.      
	      	      Since $f\in \mathcal Q$, it follows from Theorem \ref{thm:BQ-basis-ofQ} that we can write \begin{equation}
	      	      	\label{eq:equality-for-f}
	      	      	f= \sum_{i=1}^r \lambda_i \e (T_i\setminus F_i),
	      	      \end{equation} 
	      	      where $\lambda_i\in K^\times$ and $F_i$ is a non-empty blocking family for $T_i\in \Y_0$, $i=1,\dots , r$. We may assume that 
	      	      $f\e (T_i\setminus F_i)\ne 0$ for all $i$. 
	      	            	     With this assumption, we then have 
	      $f = f\e(T_i\setminus F_i)$ by the minimality of $f$. Hence multiplying the relation \eqref{eq:equality-for-f} by $\prod_{i=1} ^r \e (T_i\setminus F_i)$ we get
	      $$ f= (\sum_{i=1}^r \lambda_i)  \prod_{i=1} ^r \e (T_i\setminus F_i).$$
	      Hence $f=  \prod_{i=1} ^r \e (T_i\setminus F_i) = \e \Big( (\cup_{i=1}^r T_i)\setminus (\cup_{i=1}^r F_i)\Big)$ by Lemma \ref{lem:key-for-keyforbasis-of-L}, and thus $T= \cup_{i=1}^r T_i$, $F= \cup_{i=1}^r F_i \ne \emptyset $, by Lemma \ref{lem:blocking-are-different}. Hence $F\ne \emptyset$, and the proof is complete.   	           	
	\end{proof}

\begin{example}
	\label{exam:Q-for-nonseparated}
	If $E$ is a non-separated graph, the structure of the ideal $\mathcal Q$ is transparent, see \cite[Propositions 1.5.8 and 1.5.11]{AAS}. Each regular vertex $v$ gives rise to an idempotent $q_v = v -\sum_{e\in s^{-1}(v)} ee^*\in \mathcal Q$, and we have $\mathcal Q\subseteq \mathrm{soc} (\Co_K (E))$. The idempotents $q_v$ determine the different homogeneous components of $\mathcal Q$. These facts are no longer true for a separated graph, in which case the situation is much more involved, as we will see in the following examples. Note that, in the above situation, $\{v\}$ can be completely blocked, and $\ecb (v) = q_v$.   	
\end{example}

\begin{example}
	\label{exam:Cuntz-free-separation}
Let $X$ be any non-empty set, and let $E_X$ be the graph with just one vertex $v$ and with $E_X^1 = X$. We consider the free separation $C=C_v= \{\{e\} : e\in X \}$ on $E_X$. This example was considered in \cite[Examples 3.3 and 4.9]{ABC25}. We have
$$\soc (\Co_K^\ab (E_X,C)) = \begin{cases}
	M_{\infty} (K) & \text{ if } |X| = 1\\
\,\,\, \quad 	0 & \text{ if } |X|>1.
\end{cases}
$$
The case where $|X|= 1$ is well-known, and corresponds to a non-separated graph. In that case, $\Co_K^{\ab}(E_X,C)$ is just the usual Toeplitz algebra. 

Assume that $|X| >1$. Let $T\in \Y_0$. Then $T$ can be identified with a usual Munn tree over $X$, with the property that each maximal path does not end in $X^{-1}$. If $\gamma $ is a maximal path of $T$, then we can consider an element $\gamma \nu^{-1} y$, where $\nu$ is an arbitrary word in positive edges $x\in X$, in such a way that the word $\gamma \nu^{-1} y$ is reduced. Of course this is possible because $|X|>1$. Since $\nu$ has arbitrary length, we conclude that $\mathcal N (T)$ is infinite. Therefore $\cbY= \emptyset$ and thus $\soc ((\Co_K^\ab (E_X,C)) ) =0$.  	
		\end{example}

\begin{remark}
	\label{rem:Bavula}
	Although $\soc (\Co_K^\ab (E_X,C)) = 0$ when $|X|>1$, one may find a related algebra which has a nonzero socle, by imposing some extra commuting relations. Indeed Bavula introduced in \cite{bavula} the algebra $\mathbb S_n$ as the algebra generated by $x_1,\dots ,x_n,y_1,\dots ,y_n$ with the defining relations
	$$y_1x_1= \cdots = y_nx_n=1, \quad [x_i,x_j] = [y_i,y_j] = [x_i,y_j] = 0 \, \text{ for all } i\ne j.$$
	Clearly $\mathbb S_n$ is a quotient algebra of $\Co_K^\ab (E_{X_n},C)$, where $X_n= \{ x_1,\dots ,x_n\}$. 
	By \cite[Proposition 4.1]{bavula}, the socle of $\mathbb S_n$ is a simple, essential ideal of $\mathbb S_n$.
		\end{remark}

\begin{example}
	\label{exam:free-inverse-monoid}
Let $X$ be any non-empty set. The free inverse monoid $\FIM (X)$ is realized in \cite[Examples 3.4 and 4.10]{ABC25} as a corner of the inverse semigroup of a certain separated graph $(F_X, D)$, as follows.  
Let $F_X$ be the graph with $F_X^0 = \{ v\} \sqcup X$ and $F_X^1 = \{e_x,f_x : x\in X \}$, with $s(e_x) = s(f_x) = v$ and $r(e_x) = r(f_x) = x$ for all $x\in X$.
We take the free separation $D$ on $F_X$, so that $D_v = \{\{ e_x\},\{f_x\} : x\in X\}$ and $D_x= \emptyset$ for all $x\in X$. Then by \cite[Example 4.10]{ABC25} we have a (unique) monoid isomorphism $\varphi\colon \FIM (X) \to v\IS (F_X,D) v$ such that $\varphi (x)= e_xf_x^{-1}$ for all $x\in X$. Consequently, we get a $*$-isomorphism 
$$\varphi \colon K[\FIM (X)] \to v\cdot \Co_K^\ab (F_X,D)\cdot v$$
and so we can apply the theory of separated graphs to study the monoid algebra $K[\FIM (X)]$. 
It was shown in \cite[Proposition 4.5]{AG17} that, in the case where $|X|= 1$, the socle of $K[\FIM (X)]$ is isomorphic to $\bigoplus _{i=1}^\infty M_i(K)$, and it is an essential ideal of $K[\FIM (X)]$.

We can now generalize this result for any set $X$, as follows. The bipartite separated graph $(\ol{F_X}, \ol{D})$ consists in adding sinks $v_x, v_x'$ and edges $d_x,d_x'$ to $F_X$, for each $x\in X$, such that $s(d_x)= s(d_x')= v$, $r(d_x)=v_x$, $r(d_x')= v_x'$, and 
$$\ol{D}_v = \{ \{e_x,d_x\}, \{f_x,d_x'\} : x\in X\}.$$
Note that $d_xd_x^* = v-e_xe_x^*$ corresponds to $1-xx^*$ and $d_x'(d_x')^*$ corresponds to $1-x^*x$ under the isomorphism 
$$v\cdot \Le_K^\ab (\ol{F_X},\ol{D})\cdot v \cong v\cdot  \Co_K^\ab (F_X,D)\cdot v \cong K[\FIM (X)]$$
obtained from  Corollary \ref{cor:lfbipartite} and \cite[Example 4.10]{ABC25}. 

As shown in \cite[Example 4.10]{ABC25}, each $T\in \Y_0$ corresponds uniquely with an ordinary finite tree $T$, containing $1$, in the Cayley graph of the free group $\F$ on $X$.
We will refer to such a tree simply as a finite tree, and we will denote the set of finite trees by $\mathfrak M$. The elements of $\FIM (X)$ correspond to Munn trees $(T,g)$, where $T\in \mathfrak M$ and $g\in T$, see for instance \cite[Section 6.4]{lawson}. 

Given a finite tree $T\in \mathfrak M$, we define the {\it neighborhood} $\mathcal N (T)$ of $T$ as the set of all the elements $\gamma z \in \F$, where $\gamma \in T$, $z\in X\cup X^{-1}$, and $\gamma z \notin T$. In other words the set $\mathcal N (T)$ of {\it neighbors} of $T$ is the set of words which are at distance $1$ of $T$. Hence $\mathcal N (T)$ is always non-empty, and it is finite if and only if $X$ is finite. We say that $\gamma \in T$ is an {\it interior point} of $T$ if all the reduced words $\gamma z$, $z\in X\cup X^{-1}$, belong to $T$. Of course $T$ does not have interior points when $X$ is infinite. If $ \gamma z\in \mathcal N (T)$, we say that $\gamma z$ is a neighbor of $T$ at $\gamma \in T$.

Define an equivalence relation $\sim $ on the set $\mathfrak M$ of finite trees $T$ over $X$ by declaring that $T\sim S$ if and only if $S= g\cdot T$, where $g^{-1}\in T$. 

Suppose that $X$ is finite, and take $T\in \mathfrak M$. Then the idempotent $\ecb (T)$ corresponds to the evaluation $\e (T')$, where $T'$ is the tree obtained by adding to $T$ maximal elements $\gamma d_x$ for each $\gamma x \in \mathcal N(T)$, and $\gamma d'_x$ for each $\gamma x^{-1}\in \mathcal N (T)$, for each $x\in X$. Moreover, by Proposition \ref{prop:finite-orbits}, we have
$$\soc (\Le_K ^\ab (\ol{F_X},\ol{D})) \cong \bigoplus _{[T]\in \mathfrak M /\sim} M_{|T'|} (K).$$
Hence we get
$$\soc (K[\FIM (X)]) = p\cdot \soc (\Le_K ^\ab (\ol{F_X},\ol{D}))\cdot p \cong \bigoplus _{[T]\in \mathfrak M /\sim} M_{|T|} (K).$$
 
 In conclusion, we obtain the following result for $K[\FIM (X)]\cong v\cdot \Co_K^\ab (F_X,D) \cdot v$:
 
 $$\soc (K[\FIM (X)]) \cong \begin{cases}
 	\bigoplus _{[T]\in \mathfrak M /{\sim}} M_{|T|} (K) & \text{ if } |X| <\infty \\
 	\qquad \qquad 0 & \text{ if } |X| = \infty .
 	\end{cases}$$
 It is also easy to show that $\soc (K[\FIM (X)])$ is essential in $K[\FIM (X)]$ when $X$ is finite. 
 
\end{example}

 We fix a non-empty set $X$ and we set $A= K[\FIM (X)]$. We denote by $\mathcal I$ the ideal of $A$ generated by all the elements $1-xx^*$ and $1-x^*x$, where $x\in X$. 
 The ideal $\mathcal I$ corresponds to the ideal $v\mathcal Q v$ of $v\cdot \Co_K^\ab (F_X,D) \cdot v \cong A$.  The structure of the ideal $\mathcal I$ of $K[\FIM (X)]$ was completely determined in the case where $|X|=1$ in \cite[Lemma 4.4 and Proposition 4.5]{AG17}.
 
 We can easily translate our results in Section \ref{sect:relation-with-Cohn-algebras} to obtain a linear basis of $\mathcal I$. For a finite tree $T$ on $\F$, a finite set $F\subseteq \mathcal N (T)$ and $g\in T$, we define a {\it blocked Munn tree} $(T\setminus F,g)$ as the tree $(T',g)$, with maximal elements $\gamma d_x$ if $\gamma x\in F$, $\gamma d_x'$ if $\gamma x^{-1}\in F$, and $\gamma $ if $\gamma $ is a maximal element of $T$ such that no neighbor of $\gamma $ belongs to $F$. Note that $T'=T\setminus F$ is not a usual tree on $X$, because it has some extra paths ending in sinks $v_x$ or $v_x'$. Each blocked Munn tree $(T\setminus F, g)$ has an associated element $\e (T\setminus F)\rtimes g\in A$, where
 $$\e (T\setminus F) = \prod _{\lambda \in \max (T\setminus F)} \e (\lambda) , $$
 where of course for $\lambda = \gamma d_x$, $\e (\lambda) =  \gamma (1-xx^*)\gamma^*$ and for $\lambda = \gamma d_x'$, $\e (\lambda)= \gamma (1-x^*x)\gamma^*$. Note that the idempotents $\ecb (T)$ correspond exactly to $\e (T\setminus \mathcal N (T))$. 
 
 With this notation, the following result follows directly from Theorem \ref{thm:BQ-basis-ofQ}.    
 
 \begin{proposition}
 	\label{prop:basis-for-I}
 		Let $A= K[\FIM (X)]$, where $X$ is a non-empty set. Let $\mathcal I$ be the ideal of $A$ generated by the elements $1-xx^*$ and $1-x^*x$, for all $x\in X$. Then the set $\mathcal B (\mathcal I)$ consisting of all the elements $ \e (T\setminus F)\rtimes g$, where $(T\setminus F,g)$ is a blocked Munn tree such that all the maximal elements of $T$, except possibly $g$, are blocked by $F$, is a linear basis of $\mathcal I$. Hence 
 	$\mathcal B (\mathcal I)\cup \{1\rtimes g: g\in \F\}$ is a linear basis of $A$.    
 \end{proposition}

 The blocked Munn trees $(T\setminus F,g)$ appearing in the basis $\mathcal B (\mathcal I)$ have the property that all maximal elements, with the possible exception of $g$, are of the form $\gamma d_x$ or $\gamma d_x'$ for some $x\in X$, so they are (at least partially) blocked. Observe that this basis is closed under products. The product is computed as follows:
 \begin{align*}
 	& (T_1\setminus F_1,g_1) \cdot (T_2\setminus F_2, g_2) \\
 	& = \begin{cases} (T_1\cup g_1\cdot T_2)\setminus (F_1\cup g_1\cdot F_2), g_1\cdot g_2) & \text{ if } T_1\cap g_1\cdot F_2 = g_1\cdot T_2\cap F_1= \emptyset \\
 	\qquad \quad 0 & \text{ otherwise }.
 	\end{cases}
 \end{align*} 
 
 Finally, we show that when $X$ is finite, the ideal $\mathcal I$ has a filtered structure. We say that $T\setminus F$ has $t$ exits if $|\mathcal N (T)\setminus F| = t$. 
 
 \begin{proposition}
 	\label{prop:filtered-I}
 	Let $X$ be a finite non-empty set, and set $A= K[\FIM (X)]$. Then the ideal $\mathcal I$ of $A$ has a filtered structure $\mathcal I =  \cup_{i=0}^\infty B_i$, where $B_i$ is the linear span of the set $\e(I\setminus F)\rtimes g$, where $I\setminus F$ has at most $i+1$ exits. 
 	In particular $B_0$ is a subalgebra of $A$ properly containing $\soc (A)$.
 	\end{proposition}

  \begin{proof}
  	We have to show that $B_iB_j\subseteq B_{i+j}$ for all $i,j\in \Z^+$.
  	Suppose that $T_1\setminus F_1$ has $i+1$ exits and $T_2\setminus F_2$ has $j+1$ exits, and that $T_1\cap F_2 = T_2\cap F_1= \emptyset$. We have to show that $(T_1\cup T_2)\setminus (F_1\cup F_2)$ has $\le i+j+1$ exits. If $T_1= T_2$, then 
  	$$\mathcal N (T_1\cup T_2) -|F_1\cup F_2| \le  |\mathcal N (T_1)| -|F_1| \le i+1 \le i+j+1,$$
  	because $j\ge 0$. If $T_1\nsubseteq T_2$, take $\gamma \in T_1\setminus T_2$. Then $\gamma = \gamma _1 z \gamma_2$, where $\gamma _1 \in T_2$, $ \gamma _1 z\notin T_2$ and $z\in X\cup X^{-1}$, so that $\gamma _1 z \in \mathcal N (T_2)$ and $\gamma _1 z \notin F_2$ since $T_1\cap F_2 = \emptyset$.  
  	Now observe that
  	$$\mathcal N (T_1\cup T_2)\setminus (F_1\cup F_2 ) \subseteq (\mathcal N (T_1)\setminus F_1) \cup ((\mathcal N (T_2)\setminus F_2)\setminus \{\gamma_1 z\}),$$
  	hence $|\mathcal N (T_1\cup T_2)\setminus (F_1\cup F_2)| \le (i+1) + (j+1-1) = i+j+1$, and thus $(T_1\cup T_2)\setminus (F_1\cup F_2)$ has $\le i+j+1$ exits. By symmetry, the same inequality holds if $T_2\nsubseteq T_1$. 
    \end{proof}

The subalgebra $B_0$ of $A$ has some interesting properties. We always have $\soc (B_0)= \soc (A)$, and $B_0/\soc (A)= \soc (A/\soc (A))$ holds for $|X|=1$ by \cite[Lemma 4.4]{AG17}, but it does not hold for $|X|>1$. However by the proof of the above result, given $T_1\setminus F_1, T_2\setminus F_2$ with at most one exit, either $\e (T_1\setminus F_1) \e( T_2\setminus F_2) \in \soc (A)$, or $\e(T_1\setminus F_1) \le \e(T_2\setminus F_2)$, or $\e (T_2\setminus F_2)\le \e (T_1\setminus F_1)$.


\begin{bibdiv}
	\begin{biblist}


\bib{abrams05}{article}{
    AUTHOR = {Abrams, Gene},
    AUTHOR={Aranda Pino, Gonzalo},
     TITLE = {The {L}eavitt path algebra of a graph},
   JOURNAL = {J. Algebra},
  FJOURNAL = {Journal of Algebra},
    VOLUME = {293},
      YEAR = {2005},
    NUMBER = {2},
     PAGES = {319--334},
      ISSN = {0021-8693,1090-266X},
   MRCLASS = {46L05 (16G20)},
  MRNUMBER = {2172342},
MRREVIEWER = {N\'andor\ Sieben},
       DOI = {10.1016/j.jalgebra.2005.07.028},
       URL = {https://doi.org/10.1016/j.jalgebra.2005.07.028},
}

\bib{abrams15}{article}{
    AUTHOR = {Abrams, Gene},
     TITLE = {Leavitt path algebras: the first decade},
   JOURNAL = {Bull. Math. Sci.},
  FJOURNAL = {Bulletin of Mathematical Sciences},
    VOLUME = {5},
      YEAR = {2015},
    NUMBER = {1},
     PAGES = {59--120},
      ISSN = {1664-3607,1664-3615},
   MRCLASS = {16S99},
  MRNUMBER = {3319981},
MRREVIEWER = {Enrique\ Pardo},
       DOI = {10.1007/s13373-014-0061-7},
       URL = {https://doi.org/10.1007/s13373-014-0061-7},
}

    
\bib{AAS}{book}{
   author={Abrams, Gene},
   author={Ara, Pere},
   author={Siles Molina, Mercedes},
   title={Leavitt path algebras},
   series={Lecture Notes in Mathematics},
   volume={2191},
   publisher={Springer, London},
   date={2017},
   pages={xiii+287},
   isbn={978-1-4471-7343-4},
   isbn={978-1-4471-7344-1},
   review={\MR{3729290}},
}

\bib{ADN22}{article}{
	AUTHOR = {Abrams, G.},
	AUTHOR = {Dokuchaev, M.}
    AUTHOR = {Nam, T. G.},
	TITLE = {Realizing corners of {L}eavitt path algebras as {S}teinberg
		algebras, with corresponding connections to graph
		{$C^\ast$}-algebras},
	JOURNAL = {J. Algebra},
	FJOURNAL = {Journal of Algebra},
	VOLUME = {593},
	YEAR = {2022},
	PAGES = {72--104},
	ISSN = {0021-8693,1090-266X},
	MRCLASS = {16S88 (05C25 46L05)},
	MRNUMBER = {4343723},
	DOI = {10.1016/j.jalgebra.2021.11.004},
	URL = {https://doi.org/10.1016/j.jalgebra.2021.11.004},
}

\bib{zel}{article}{
	AUTHOR = {Alahmadi, Adel},  
	AUTHOR = {Alsulami, Hamed},
	AUTHOR = {Jain, S.K.},
	AUTHOR = {Zelmanov, Efim},
	TITLE = {Leavitt path algebras of finite {G}elfand-{K}irillov
		dimension},
	JOURNAL = {J. Algebra Appl.},
	FJOURNAL = {Journal of Algebra and its Applications},
	VOLUME = {11},
	YEAR = {2012},
	NUMBER = {6},
	PAGES = {1250225, 6},
	ISSN = {0219-4988,1793-6829},
	MRCLASS = {16S99 (05C20 16P90)},
	MRNUMBER = {2997464},
	MRREVIEWER = {G\"unter\ R.\ Krause},
	DOI = {10.1142/S0219498812502258},
	URL = {https://doi.org/10.1142/S0219498812502258},
}


\bib{AraWeighted}{article}{
    AUTHOR = {Ara, Pere},
     TITLE = {Leavitt path algebras of weighted and separated graphs},
   JOURNAL = {J. Aust. Math. Soc.},
    VOLUME = {115},
      YEAR = {2023},
    NUMBER = {1},
     PAGES = {1--25},
      ISSN = {1446-7887,1446-8107},
    review={\MR{4615463}},
       DOI = {10.1017/S1446788722000155},
       URL = {https://doi.org/10.1017/S1446788722000155},
}

 

\bib{ABPS}{article}{
    AUTHOR = {Ara, Pere}, 
    AUTHOR = {Bosa, Joan}, 
    AUTHOR = {Pardo, Enrique},
    AUTHOR = {Sims, Aidan},
     TITLE = {The groupoids of adaptable separated graphs and their type
              semigroups},
   JOURNAL = {Int. Math. Res. Not. IMRN},
      YEAR = {2021},
    NUMBER = {20},
     PAGES = {15444--15496},
      ISSN = {1073-7928,1687-0247},
      review={\MR{4329873}},
       DOI = {10.1093/imrn/rnaa022},
       URL = {https://doi.org/10.1093/imrn/rnaa022},
}

\bib{ABC23}{article}{
    AUTHOR = {Ara, Pere}, 
    AUTHOR = {Buss, Alcides},
    AUTHOR = {Dalla Costa, Ado},
     TITLE = {Free actions of groups on separated graph {$C^*$}-algebras},
   JOURNAL = {Trans. Amer. Math. Soc.},
    VOLUME = {376},
      YEAR = {2023},
    NUMBER = {4},
     PAGES = {2875--2919},
      ISSN = {0002-9947},
  review={\MR{4557884}},
       DOI = {10.1090/tran/8839},
       URL = {https://doi.org/10.1090/tran/8839},
}


\bib{ABC25}{article}{
	AUTHOR = {Ara, Pere},
	AUTHOR = {Buss, Alcides},
	AUTHOR = {Dalla Costa, Ado},
	TITLE = {Inverse semigroups of separated graphs and associated
		algebras},
	JOURNAL = {Bull. Braz. Math. Soc. (N.S.)},
	FJOURNAL = {Bulletin of the Brazilian Mathematical Society. New Series.
		Boletim da Sociedade Brasileira de Matem\'atica},
	VOLUME = {56},
	YEAR = {2025},
	NUMBER = {3},
	PAGES = {Paper No. 38, 55},
	ISSN = {1678-7544,1678-7714},
	MRCLASS = {46L55 (20M18)},
	MRNUMBER = {4922835},
	DOI = {10.1007/s00574-025-00462-7},
	URL = {https://doi.org/10.1007/s00574-025-00462-7},
}


\bib{AC24}{article}{
    AUTHOR = {Ara, Pere},
    AUTHOR = {Claramunt, Joan},
     TITLE = {A correspondence between surjective local homeomorphisms and a
              family of separated graphs},
   JOURNAL = {Discrete Contin. Dyn. Syst.},
  FJOURNAL = {Discrete and Continuous Dynamical Systems},
    VOLUME = {44},
      YEAR = {2024},
    NUMBER = {5},
     PAGES = {1178--1266},
      ISSN = {1078-0947,1553-5231},
   MRCLASS = {37B10 (16S88 37A55 46L55)},
  MRNUMBER = {4714538},
       DOI = {10.3934/dcds.2023143},
       URL = {https://doi.org/10.3934/dcds.2023143},
}


\bib{Ara-Exel:Dynamical_systems}{article}{
  author={Ara, Pere},
  author={Exel, Ruy},
  title={Dynamical systems associated to separated graphs, graph algebras, and paradoxical decompositions},
  journal={Adv. Math.},
  volume={252},
  date={2014},
  pages={748--804},
  issn={0001-8708},
  review={\MR {3144248}},
  doi={10.1016/j.aim.2013.11.009},
}

\bib{AG12}{article}{
    AUTHOR = {Ara, Pere}, 
    AUTHOR = {Goodearl, Kenneth R.},
     TITLE = {Leavitt path algebras of separated graphs},
   JOURNAL = {J. Reine Angew. Math.},
    VOLUME = {669},
      YEAR = {2012},
     PAGES = {165--224},
      ISSN = {0075-4102},
      review={\MR{2980456}},
       DOI = {10.1515/crelle.2011.146},
       URL = {https://doi.org/10.1515/crelle.2011.146},
}

\bib{AG17}{article}{
	AUTHOR = {Ara, Pere}
	AUTHOR = {Goodearl, Kenneth R.},
	TITLE = {The realization problem for some wild monoids and the {A}tiyah
		problem},
	JOURNAL = {Trans. Amer. Math. Soc.},
	FJOURNAL = {Transactions of the American Mathematical Society},
	VOLUME = {369},
	YEAR = {2017},
	NUMBER = {8},
	PAGES = {5665--5710},
	ISSN = {0002-9947,1088-6850},
	MRCLASS = {16D40 (16D70 16E50 19A13 20M25)},
	MRNUMBER = {3646775},
	MRREVIEWER = {Markus\ Schmidmeier},
	DOI = {10.1090/tran/6889},
	URL = {https://doi.org/10.1090/tran/6889},
}


\bib{AraLolk}{article}{
    AUTHOR = {Ara, Pere},
    author = {Lolk, Matias},
     TITLE = {Convex subshifts, separated {B}ratteli diagrams, and ideal
              structure of tame separated graph algebras},
   JOURNAL = {Adv. Math.},
    VOLUME = {328},
      YEAR = {2018},
     PAGES = {367--435},
      ISSN = {0001-8708,1090-2082},
      DOI = {10.1016/j.aim.2018.01.020},
       URL = {https://doi.org/10.1016/j.aim.2018.01.020},
}

\bib{AMP}{article}{
    AUTHOR = {Ara, Pere},
    AUTHOR = {Moreno, M. \'Angeles},
    AUTHOR  = {Pardo, Enrique},
     TITLE = {Nonstable {$K$}-theory for graph algebras},
   JOURNAL = {Algebr. Represent. Theory},
  FJOURNAL = {Algebras and Representation Theory},
    VOLUME = {10},
      YEAR = {2007},
    NUMBER = {2},
     PAGES = {157--178},
      ISSN = {1386-923X,1572-9079},
   MRCLASS = {46L80 (46L05)},
  MRNUMBER = {2310414},
MRREVIEWER = {Mark\ Tomforde},
       DOI = {10.1007/s10468-006-9044-z},
       URL = {https://doi.org/10.1007/s10468-006-9044-z},
}

\bib{AMBMGS2010}{article}{
    AUTHOR = {Aranda Pino, Gonzalo},
    AUTHOR = {Mart\'in Barquero, Dolores},
    AUTHOR = {Mart\'in Gonz\'alez, C\'andido},
    AUTHOR = {Siles Molina, Mercedes},
     TITLE = {Socle theory for {L}eavitt path algebras of arbitrary graphs},
   JOURNAL = {Rev. Mat. Iberoam.},
  FJOURNAL = {Revista Matem\'atica Iberoamericana},
    VOLUME = {26},
      YEAR = {2010},
    NUMBER = {2},
     PAGES = {611--638},
      ISSN = {0213-2230,2235-0616},
   MRCLASS = {16S99 (05C25)},
  MRNUMBER = {2677009},
MRREVIEWER = {Pere\ Ara},
       DOI = {10.4171/RMI/611},
       URL = {https://doi.org/10.4171/RMI/611},
}




\bib{AshHall}{article}{
  author={Ash, C. J.}, 
  author={Hall, T. E.},
  title={Inverse semigroups on graphs},
  journal={Semigroup Forum},
  volume={11},
  year={1975/76},
  number={2},
  pages={140--145},
  issn={0037-1912},
  review={\MR{387449}},
  doi={10.1007/BF02195262},
  url={https://doi.org/10.1007/BF02195262},
}



\bib{Bergman-Diamond}{article}{
    AUTHOR = {Bergman, George M.},
     TITLE = {The diamond lemma for ring theory},
   JOURNAL = {Adv. in Math.},
    VOLUME = {29},
      YEAR = {1978},
    NUMBER = {2},
     PAGES = {178--218},
      ISSN = {0001-8708},
      review={\MR{506890}},
       DOI = {10.1016/0001-8708(78)90010-5},
       URL = {https://doi.org/10.1016/0001-8708(78)90010-5},
}


\bib{bavula}{article}{
	AUTHOR = {Bavula, V. V.},
	TITLE = {The algebra of one-sided inverses of a polynomial algebra},
	JOURNAL = {J. Pure Appl. Algebra},
	FJOURNAL = {Journal of Pure and Applied Algebra},
	VOLUME = {214},
	YEAR = {2010},
	NUMBER = {10},
	PAGES = {1874--1897},
	ISSN = {0022-4049,1873-1376},
	MRCLASS = {16S99},
	MRNUMBER = {2608115},
	MRREVIEWER = {J.\ Kuzmanovich},
	DOI = {10.1016/j.jpaa.2009.12.033},
	URL = {https://doi.org/10.1016/j.jpaa.2009.12.033},
}


\bib{BaH2015}{article}{
    AUTHOR = {Brown, Jonathan H.},
    AUTHOR = {an Huef, Astrid},
     TITLE = {The socle and semisimplicity of a {K}umjian-{P}ask algebra},
   JOURNAL = {Comm. Algebra},
  FJOURNAL = {Communications in Algebra},
    VOLUME = {43},
      YEAR = {2015},
    NUMBER = {7},
     PAGES = {2703--2723},
      ISSN = {0092-7872,1532-4125},
   MRCLASS = {16S99 (16D60 16D70 16W50)},
  MRNUMBER = {3354056},
MRREVIEWER = {K.\ R.\ Goodearl},
       DOI = {10.1080/00927872.2014.888560},
       URL = {https://doi.org/10.1080/00927872.2014.888560},
}


\bib{CFST2014}{article}{
    AUTHOR = {Clark, Lisa Orloff},
    AUTHOR = {Farthing, Cynthia},
    AUTHOR = {Sims, Aidan},
    AUTHOR = {Tomforde, Mark},
     TITLE = {A groupoid generalisation of {L}eavitt path algebras},
   JOURNAL = {Semigroup Forum},
  FJOURNAL = {Semigroup Forum},
    VOLUME = {89},
      YEAR = {2014},
    NUMBER = {3},
     PAGES = {501--517},
      ISSN = {0037-1912,1432-2137},
   MRCLASS = {16S99 (20L05 22A22)},
  MRNUMBER = {3274831},
MRREVIEWER = {Roozbeh\ Hazrat},
       DOI = {10.1007/s00233-014-9594-z},
       URL = {https://doi.org/10.1007/s00233-014-9594-z},
}

\bib{CCMMR25}{article}{
    author={Lisa Orloff Clark}
    author={Cristóbal Gil Canto}
    author ={Dolores Martín Barquero}
    author = {Cándido Martín González}
    author = {Iván Ruiz Campos}
    TITLE = {On the socle of a class of {S}teinberg algebras},
   JOURNAL = {Linear Algebra Appl.},
  FJOURNAL = {Linear Algebra and its Applications},
    VOLUME = {728},
      YEAR = {2026},
     PAGES = {449--464},
      ISSN = {0024-3795,1873-1856},
   MRCLASS = {16S99 (16D25 16D70 16S88 22A22 46L55)},
  MRNUMBER = {4964174},
       DOI = {10.1016/j.laa.2025.09.016},
       URL = {https://doi.org/10.1016/j.laa.2025.09.016},
}



\bib{CorHaz24}{article}{
    AUTHOR = {Corti\~nas, Guillermo}
    AUTHOR = {Hazrat, Roozbeh},
     TITLE = {Classification conjectures for {L}eavitt path algebras},
   JOURNAL = {Bull. Lond. Math. Soc.},
  FJOURNAL = {Bulletin of the London Mathematical Society},
    VOLUME = {56},
      YEAR = {2024},
    NUMBER = {10},
     PAGES = {3011--3060},
      ISSN = {0024-6093,1469-2120},
   MRCLASS = {16S88 (19K14 19K35 46L80)},
  MRNUMBER = {4808570},
MRREVIEWER = {Jonathan\ H.\ Brown},
       DOI = {10.1112/blms.13139},
       URL = {https://doi.org/10.1112/blms.13139},
}

\bib{CorMon21}{article}{
    AUTHOR = {Corti\~nas, Guillermo},
    AUTHOR = {Montero, Diego},
     TITLE = {Algebraic bivariant {$K$}-theory and {L}eavitt path algebras},
   JOURNAL = {J. Noncommut. Geom.},
  FJOURNAL = {Journal of Noncommutative Geometry},
    VOLUME = {15},
      YEAR = {2021},
    NUMBER = {1},
     PAGES = {113--146},
      ISSN = {1661-6952,1661-6960},
   MRCLASS = {19K35 (16D70 16S88 19D50)},
  MRNUMBER = {4248209},
       DOI = {10.4171/jncg/397},
       URL = {https://doi.org/10.4171/jncg/397},
}


\bib{Exel:Inverse_combinatorial}{article}{
  author={Exel, Ruy},
  title={Inverse semigroups and combinatorial $C^*$\nobreakdash-algebras},
  journal={Bull. Braz. Math. Soc. (N.S.)},
  volume={39},
  date={2008},
  number={2},
  pages={191--313},
  issn={1678-7544},
  review={\MR{2419901}},
  doi={10.1007/s00574-008-0080-7},
}


\bib{exel2025}{article}{
      title={Consonant inverse semigroups}, 
      author={Ruy Exel},
      year={2025},
      eprint={arXiv:2508.17552 [math:OA]},
      url={https://arxiv.org/abs/2508.17552}, 
}


\bib{GR2025}{article}{
    AUTHOR = {Gon\c calves, Daniel},
    AUTHOR = {Royer, Danilo},
     TITLE = {Irreducible representations of one-sided subshift algebras},
   JOURNAL = {Results Math.},
  FJOURNAL = {Results in Mathematics},
    VOLUME = {80},
      YEAR = {2025},
    NUMBER = {6},
     PAGES = {Paper No. 186, 16},
      ISSN = {1422-6383,1420-9012},
   MRCLASS = {16S10 (16G20 16G30 16S88)},
  MRNUMBER = {4950605},
       DOI = {10.1007/s00025-025-02504-4},
       URL = {https://doi.org/10.1007/s00025-025-02504-4},
}


\bib{jac}{book}{
	AUTHOR = {Jacobson, Nathan},
	TITLE = {Structure of rings},
	SERIES = {American Mathematical Society Colloquium Publications},
	VOLUME = {Vol. 37},
	EDITION = {Revised},
	PUBLISHER = {American Mathematical Society, Providence, RI},
	YEAR = {1964},
	PAGES = {ix+299},
	MRCLASS = {16.00 (13.00)},
	MRNUMBER = {222106},
	MRREVIEWER = {Carl\ Faith},
}



\bib{JonesLawson}{article}{
  author={Jones, David G.},
  author={Lawson, Mark V.},
  title={Graph inverse semigroups: their characterization and completion},
  journal={J. Algebra},
  volume={409},
  year={2014},
  pages={444--473},
  issn={0021-8693},
  review={\MR {3198850}},
  doi={10.1016/j.jalgebra.2014.04.001},
  url={https://doi.org/10.1016/j.jalgebra.2014.04.001},
}



\bib{KocOzaydin2020}{article}{
    AUTHOR = {Ko\c c, Ayten}, 
    AUTHOR = {\"Ozayd\i n, Murad},
     TITLE = {Representations of {L}eavitt path algebras},
   JOURNAL = {J. Pure Appl. Algebra},
  FJOURNAL = {Journal of Pure and Applied Algebra},
    VOLUME = {224},
      YEAR = {2020},
    NUMBER = {3},
     PAGES = {1297--1319},
      ISSN = {0022-4049,1873-1376},
   MRCLASS = {16G20 (16D90 16G60 16S88)},
  MRNUMBER = {4009579},
MRREVIEWER = {Alireza\ Nasr-Isfahani},
       DOI = {10.1016/j.jpaa.2019.07.018},
       URL = {https://doi.org/10.1016/j.jpaa.2019.07.018},
}


\bib{lawson}{book}{
  author={Lawson, Mark V.},
  title={Inverse semigroups},
  note={The theory of partial symmetries},
  publisher={World Scientific Publishing Co., Inc., River Edge, NJ},
  year={1998},
  pages={xiv+411},
  isbn={981-02-3316-7},
  review={\MR {1694900}},
  doi={10.1142/9789812816689},
  url={https://doi.org/10.1142/9789812816689},
}



\bib{Lolk:tame}{article}{
  author={Lolk, Matias},
  title={Exchange rings and real rank zero C*-algebras associated with finitely separated graphs},  
  journal={Preprint arXiv:1705.04494},
  year={2017},
}


\bib{LuoWhangWei23}{article}{
    AUTHOR = {Luo, Yongle},
    AUTHOR = {Wang, Zhengpan},
    AUTHOR = {Wei, Jiaqun},
     TITLE = {Distributivity in congruence lattices of graph inverse
              semigroups},
   JOURNAL = {Comm. Algebra},
  FJOURNAL = {Communications in Algebra},
    VOLUME = {51},
      YEAR = {2023},
    NUMBER = {12},
     PAGES = {5046--5053},
      ISSN = {0092-7872,1532-4125},
   MRCLASS = {20M18 (05C20 06D05)},
  MRNUMBER = {4652643},
MRREVIEWER = {Desmond\ G.\ FitzGerald},
       DOI = {10.1080/00927872.2023.2224450},
       URL = {https://doi.org/10.1080/00927872.2023.2224450},
}

\bib{marg-meakin-93}{article}{
	AUTHOR={Margolis, Stuart W.},
	AUTHOR={Meakin, John},
	TITLE = {Free inverse monoids and graph immersions},
	JOURNAL = {Internat. J. Algebra Comput.},
	VOLUME = {3},
	YEAR = {1993},
	NUMBER = {1},
	PAGES = {79--99},
	ISSN = {0218-1967,1793-6500},
       review={\MR {1214007}},
	DOI = {10.1142/S021819679300007X},
	URL = {https://doi.org/10.1142/S021819679300007X},
}

\bib{meakin-milan-wang-2021}{article}{
    AUTHOR = {Meakin, John},
    author = {Milan, David},
    author = {Wang, Zhengpan},
     TITLE = {On a class of inverse semigroups related to {L}eavitt path
              algebras},
   JOURNAL = {Adv. Math.},
    VOLUME = {384},
      YEAR = {2021},
     PAGES = {Paper No. 107729, 37},
      ISSN = {0001-8708,1090-2082},
 review = {\MR{4242903}},
       DOI = {10.1016/j.aim.2021.107729},
       URL = {https://doi.org/10.1016/j.aim.2021.107729},
}


\bib{meakin-wang-2021}{article}{
  author={Meakin, John},
  author = {Wang, Zhengpan},
  title={On graph inverse semigroups},
  journal={Semigroup Forum},
  volume={102},
  year={2021},
  number={1},
  pages={217--234},
  issn={0037-1912},
 review={\MR {4214502}},
  doi={10.1007/s00233-020-10130-5},
  url={https://doi.org/10.1007/s00233-020-10130-5},
}



\bib{mesyan-mitchell-2016}{article}{
   author={Mesyan, Zachary},
   author={Mitchell, J. D.},
   title={The structure of a graph inverse semigroup},
   journal={Semigroup Forum},
   volume={93},
   date={2016},
   number={1},
   pages={111--130},
   issn={0037-1912},
   review={\MR{3528431}},
   doi={10.1007/s00233-016-9793-x},
}


\bib{Paterson:Groupoids}{book}{
  author={Paterson, Alan L. T.},
  title={Groupoids, inverse semigroups, and their operator algebras},
  series={Progress in Mathematics},
  volume={170},
  publisher={Birkh\"auser Boston Inc.},
  place={Boston, MA},
  date={1999},
  pages={xvi+274},
  isbn={0-8176-4051-7},
  review={\MR{1724106}},
  doi={10.1007/978-1-4612-1774-9},
}

\bib{Ruiz25}{article}{
    AUTHOR = {Ruiz, Efren},
     TITLE = {The algebraic {K}irchberg-{P}hillips question for {L}eavitt
              path algebras},
   JOURNAL = {Bull. Lond. Math. Soc.},
  FJOURNAL = {Bulletin of the London Mathematical Society},
    VOLUME = {57},
      YEAR = {2025},
    NUMBER = {4},
     PAGES = {1229--1248},
      ISSN = {0024-6093,1469-2120},
   MRCLASS = {16S88 (37B10 46L35)},
  MRNUMBER = {4894335},
       DOI = {10.1112/blms.70027},
       URL = {https://doi.org/10.1112/blms.70027},
}


\bib{Steinberg2010}{article}{
    AUTHOR = {Steinberg, Benjamin},
     TITLE = {A groupoid approach to discrete inverse semigroup algebras},
   JOURNAL = {Adv. Math.},
    VOLUME = {223},
      YEAR = {2010},
    NUMBER = {2},
     PAGES = {689--727},
      ISSN = {0001-8708,1090-2082},
       DOI = {10.1016/j.aim.2009.09.001},
       URL = {https://doi.org/10.1016/j.aim.2009.09.001},
}

\bib{Steinberg2019}{article}{
	AUTHOR = {Steinberg, Benjamin},
	TITLE = {Prime \'etale groupoid algebras with applications to inverse
		semigroup and {L}eavitt path algebras},
	JOURNAL = {J. Pure Appl. Algebra},
	FJOURNAL = {Journal of Pure and Applied Algebra},
	VOLUME = {223},
	YEAR = {2019},
	NUMBER = {6},
	PAGES = {2474--2488},
	ISSN = {0022-4049,1873-1376},
	MRCLASS = {20M18 (16S36 16S99 18F20 20M25 22A22)},
	MRNUMBER = {3906559},
	MRREVIEWER = {Leonid\ M.\ Martynov},
	DOI = {10.1016/j.jpaa.2018.09.003},
	URL = {https://doi.org/10.1016/j.jpaa.2018.09.003},
}



\end{biblist}
\end{bibdiv}
\end{document}